\documentclass[11pt]{amsart}

\usepackage{amsmath, amssymb, amsthm}
\usepackage{comment}
\usepackage[top=1in, bottom=1in, left=1in, right=1in, marginparwidth=1in, marginparsep=0.1in]{geometry}
\usepackage[unicode,breaklinks=true,colorlinks=true]{hyperref}
\usepackage{graphicx}
\usepackage{pgfplots}
\usepackage{subcaption}
\usepackage[dvipsnames]{xcolor}

\usepackage[foot]{amsaddr}

\usepackage{etoolbox}

\pgfplotsset{compat=1.18}
\usepgfplotslibrary{fillbetween}

\definecolor{LandauBlue}{RGB}{40, 120, 190}
\definecolor{ParamPink}{RGB}{250, 230, 230}
\definecolor{BoundaryRed}{RGB}{215, 60, 50}
\definecolor{CurveGreen}{RGB}{50, 130, 70}
\definecolor{CurveOrange}{RGB}{245, 110, 25}
\definecolor{CurvePurple}{RGB}{100, 70, 170}

\pgfplotsset{
    Landau/.style = {
        very thick, 
        smooth, 
        samples=260, 
        no markers},
    Boundary/.style = {
        BoundaryRed, 
        dashed, 
        thick, 
        samples=200}
}

\tikzset{
    LabelStyle/.style={
        font = \small,
        fill = white,
        fill opacity = .82,
        text opacity = 1,
        inner sep = 1.5pt,
        rounded corners=2pt
    }
}

\numberwithin{equation}{section}
\newtheorem{theorem}{Theorem}[section]
\newtheorem{proposition}[theorem]{Proposition}
\newtheorem{lemma}[theorem]{Lemma}
\newtheorem{definition}[theorem]{Definition}

\newtheorem{assumption}[theorem]{Assumption}

\theoremstyle{remark}
\newtheorem{remark}[theorem]{Remark}

\let\OLDthebibliography\thebibliography
\renewcommand\thebibliography[1]{
    \OLDthebibliography{#1}
    \setlength{\parskip}{1pt}
    \setlength{\itemsep}{1pt plus 0.3ex}
}

\def\R{\mathbb R}

\def\cF{\mathcal F}
\def\cL{\mathcal L}
\def\cO{\mathcal O}

\newbool{showDetails}
\setbool{showDetails}{false}
\newcommand{\detail}[1]{\ifbool{showDetails}{\onote{#1}}{}}
\ifbool{showDetails}{
    \specialcomment{details}{\begingroup\color{Orange}}{\endgroup}
}{
    \excludecomment{details}
}

\title[Singular stationary Navier-Stokes flows]{Singular stationary Navier-Stokes flows: examples and stability}

\author{Zachary Bradshaw}
\address{Zachary Bradshaw, Department of Mathematical Sciences, University of Arkansas, Fayetteville, AR. Email: \url{zb002@uark.edu}.}
\author{Dakota Palmer}
\address{Dakota Palmer, Department of Mathematical Sciences, University of Arkansas, Fayetteville, AR. Email: \url{djp007@uark.edu}.}

\date{\today}

\begin{document}

\begin{abstract}
    Landau solutions, when oriented along the vertical axis, represent a one parameter family of exact, self-similar, axisymmetric, swirl-free solutions to the 3D stationary Navier-Stokes equations forced by an upward facing point-source of momentum at the origin. 
    They have an isolated singularity at the origin. 
    Other singular steady state solutions can be derived from a similar framework. 
    Some of these variants have been proposed as models for physical scenarios. 
    For example, a solution first found by Squire has been proposed as a model of a fluid entrained to a radially discharging surface layer of oil.
    Another, Serrin's swirling vortex, exhibits qualitative features shared with some tornadoes, like a two-cell structure consisting of a central downdraft and peripheral updraft as well as swirl.
    The first objective of this paper is to provide a detailed analysis of these and other examples, especially when boundaries are present.
    In this direction we find several new, physically motivated classes of solutions and identify new connections between the physics literature and the mathematics literature.
    Many of these examples are formulated on the half-space, but much of the mathematical literature on singular steady-state solutions is for $\R^3$.
    The second objective of this paper is to establish asymptotic stability for many of these solutions by formulating the problem in domains with boundaries.
    Our most general result in this direction uses a new approach to asymptotic stability that is based on eventual regularity.
    In a different direction, we prove that a class of singular steady-state solutions motivated by Serrin's swirling vortex are stable under certain axisymmetric perturbations.
    This requires a novel observation because these solutions are too singular---they are called ``Type III'' in the literature---to be directly amenable to existing approaches.
\end{abstract}

\maketitle

\tableofcontents

\section{Introduction}
\label{sec::Introduction}

The stationary (incompressible) Navier-Stokes equations in a domain $\Omega\subseteq\R^3$ are 
\begin{equation}
    \tag{SNS}
    \label{sns}
    \left\{\begin{aligned}
        -\Delta V+{\nabla\cdot (V\otimes V)}+\nabla q&=f & &\mathrm{in}\;\Omega \\
        \nabla\cdot V & = 0                   & &\mathrm{in}\;\Omega \\
        V|_{\partial\Omega}(x)&=V_*(x)        & &\mathrm{on}\;\partial\Omega
    \end{aligned}\right.,
\end{equation}
where $V:\Omega\to\R^3$ is the velocity of the fluid, 
$q:\Omega\to\R$ is the corresponding pressure, 
$f:\Omega\to\R^3$ is the external force, 
and $V_*:\partial\Omega\to\R^3$ is the given boundary value.\par 

The stationary Navier-Stokes equations have a non-trivial scaling symmetry. 
Suppose the pair $(V,q)$ solves \eqref{sns} with force $f$ in $\Omega=\R^3$. 
For all $\lambda>0$ the pair $(V_\lambda,q_\lambda)$ defined by 
\begin{align*}
    V_\lambda(x) = \lambda V(\lambda x) 
    \qquad 
    \text{and} 
    \qquad 
    q_\lambda(x) = \lambda^2 q(\lambda x),
\end{align*}
also solves \eqref{sns} with force $f_\lambda=\lambda^3 f(\lambda x)$. 
It is natural to study solutions which are invariant under this scaling symmetry. 
A pair $(V,q)$ is called \textit{self-similar} when $(V_\lambda,q_\lambda)=(V,q)$ for all $\lambda>0$.
In this case, $V$ is $(-1)$-homogeneous and $q$ is $(-2)$-homogeneous. 
In spherical coordinates $(r,\theta,\phi)$, $V$ is $(-1)$-homogeneous if and only if $\partial_r V=-\frac{1}{r}V$.
We call a solution \textit{axisymmetric} if $\partial_\phi V=0$ and \textit{no-swirl} if $V_\phi=0$ where 
\begin{align*}
    V=V_r e_r + V_\theta e_\theta + V_\phi e_\phi.
\end{align*}
Many $(-1)$-homogeneous solutions naturally live in the space $L^{3,\infty}(\Omega)$.

When considering a stationary solution, it is natural to ask if it is asymptotically stable. 
In other words, are there conditions under which the evolutionary system drives perturbations of a steady-state toward the steady-state? 
There are a number of special solutions to the Navier-Stokes equations which are known to be asymptotically stable, e.g., the Burgers vortex \cite{GaMa} or Landau solutions under a smallness condition \cite{KarchPil}.
In this paper we will analyze a number of examples of singular stationary solutions which are related to the Landau solutions and will then explore the asymptotic stability of classes of solutions including these examples.
The following is a list of the examples which are of interest to us and which we explore further in Section \ref{sec::Examples}:

\begin{itemize}
    \item \textbf{Solutions driven by point sources of momentum.} 
    Landau's original solutions are driven by a point-source of momentum at the origin and were featured in the first paper on $L^2$-asymptotic stability by Karch and Pilarczyk \cite{KarchPil}.
    In Section \ref{subsec::Examples::LandauSolutions}, we generalize this class of solutions under a smallness condition by constructing solutions which can have   multiple isolated singularities within a large class of domains including, e.g., $\R^3$ and $\R^3_+$.
    
    \item \textbf{The solutions of Li, Li and Yan.}
    In \cite{LiLiYanI, LiLiYanII, LiLiYanIII}, Li, Li and Yan study axisymmetric self-similar solutions to the Navier-Stokes equations which are defined away from either a half-axis or a full axis, along which they are singular---see also \cite{LiLiYanIV, LiLiYanV,PaullPillowI,PaullPillowII,PaullPillowIII}.
    In addition to describing some of their results, in Section \ref{subsec::Examples::LiLiYanTypeII} we observe that their class of Type II-singular solutions belongs to $L^{3,\infty}(\R^3)$---see Lemma \ref{lem::TypeIISolutionsAreInWeakL3}. 
    This observation  allows us to generalize this class of solutions by constructing solutions which are driven by a finite number of line sources under a smallness condition---see Section \ref{subsec::Examples::SolutionsWithMultipleLineSources}. 
    We will prove this generalization of Type II-solutions are asymptotically stable.
    
    \item \textbf{Squire's surface discharge flows.} 
    In \cite{Squire2, Wang}, solutions are derived in $\R^3_+$ which are self-similar, axisymmetric, no-swirl and which have a no-penetration boundary condition which consists of a radial, self-similar vector field on the boundary. 
    These solutions are motivated by a physical scenario in which a fluid in a volume is entrained to a spreading layer of, e.g., oil on its surface.
    We make two new observations about these solutions.
    First, we show that they exist within the Li, Li and Yan framework and correspond to a flow in $\R^3$ which has a singularity on the half-axis of symmetry outside of the half-space where the solution converges (the singularity is Type III, which is more severe than a Type II-singularity).
    Second, by appealing to results of Li, Li and Yan, we show that these solutions can be perturbed to \textit{swirling}, self-similar, axisymmetric solutions. 
    We hypothesize that these might model flows entrained to a spreading layer of oil that is being discharged from a \textit{rotating} pipe and, as such, constitute a newly identified class of physically motivated solutions within the framework of Li, Li and Yan.
    See Sections \ref{subsec::Examples::SquireSurfaceDischargeFlows} and \ref{subsec::Examples::SquireSurfaceDischargeFlowsWithSwirl}. 
    We additionally identify examples within the Li, Li and Yan framework that have a mild (Type II) singularity along the vertical axis and satisfy the same no-penetration boundary conditions as Squire's surface discharge flows---see Section \ref{subsec::Examples::LiLiYanTypeII}. 
    We will prove that all of these solutions are asymptotically stable under a smallness condition.

    \item \textbf{Serrin's swirling vortex and Type II.5 solutions.}
    There are no nontrivial self-similar, axisymmetric solutions on $\R^3_+$ which converge within $\R^3_+$ and are no-slip on the boundary \cite{KMT}. 
    This is why Squire's solutions have non-zero boundary conditions.
    If a no-slip boundary is enforced, then such a solution becomes singular on an interior ray.
    In \cite{Serrin}, Serrin constructed axisymmetric, self-similar solutions via a different method.
    This built upon earlier work by Goldshtik \cite{Goldshtik}.
    Notably, these solutions can possess swirl and multiple cells and were consequently proposed by Serrin as a tornado model.
    Serrin's solutions are not in $L^{3,\infty}(\R^3_+)$---for this reason they are referred to as Type III solutions in \cite{LiYan}---and are therefore excluded from previous asymptotic stability results.
    Interestingly, Serrin's vortex is only Type III in its horizontal components with the vertical component being less singular.
    The same is true of some solutions on $\R^3$ \cite{LiLiYanI}.
    Based on these examples, we will introduce an intermediate class of \textit{Type II.5 solutions}---see Definition \ref{def::TypeII.5} and Section \ref{subsec::Examples::TypeII.5Solutions}---and prove that some Type II.5 solutions are $L^2(\R^3)$-asymptotically stable under  certain \textit{axisymmetric} perturbations---see Section \ref{sec::TypeII.5Stability}. The Type II.5 solutions in our theorem possess a source of momentum in the horizontal direction at the singular line. A sink appears to have a destabilizing effect. 
    We emphasize that this is the first result proving asymptotic stability for profiles that have a singularity on a line which is strictly $O(|\rho|^{-1})$ where $\rho$ is the distance to the singular line.
\end{itemize}

\bigskip 

\noindent\textbf{The perturbed Navier-Stokes equations.} 
Asymptotic stability asserts that a perturbation of a steady-state approaches the steady-state as $t\to\infty$.
We presently formalize this concept.
Let $V$ be a solution of \eqref{sns} in $\Omega$ with boundary condition $V_*$ and external force $f$.
Suppose $w$ is a solution to the (evolutionary) Navier-Stokes equations in $\Omega$ with boundary data $V_*$, external force $f$ and initial data $w_0$.
Specifically, $w$ solves 
\begin{equation}
    \tag{NS}
    \label{ns}
    \left\{\begin{aligned}
        \partial_t w-\Delta w+{\nabla \cdot (w\otimes w)}+\nabla p&=f & &\mathrm{in}\;\Omega\times(0,\infty)\\
        \nabla\cdot w&=0                                  & &\mathrm{in}\;\Omega\times(0,\infty)\\
        w|_{t=0}&=w_0                                     & &\mathrm{in}\;\Omega\times\{0\}\\
        w|_{\partial\Omega}(x)&=V_*(x)                    & &\mathrm{on}\;\partial\Omega\times(0,\infty).
    \end{aligned}\right.
\end{equation}
We say the solution $V$ is asymptotically stable in a Banach space $X$ if there exists $r>0$ such that for all $w_0-V\in X$ with 
\begin{align*}
    \|w_0-V\|_X<r,
\end{align*}
we have that a corresponding solution $w$ of \eqref{ns} satisfies 
\begin{align*}
    \lim_{t\to\infty}\|w(t)-V\|_X=0.
\end{align*}
We say $V$ is globally asymptotically stable if $r=\infty$.

Let 
\begin{align*}
    u(t)=w(t)-V.
\end{align*}
Then $u(t)$ solves the perturbed Navier-Stokes equations
\begin{equation}
    \tag{PNS}
    \label{pns}
    \left\{\begin{aligned}
        \partial_t u-\Delta u+{\nabla \cdot (u\otimes (u+V)) +\nabla\cdot ( V\otimes u )}+\nabla \pi&=0 & &\mathrm{in}\;\Omega\times(0,\infty)\\
        \nabla\cdot u&=0                                                                    & &\mathrm{in}\;\Omega\times(0,\infty)\\
        u|_{t=0}&=u_0                                                                       & &\mathrm{in}\;\Omega\times\{0\}\\
        u|_{\partial\Omega}(x)&=0                                                           & &\mathrm{on}\;\partial\Omega\times(0,\infty).
    \end{aligned}\right.
\end{equation}
where $u_0=w_0-V$. { In our applications $V$ is singular at some points in $\Omega$. It is thus possible that  $V\cdot \nabla u \neq \nabla \cdot (V\otimes u)$ as distributions over all of $\Omega$ because the singularities can act as sources or sinks. However, if we consider the solution off of the singular set, then this conversion is justified because $V\cdot \nabla u = \nabla \cdot (V\otimes u)$ point-wise. Note that, formally, this issue does not occur for the term ${\nabla \cdot (u\otimes (u+V))}$ because $u$ is divergence-free across the full domain. Furthermore, the possible distributional disagreement we identify here does not apply to Landau solutions---see \cite[Proposition 2.1]{CannoneKarch}. We will use the divergence and convective forms of this term interchangeably in these cases.
This distinction will only matter when we deal with \textit{Type II.5 solutions}, at which point we are careful \textit{not} to use these forms interchangeably. 
}

Plainly, to prove $V$ is asymptotically stable with respect to some norm, it suffices to show that $u$ decays in this norm as $t\to \infty$.
We take $X=L^2(\Omega)$ in this paper and work in a class of finite-energy weak solutions to \eqref{pns} which are analogues of the solutions constructed by Leray in \cite{Leray}.

Let $\dot{H}^1_{0,\sigma}(\Omega)$ be the homogeneous Sobolev spaces whose elements are the closure of the  divergence free vectors of test functions in $H^1$.

\begin{definition}
    \label{def::WeakSolution}
    \cite[Definition 2.6]{KarchPilSchonbek}.
    Fix initial data $u_0\in L^2_\sigma(\Omega)$. 
    We say $u=u(x,t)$ is a \textbf{weak solution} to \eqref{pns} in $\Omega\times[0,T)$  where $T>0$ can be taken to be $\infty$ if and only if 
    \begin{itemize}
        \item (Energy class)
        \begin{equation}
            u\in C_w([0,T),L^2_\sigma(\Omega))\cap L^2([0,T),\dot{H}^1_{0,\sigma}(\Omega)).
        \end{equation}
        
        \item (Satisfies \eqref{pns} weakly) For all $\varphi\in C([0,T),H^1_{0,\sigma}(\Omega))\cap C^1([0,T),L^2_\sigma(\Omega))$,
        \begin{equation}
            \label{eq::pnsDistributionalSolution}
            \begin{aligned}
                \int_\Omega u(t)\cdot\varphi(t)\,dx &+\int_s^t\int_\Omega \nabla u:\nabla\varphi \,dx\,d\tau+\int_s^t \int_\Omega (u\cdot\nabla)u\cdot\varphi \,dx\,d\tau-\int_s^t \int_\Omega (u\cdot\nabla)\varphi\cdot V \,dx\,d\tau\\
                &+\int_s^t\int_\Omega {\nabla \cdot (V\otimes u)}\cdot\varphi \,dx\,d\tau=\int_\Omega u(s)\cdot\varphi(s)\,dx+\int_s^t\int_\Omega u\cdot\partial_\tau \varphi \,dx\,d\tau,
            \end{aligned}
        \end{equation}
        for all $0\leq s\leq t<T$.
        
        \item (Satisfies the energy inequality). 
        We have that,
        \begin{equation}
            \label{eq::EnergyInequality}
            \|u(t)\|_{L^2(\Omega)}^2+2\int_s^t\|\nabla u(\tau)\|_{L^2(\Omega)}^2 \,d\tau \leq \|u(s)\|_{L^2(\Omega)}^2+2\int_s^t \int_\Omega (u\cdot\nabla)u\cdot V{-}\nabla \cdot (V\otimes u )\cdot u\,dx\,d\tau,
        \end{equation}
        for almost every $s\in [0,T)$ (including $s=0$) and every $t\in (s,T)$.
    \end{itemize}
\end{definition}

{
The failure of $V$ to be divergence-free has an impact on the energy inequality because
\[
\int \nabla \cdot (V\otimes u )\cdot u \,dx\neq 0. 
\]When $V$ is small in $L^{3,\infty}$, however, this term can be absorbed in the diffusion. This subtlety  does not effect Landau solutions which are divergence free as distributions across the origin.  Practically, if $V$ is small in a space like $L^{3,\infty}$ then the extra term can be absorbed in the diffusion, in which case this issue is immaterial. In this case the same existence theory for weak solutions to the Navier-Stokes equations is available for the perturbed Navier-Stokes equations.  The following theorem, which for $\R^3$ appears in \cite{KarchPilSchonbek}, makes this assertion precise. Looking ahead we note that things are more nuanced when we deal with Type III $V$ which do not belong to $L^{3,\infty}$.
}


\begin{theorem}[Existence of weak solutions to perturbed SNS]
    \label{existenceThm}
    Suppose $\Omega\subseteq\R^3$ is open and connected. 
    Suppose additionally that $\Omega$ is $\R^3$, $\R^3_+$, or is a bounded or exterior domain with Lipschitz boundary. 
    Suppose $V\in C_w([0,\infty);L^{3,\infty}(\Omega))$ is divergence free. 
    There exists $\epsilon_*>0$ such that if 
    \begin{align*}
        \|V\|_{L^\infty(0,\infty;L^{3,\infty}(\Omega))}<\epsilon_*,
    \end{align*}
    then for any initial data $u_0\in L^2_\sigma(\Omega)$ there exists a vector field $u:\Omega\times [0,\infty)\to\R^3$ such that $u$ is a weak solution to \eqref{pns} in $\Omega\times[0,T]$ for all $T>0$.
\end{theorem}

The proof of this result is identical to the proof in \cite{KarchPilSchonbek} which is done for $\Omega=\R^3$ except we are now possibly dealing with a boundary. 
Note that proofs of existence of weak solutions in the Leray sense are robust with respect to the domain.
For example, \cite[Theorem 3.10]{TsaiBook}, which asserts existence of Leray-Hopf weak solutions to the non-perturbed problem, accommodates $\R^3$, $\R^3_+$ and bounded or exterior domains with Lipschitz boundary.
For us, the crucial fact is that the estimate
\begin{align*}
   \int_\Omega (u\cdot\nabla)u\cdot V+\nabla \cdot (V\otimes u )\cdot u\,dx   \leq C\epsilon_*\|\nabla u\|^2_{L^2(\Omega)},
\end{align*}
holds whenever $u\in H^1_0(\Omega)$ due to a Lorentz space version of the Sobolev inequality.
As this is the insight which reduces the construction of weak solutions to \eqref{pns} to the construction of weak solutions to \eqref{ns} in \cite{KarchPilSchonbek}, we are confident this theorem holds without including a formal proof which would be redundant to the existing literature.

\bigskip

\noindent\textbf{Asymptotic stability results.}
The $L^2(\R^3)$-asymptotic stability around a profile $V$ which is small in $L^{3,\infty}(\R^3)$ (or other spaces with the same scaling) has been studied extensively in \cite{LiYan, KarchPil, KarchPilSchonbek, CKPW}. 
All of the results obtained in those papers assert global asymptotic stability in $L^2(\R^3)$ in that there is no smallness condition on the initial data $u_0$. 
This paper establishes global $L^2(\Omega)$-asymptotic stability results provided $V$ in several new contexts provided $V$ is suitably small.

The papers \cite{KarchPil, KarchPilSchonbek} laid out two different approaches to asymptotic stability. 
In \cite{KarchPil}, Landau solutions were considered in isolation and asymptotic stability was proven using semigroup techniques.
In particular, a semigroup was constructed which included the terms in \eqref{pns} containing $V$ and then the nonlinear problem was viewed perturbatively  so that the nonlinear problem inherits the decay properties of the linear problem.
Because $V$ is built into the semigroup, this approach only works when $V$ is stationary.
As we will see, it is maximally flexible with respect to the domain in view and can be easily generalized beyond $\R^3$. 
In \cite{KarchPilSchonbek}, on the other hand, $V$ is handled directly using energy estimates and the Fourier splitting argument of Schonbek and coauthors to establish $L^2(\R^3)$-decay.
This approach allows $V$ to depend on time but, due to its use of Fourier splitting, is restricted to $\R^3$.
We will use both of these approaches in our results and, additionally, develop a third approach that will allow us to establish asymptotic stability of small non-stationary solutions in domains other than $\R^3$.

\medskip

Our first asymptotic stability result is formulated for stationary profiles. 
While this result is encompassed by \cite{KarchPilSchonbek} when $\Omega=\R^3$, it is new for other choices of $\Omega$, including $\Omega=\R^3_+$.
This result is therefore physically meaningful in that it implies, e.g., Squire's surface discharge flows and their swirling extensions are asymptotically stable under a smallness condition.

\begin{theorem}[Asymptotic stability of stationary solutions]
    \label{mainthm1}
    Fix a domain $\Omega\subseteq\R^3$.\footnote{Throughout this paper domain just means an open and connected set.}
    Suppose $V\in L^{3,\infty}(\Omega)$ is divergence free.
    There exists a positive value $\epsilon_*>0$ such that if $\|V\|_{L^{3,\infty}(\Omega)}<\epsilon_*$ and $u$ is a weak solution of \eqref{pns} in $\Omega$ with no-slip boundary conditions then 
    \begin{align*}
        \lim_{t\to\infty}\|u(t)\|_{L^2(\Omega)}=0.
    \end{align*}
\end{theorem}

The proof of this result is based on the approach taken in \cite{KarchPil} which encodes the terms in \eqref{pns} containing $V$ inside a semigroup. 
The main observation we are making is that this construction is very flexible with regard to the domain in view.
Beyond that, the proof is a straightforward reproduction of the argument in \cite{KarchPil}.

\medskip

In \cite{LiYan} some solutions which are singular on lines are shown to be asymptotically stable. 
The allowed singularities are, however, less severe than $\rho^{-1}$ where $\rho$ is the distance to the $x_3$-axis. 
These solutions are called Type II. 
Solutions where the magnitude along a singular ray is strictly $O(\rho^{-1})$ are shown to exist in \cite{LiLiYanI}. 
These solutions are called Type III. 
{
Previously, there were no results establishing asymptotic stability for perturbations of Type III solutions. In short, this is because the endpoint Hardy inequality is excluded in 2D, that is,
\[
\int_{\R^2} \frac {|f|^2} {|x|^2}\,dx \nleq C \|\nabla f\|_{L^2}^2,
\]
for a constant $C$ and all $f\in \dot H^1$. The estimate goes through, however, if $u$ has certain odd symmetries. Our initial aim was to see if symmetries could be used to prove asymptotic stability for Type III solutions. We provide a partial resolution of this. Our proof shows that, in general, there is a clear mechanism that could drive instability for Type III solutions but this can be avoided in certain sub-classes. We need two types of restrictions. 

The first restricts to perturbations that are axisymmetric. This will allow us to exploit} the fact that, in certain examples of Type III solutions, the vertical component is less singular at the axis than the horizontal components. 
This motivates the following definition of an intermediate class of singular solution to \eqref{sns}.

For $x=(x_1,x_2,x_3)\in\R^3$, let $x_h=(x_1,x_2)$ be the horizontal part of $x$.

\begin{definition}[Type II.5 solutions]
    \label{def::TypeII.5}
    Assume $V$ is a self-similar, axisymmetric solution to \eqref{sns} which belongs to $L^\infty_{\mathrm{loc}}(\Omega\setminus\{x:x_h=0\})$ where $\Omega=\R^3$ or $\Omega=\R^3_+$.
    If $\Omega=\R^3$, then we say that $V$ is a Type II.5 solution (or has a Type II.5 singularity) if 
    \begin{align*}
        |V_{x_3}(x_h,\pm 1)|=O(|x_h|^{-2/3+})
    \end{align*}
    while $V_{x_i}(x_h,\pm 1)=O(|x_h|^{-1})$ for $i=1,2$ and satisfy in a neighborhood of $x_2=0$ the asymptotic expansions 
    \begin{align*}
        V_i(x_1,0,\pm 1)=\frac{c_{i,\pm}}{x_1}+O(|x_1|^{-2/3+}),
    \end{align*}
    where $c_{i,\pm}$ are constants and $x_1>0$. If $\Omega=\R^3_+$, then we require these conditions are satisfied at $x_3=1$ for some constants $c_i$.
\end{definition}

Note that we only defined the asymptotic expansion in the $x_1$-direction where $x_3=\pm 1$ and $x_0=0$. 
By axisymmetry and self-similarity, however, this determines the asymptotics at $x_h=0$ within a conical region around the entire $x_3$ axis.
Let $\vec c_{\pm}$ be the position vector with components $c_{i,\pm}$. 

{ The second restriction is that there is a source of momentum in the horizontal direction at the line singularity. This is formalized by requiring that \begin{align} \label{cond::source} c_{1,\pm}\geq 0.\end{align}
This restriction plays an indirect role in stability as we explain after the next theorem.}
\par 

The structure of Type II.5 solutions is motivated by Serrin's swirling vortex as well as some other examples in \cite{LiLiYanI}---see Section \ref{subsec::Examples::TypeII.5Solutions} where we show there are examples of Type II.5 solutions on $\R^3$ which are outside the asymptotic stability of \cite{LiYan}.

\begin{theorem}[Asymptotic stability of Type II.5 solutions]
    \label{mainthm2}
    Suppose that $V$ is a Type II.5 solution on $\R^3$.
    Then, 
    \begin{align*}
        \lim_{t\to\infty}\|u(t)\|_{L^2(\R^3)}=0,
    \end{align*}
    for all axisymmetric weak solutions $u$ of \eqref{pns} on $\R^3\times (0,\infty)$ { where
    \eqref{eq::EnergyInequality} in the definition of weak solutions is replaced with the condition that
     \begin{align}
      \label{eq::EnergyInequalityModified}
            \|u(t)\|_{L^2}^2+\int_s^t\|\nabla u(\tau)\|_{L^2}^2 \,d\tau \leq \|u(s)\|_{L^2}^2 \text{ for a.e. $s\geq 0$ including $s=0$ and every $t>s$},
    \end{align}
    and which additionally satisfy the generalized energy inequality \eqref{eq::GeneralizedEnergyInequality}.
    }
\end{theorem}

A boundedness property of Type II.5 solutions is formulated more carefully in Section \ref{sec::TypeII.5Stability}---see Definition \ref{def::TypeII.5Boundedness}---but roughly says that $V$ can be decomposed into vector fields $a$ and $b$ where $b$ is in $L^{3,\infty}(\Omega)$ and 
\begin{align*}
    |\partial_{x_3}a|(x)\leq \frac{M}{|x|^2},
\end{align*}
for all $x$ off of the vertical axis.
Essentially what is happening is that the horizontal components of $V$, which are built into $a$, are very singular but their leading order terms near the vertical axis are independent of $x_3$ indicating vertical derivatives are well-behaved. This is reminiscent of a key ingredient in the proof of asymptotic stability for the Burgers vortex \cite{GaMa}, although these may be unrelated.
In other directions where derivatives are not well-behaved, axisymmetry allows us to use an endpoint 2D Hardy inequality that fails generally.


We have not asserted that $V$ is small in this theorem, { nor have we used the assumption that there is a momentum source at the axial singularity.}
Decay comes instead from being assumed to satisfy 
\eqref{eq::EnergyInequalityModified}
on $[0,\infty)$ as this implies $\int_t^\infty \|\nabla u\|_{L^2(\Omega)}^2 d\tau$ becomes small. For Type II.5 solutions, both smallness { and the source condition \eqref{cond::source} are needed to establish this energy estimate as well as the generalized energy estimate. Therefore, although we do not explicitly include them in the theorem, they should not be understood as separate from the theorem.  
The smallness condition is needed for the same reason it appears in, e.g., \cite{KarchPilSchonbek}.  For the source condition, the issue is that the right-hand side of  \eqref{eq::EnergyInequality} is not necessarily finite for Type II.5 solutions but we  still expect to have \eqref{eq::EnergyInequalityModified} when there is a small source at the singularity in the horizontal directions because this leads to a good sign in the extraneous term which has a good sign.  In Remark \ref{rem::SpecExistenceTypeII.5} we provide more details on this.  While our main goal was to understand stability in some Type III solutions, the occurrence of this good sign also points to a possible mechanism for instability, which could occur when there is a sink in which case there is a term with a bad sign.}

\medskip

The preceding results only assert asymptotic stability around stationary flows. 
For non-stationary flows, we do not know how to prove this result in a fully domain independent way. 
Instead, we are able to prove it for a class of domains which we call \textbf{admissible}. We will define these carefully at the start of Section \ref{sec::NonStationaryStability} but, informally, 
these are  domains where a small-data global well-posedness result holds for \eqref{pns} in $L^{3,\infty}(\Omega)$. 
We will show this to be true in, e.g., $\R^3_+$.

\begin{theorem}[Asymptotic stability of non-stationary solutions]
    \label{mainthm3}
    Assume $\Omega\subseteq\R^3$ is admissible. 
    Suppose that $V\in L^\infty(0,\infty;L^{3,\infty}(\Omega))$ is divergence free. 
    There exists $\epsilon_*>0$ so that, if 
    \begin{align*}
        \|V\|_{L^\infty(0,\infty;L^{3,\infty}(\Omega))}<\epsilon_*,
    \end{align*}
    and $u$ is a weak solution to \eqref{pns} for $V$, then 
    \begin{align*}
        \lim_{t\to\infty}\|u(t)\|_{L^2(\Omega)}=0.
    \end{align*}
\end{theorem}

In comparison to the previous literature, this result requires a new observation.
In particular, on $\R^3$ the asymptotic stability of non-stationary profiles is proven using Fourier splitting \cite{KarchPilSchonbek}.
This is not directly available on other domains like  $\R^3_+$.
Instead, we use eventual regularity to obtain asymptotic stability.
In particular, weak solutions to \eqref{pns} eventually become small in $L^{3,\infty}(\Omega)$. 
Once this occurs, they stay small in $L^{3,\infty}(\Omega)$ provided a small-data global well-posedness result---see Theorem \ref{thm::GlobalWellPosedness}---holds in $L^{3,\infty}(\Omega)$.
This means that the nonlinear term has the same smallness as the linear product terms involving $V$.
Intuitively, it is then possible to show everything goes to zero by absorbing all the product terms into the left-hand side of a mild solution estimate.

While the examples we consider in this paper are stationary, the non-stationary result can be applied to obtain a new asymptotic stability  result for small forward self-similar solutions on $\R^3_+$---see, e.g., \cite{JS,KT,BTrot}---as well as solutions where a point source of momentum moves with time as constructed by Karch and Zheng \cite{KZ}.  

\bigskip
\noindent{\bf Discussion of the approaches to asymptotic stability.} 
We conclude this section with summary remarks about the three different approaches taken to proving Theorems \ref{mainthm1}, \ref{mainthm2} and \ref{mainthm3}. 
There are two broad paradigms for generating estimates for Navier-Stokes flows: Energy estimates and fundamental solution/semigroup techniques. 
The proof of Theorem \ref{mainthm1} uses semigroup theory and is based on \cite{KarchPil}. 
The strength of this approach is that it is flexible with respect to the domain of the problem. 
The drawback is that it only works for perturbations around stationary solutions. 

The proof of Theorem \ref{mainthm2} is a hybrid approach in that it primarily uses energy estimates but also uses small-data global well-posedness (proven using semigroups) and weak-strong uniqueness for \eqref{pns} in $L^{3,\infty}$. 
It is restricted to domains where small-data global well-posedness can be proven.  
We identify sufficient conditions for this in Section \ref{sec::NonStationaryStability} but note this is not a topic that has been explored exhaustively. 

An additional drawback of the preceding approaches which  use semigroups is that the term $V\cdot \nabla u$ appears. 
In a purely energy method-based approach, this term vanishes. It presents an obstacle when bounding perturbations around Type II.5 solutions.
This led us to use the Fourier splitting approach of \cite{KarchPilSchonbek} to prove Theorem \ref{mainthm3}. 
A strength of this approach is that it can accommodate perturbations around time-dependent solutions to \eqref{ns}.  
The drawback of this approach is that, without major adaptations, it is limited to $\R^3$ due to its use of the Fourier transform. 
Furthermore, it only applies to the subset of weak solutions that  satisfy a generalized energy inequality. It is not known if this is a strict subset. 
Theorem \ref{mainthm2}, on the other hand, is flexible with respect to the domain of the problem and applies to any weak solution.  

These approaches arose from earlier work on the $L^2$-decay of Leray weak solutions to  the non-perturbed Navier-Stokes problem. 
A nice summary of this is contained in \cite{Meyer2}. 
The first development was due to Kato \cite{Kato} who proved $L^2$-decay for Leray weak solutions by exploiting the fact that the $L^3$-norm of a Leray weak solution eventually becomes small.
This is also the idea behind the proof of Theorem \ref{mainthm3} but we must work in $L^{3,\infty}$ to accommodate the singular terms in \eqref{pns}.  
Kato's approach gave stronger algebraic decay rates when $u_0$ is also in $L^p$ for $1\leq p<2$. 
The Fourier-splitting approach to $L^2$-decay was subsequently developed by M.E.~Schonbek \cite{Schon1,Schon2} and improved Kato's algebraic bounds. 
See also \cite{KaMi,MiSchon}. 
Brandolese later showed that symmetries could be used to achieve faster decay rates than are available for generic solutions \cite{Brandolese}.

Asymptotic stability has also been established for $L^p$-perturbations for some $p\neq 2$ by Li, Zhang and Zhang \cite{LZZ}---see also \cite{BW}.

\bigskip
\noindent{\bf Organization.} 
In Section \ref{sec::Preliminaries} we document some useful inequalities and define some function spaces. 
Section \ref{sec::Examples} includes a detailed analysis of existing and new classes of singular stationary solutions emphasizing what happens when boundaries are present. 
The proofs of Theorems \ref{mainthm1}, \ref{mainthm2} and \ref{mainthm3} are contained in Sections \ref{sec::StationaryStability}, \ref{sec::TypeII.5Stability} and \ref{sec::NonStationaryStability}.

{
\bigskip
\noindent{\bf Statement on the use of AI.} After submitting this paper for publication we asked Claude Opus 5.5 to generate a sample referee report. It found an error in our original draft where we had assumed 
\[
\int ( V\cdot \nabla u) \cdot u \,dx=0, 
\]
which fails because $V$ is not necessarily divergence-free at the singular points of $V$. This only impacted our work on Type II.5 solutions because, for Type II solutions, the same estimate that controls the term $\int (u\cdot\nabla V) \cdot u \,dx$ can be applied to this term.  In order to correct this we needed to re-formulate the problem we are perturbing around to be in divergence form. We also needed to impose the condition that $V$ experiences a source of momentum in the horizontal direction along its singular axis. Our corrections were generated independently of AI. Beyond this, AI was not used in the formulation or drafting of the mathematical work in this paper. 
 }

\section{Preliminaries}
\label{sec::Preliminaries}

Fix a point $x\in\R^3$, denote its components in the standard basis by $(x_1,x_2,x_3)$. 
We let $x_h=(x_1,x_2)$.
A domain $\Omega\subseteq\R^3$ is any open and connected set.
Define the Hilbert space $H_0^1(\Omega)$ as the closure of $C_c^\infty(\Omega)$ in the $H^1$ norm. 
For domains where the trace map is well-defined, this corresponds to $f\in H^1(\Omega)$ with zero trace.

We use the subscript $\sigma$ on Banach spaces to mean divergence free, e.g., $C_{c,\sigma}^\infty(\Omega)$ denotes the divergence free test functions.
Here we are omitting the codomain in the more precise notation $C_{c,\sigma}^\infty(\Omega;\R^n)$. 
When context demands it for clarity, we will include the exponent in situations like this.
For any domain $\Omega$ we define 
\begin{align*}
    L^2_\sigma(\Omega):=\overline{C_{c,\sigma}^\infty(\Omega)}^{L^2(\Omega)}
    \qquad 
    \text{and}
    \qquad 
    H^1_{0,\sigma}(\Omega):=\overline{C_{c,\sigma}^\infty(\Omega)}^{H^1(\Omega)}.
\end{align*}

Fix $1\leq p<\infty$ and $1\leq q\leq\infty$. 
For a real (or vector) valued function $f$ on $\Omega\subseteq\R^3$, the $L^{p,q}(\Omega)$ Lorentz space \textit{quasinorm} is given by the formula 
\begin{equation}
    \label{eq::LorentzQuasinorm}
    \|f\|_{L^{p,q}(\Omega)}:=\begin{cases}
        p^{1/q}\left(\int_0^\infty \lambda^{q-1}\mu(\{x:|f(x)|\geq\lambda\})^\frac{q}{p}\right)^{1/q} & 1\leq q<\infty\\
        \sup_{\lambda>0}(\lambda^p\mu(\{x:|f(x)|\geq \lambda\}))^{1/p}                                   & q=\infty
    \end{cases},
\end{equation}
The Banach spaces $L^{p,q}(\Omega)$ is the space of measurable functions $f$ (up to equivalence almost everywhere) which satisfies 
\begin{align*}
    \|f\|_{L^{p,q}(\Omega)}<\infty.
\end{align*}
We usually denote $L^{p,q}(\Omega)$ by just $L^{p,q}$ if the domain is clear. 
The space $L^{p,\infty}(\Omega)$ corresponds to the usual weak-$L^p(\Omega)$ space.

The following theorem recalls useful inequalities in the Lorentz space setting.

\begin{theorem}
    \label{thm::LorentzInequalities}
    Fix a domain $\Omega\subseteq\R^3$ and real numbers $1\leq p,q,r<\infty$ and $1\leq h-1,h-2,h_3\leq\infty$.
    Suppose $f\in L^{p,h_1}(\Omega)$ and $g\in L^{q,h_2}(\Omega)$.
    \begin{enumerate}
        \item[(a)] (Lorentz-Holder) Suppose 
        \begin{align*}
            \frac{1}{p}+\frac{1}{q}=\frac{1}{r}
            \qquad 
            \text{and}
            \qquad 
            \frac{1}{h_1}+\frac{1}{h_1}=\frac{1}{h_3}.
        \end{align*}
        Then $fg\in L^{p,h_3}(\Omega)$ and satisfies
        \begin{equation}
            \label{eq::LorentzHolder}
            \|fg\|_{L^{r,h_3}(\Omega)}\leq \|f\|_{L^{p,h_1}(\Omega)} \|g\|_{L^{q,h_2}(\Omega)},
        \end{equation}

        \item[(b)] (Lorentz-Young) Suppose $\Omega=\R^3$,
        \begin{align*}
            \frac{1}{p}+\frac{1}{q}=1+\frac{1}{r}
            \qquad
            \text{and}
            \qquad 
            \frac{1}{h_1}+\frac{1}{h_2}=\frac{1}{h_3}.
        \end{align*}
        Then $f*g\in L^{r,h_3}(\R^3)$ and satisfies
        \begin{equation}
            \label{eq::LorentzYoung}
            \|f*g\|_{L^{r,h_3}(\R^3)}\leq C\|f\|_{L^{p,h_1}(\R^3)}\|g\|_{L^{q,h_2}(\R^3)},
        \end{equation}
        for some constant $C>0$ independent of $f$ and $g$.

        \item[(c)] (Lorentz-Sobolev) If $f\in H^1_0(\Omega)$ then $f\in L^{6,2}(\Omega)$ and satisfies
        \begin{equation}
            \label{eq::LorentzSobolev}
            \|f\|_{L^{6,2}(\Omega)}\leq C_S\|\nabla f\|_{L^2(\Omega)},
        \end{equation}
        for some constant $C_S>0$ independent of $f$.
    \end{enumerate}
\end{theorem}

The Lorentz-Holder inequality is well known. The Lorentz-Young inequality is usually referred to as O'Neil's inequality. The Lorentz-Sobolev inequality is usually stated for $f\in H^1(\R^3)$---see \cite{MM}. For $f\in C_c^\infty(\Omega)$ we can extend by zero to obtain \eqref{eq::LorentzSobolev} for $C_c^\infty(\Omega)$ from \eqref{eq::LorentzSobolev} for $H^1(\R^3)$. Then \eqref{eq::LorentzSobolev} for $H^1_0(\Omega)$ follows by density.

\section{Examples of singular steady-state solutions}
\label{sec::Examples}

\subsection{Landau solutions and solutions with multiple point sources}
\label{subsec::Examples::LandauSolutions}

Landau discovered a well-known two-parameter family of $(-1)$-homogeneous solutions, known as the Landau solutions, in 1944 \cite{Landau}. 
In spherical coordinates, they are rotations of 
\begin{align*}
    V_r=\frac{2}{r}\left(\frac{\lambda^2-1}{(\lambda-\cos(\theta)^2}-1\right),
    \qquad 
    V_\theta=-\frac{2\sin(\theta)}{r(\lambda-\cos(\theta))},\qquad V_\phi=0, 
    \qquad 
    q=-\frac{4(\lambda\cos(\theta)-1)}{r^2(\lambda-\cos(\theta))^2},
\end{align*}
for $\lambda>1$.
Landau solutions are used to model a submerged jet.
They were independently discovered by H. Squire in 1951 \cite{Squire}.
In 2006, Vladim\`{\i}r \v{S}ver\`ak shows that Landau solutions are the only $(-1)$-homogeneous solutions to \eqref{sns} in $\R^3\setminus\{0\}$ \cite{Sverak}.
These solutions also capture the leading order behavior of small, stationary solutions to \eqref{ns} \cite{KoSv,MiTs}.
Higher order asymptotics as $|x|\to\infty$ have recently been worked out by Jia and \v{S}ver\`ak \cite{JiSv}.

In the classical literature these solutions are derived by assuming additional symmetries to reduce the momentum equation in \eqref{sns} to the single ODE
\begin{equation}
    \label{eq::RicattiODE}
    (1-x^2)v_\theta'+2xv_\theta+\frac{1}{2}v_\theta^2=c_1(1-x)+c_2(1+x)+c_3(1-x^2),
\end{equation}
with $c_1=c_2=c_3=0$.
Here, $x=\cos(\theta)$, $v_\theta=V_\theta\sin(\theta)$ and $'$ denotes differentiation in $x$.
This reduction was first performed by Slezkin in 1934 \cite{Slezkin}.
In 1950 Yatseyev classified the solutions to the ODE using hypergeometric functions \cite{Yatseyev}.
More recently, Li, Li and Yan study generalizations that have singularities on a single axis or a single half-axis \cite{LiLiYanI, LiLiYanII, LiLiYanIII}.
On $\R^3_+$, homogeneous solutions can be constructed which converge in $\R^3_+$ but which are non-zero on the boundary \cite{Squire2} or which have an internal singular axis but which vanish on the boundary \cite{Serrin}.
These extensions are discussed in Sections \ref{subsec::Examples::LiLiYanTypeII} and \ref{subsec::Examples::SquireSurfaceDischargeFlows} and correspond to certain choices of solutions to \eqref{eq::RicattiODE} where $c_i\neq 0$.

Solutions can typically be constructed for small terms generalizing those of Landau solutions.
For example, Decaster and Iftimie constructed generalizations of Landau solutions for any sources that are localized and $(-3)$-homogeneous \cite{DeIf}.
Karch and Zheng \cite{KZ} constructed an analogue to the Landau solutions which is time-dependent and forced by a moving point source of momentum under a smallness condition.
These solutions are singular on a sufficiently regular curve in space-time and stable \cite{KSS}.

Broman and Rudenko studied this problem and have a picture \cite[Figure 2]{BR} of a double sided submerged jet which they say has applications to studying acoustic waves.
They obtain this by reflecting a flow example found in \cite{Squire2}---see Section \ref{subsec::Examples::SquireSurfaceDischargeFlows}---across the planar boundary.
The ``solution'' they generate is $(-1)$-homogeneous but is not a Landau jet.
This contradicts \cite{Sverak}, indicating the field they obtain by reflection is not actually a solution to Navier-Stokes.
The issues seems to be that they did not check that the Navier-Stokes equations are satisfied across the $\theta=\pi/2$ plane.
We find this picture interesting, however, and are curious to see if it can be replicated by introducing two opposing facing point-sources of momentum at point spaced a non-zero distance apart---see Figure \ref{fig::OpposingPoint} for an image that recovers aspects of this.
A similar picture can be sought within the framework of \cite{LiLiYanIII} where the solutions are singular on the axis of symmetry and driven by opposing jets.

In the remainder of this section we construct small stationary solutions which are subjected to an arbitrary number of point sources of momentum for a variety of domains.
This includes the example mentioned above capturing some of the qualitative features of the \cite[Figure 2]{BR}. 
Small Landau solutions can be obtained from a fixed point argument starting with the fundamental solution to the Stokes equations \cite{DeIf}.
In particular, let $V^{(0)}$ be the fundamental solution to the Stokes equations for a fixed point source of momentum $\delta_0 a\vec{e}_3$ where $a\in\R$. 
Now consider in $L^{3,\infty}(\R^3)$. the problem 
\begin{align*}
    -\Delta V^{(1)}+(V^{(0)}\cdot \nabla)V^{(0)}+\nabla q_1=\delta_0 a\vec{e}_3,
    \qquad 
    \nabla\cdot V^{(1)}=0.
\end{align*}
We would like to solve this.
Instead of solving this directly, we subtract $V^{(0)}$ to obtain
\begin{align*}
    -\Delta(V^{(1)}-V^{(0)})+(V^{(0)}\cdot\nabla)V^{(0)}+\nabla(q_1-q_0)=0,
    \qquad 
    \nabla\cdot V^{(1)}=0.
\end{align*}
Then, $V^{(1)}-V^{(0)}$ can be constructed by solving the Stokes equation with force $-(V^{(0)}\cdot\nabla)V^{(0)}$. 
Let $G$ be the Green tensor for the linear Stokes operator in $\R^3$. 
Then formally we have
\begin{align*}
    V^{(1)}=-G *\left(\nabla\cdot (V^{(0)}\otimes V^{(0)})\right)+V^{(0)}.
\end{align*}
We repeat this letting $V_i$ satisfy 
\begin{align*}
    -\Delta V^{(k)}+(V^{(k-1)}\cdot\nabla)V^{(k-1)}+\nabla q_k=\delta_0 a\vec{e}_3,
    \qquad 
    \nabla\cdot V^{(k)}=0.
\end{align*}
One can than prove that the sequence $\{V^{(k)}\}_{k=1}^\infty$ eventually stabilizes using a fixed point argument and the following lemma.

\begin{lemma}
    \label{lem::GreenTensorConvolutionEstimate}
    For a domain $\Omega\subseteq\R^3$, let $(G_\Omega,\Pi_\Omega)$ denote the Green tensor for the Stokes equation with zero boundary conditions. 
    Fix $\Omega$ such that $G_\Omega(x,y)$ satisfies
    \begin{align*}
        |\nabla_y G_\Omega(x,y)|\leq \frac{C}{|x-y|^2},
    \end{align*}
    for some constant $C>0$ and all $x,y\in\Omega$. 
    Define the bilinear operator 
    \begin{align*}
        B_\Omega(u,v):=\int_\Omega \nabla_y G_\Omega(x,y):(u\otimes v)(y)dy.
    \end{align*}
    Then there exists some constant $C_\Omega>0$ such that 
    \begin{align*}
        \|B_\Omega(x,y)\|_{L^{3,\infty}(\Omega)}\leq C_\Omega\|u\|_{L^{3,\infty}(\Omega)}\|v\|_{L^{3,\infty}(\Omega)},
    \end{align*}
    for any vector fields $u,v\in L^{3,\infty}(\Omega)$.
\end{lemma}

We supply two examples of $\Omega$ beyond $\Omega=\R^3$ which satisfies the above condition. Firstly, the half-space satisfies $|\nabla_y G_{\R^3_+}(x,y)|\leq \frac{C}{|x-y|^2}$ by \cite[Theorem 2.5]{KMT}. Secondly, let $\Omega$ be a bounded domain with $C^{1,\mathrm{Dini}}$ boundary. In \cite[Theorem 7.3]{ChoiDong} and \cite[Remark 3.2]{ChoiDong} it is shown that $|\nabla_y G_\Omega(x,y)|\leq \frac{C}{|x-y|^2}$ whenever $|x-y|\leq 1$. By compactness of $\overline{\Omega}$ we conclude the global estimate.

\begin{details}
    The exact theorem from \cite[Theorem 2.5]{KMT} is:
    \begin{theorem}
        Fix $n\geq 2$. Let $x,y\in\R^n_+$ and $i,j\in\{1,...,n\}$. Let $\alpha$ and $\beta$ be multi-indices with $|\alpha|+|\beta|=m>0$. Then 
        \begin{align*}
            \left|\nabla_x^\alpha\nabla_y^\beta G_{ij}(x,y)\right|\leq \frac{C_m}{|x-y|^{n-2+m}}.
        \end{align*}
        If... (special cases).
        Above $C_m$ are independent of $x,y\in\R^n_+$.
    \end{theorem}
    We apply this theorem three times with $n=3$, $\alpha=(0,0,0)$, $\beta\in\{(1,0,0),(0,1,0),(0,0,1)\}$, each giving $m=0+1=1$. So we have 
    \begin{align*}
        \left|\nabla_y G_{ij}(x,y)\right|\leq \left|\nabla_y^{(1,0,0)}G_{ij}(x,y)\right|+\left|\nabla_y^{(0,1,0)}G_{ij}(x,y)\right|+\left|\nabla_y^{(0,0,1)}G_{ij}(x,y)\right| \leq \frac{3C_1}{|x-y|^2}
    \end{align*}
\end{details}

\begin{details}
    The exact theorem from \cite[Theorem 7.3]{ChoiDong} is:
    \begin{theorem}
        Let $d\geq 3$ and $\Omega$ be a domain in $\R^d$ having a $C^{1,\mathrm{Dini}}$ boundary as in Definition 2.2. 
        Suppose that the coefficients $A^{\alpha\beta}$ of $\cL$ are of Dini mean oscillation in $\Omega$ satisfying Definition 2.1 (b) with a Dini function $\omega=\omega_A$. 
        Then under Assumption 3.1, there exists a unique Green function $(G_\Omega,\Pi_\Omega)$ for the flow velocity of $\cL$ in $\Omega$ such that for any $y\in\Omega$, 
        \begin{align*}
            G_\Omega(\cdot,y)\;\text{ is continuously differentiable in }\;\bar{\Omega}\setminus\{y\}
        \end{align*}
        and 
        \begin{align*}
            \Pi_\Omega(\cdot,y)\;\text{ is continuous in }\;\bar{\Omega}\setminus\{y\}.
        \end{align*}
        Moreover, for any $x,y\in\Omega$ with $0<|x-y|\leq 1$, we have that 
        \begin{align*}
            |G_\Omega(x,y)|\leq C\min\{d_x,|x-y|\}\cdot\min\{d_y,|x-y|\}\cdot|x-y|^{-d},
        \end{align*}
        \begin{align*}
            |\nabla_xG_\Omega(x,y)|+|\Pi_\Omega(x,y)|\leq C|x-y|^{1-d},
        \end{align*}
        where $C=C(d,\lambda,\omega_A,K_0,R_0,\rho_0)$. 
        Furthermore, if $(G^*,\Pi^*)$ is the Green function for the flow velocity of $\cL^*$ in $\Omega$, then we have 
        \begin{align*}
            G_\Omega(x,y)=G_\Omega^*(y,x)^\top\qquad\text{for all }\;x,y\in\Omega,\;\; x\neq y.
        \end{align*}
    \end{theorem}
    We apply this result with $d=3$, $\Omega\subseteq\R^3$ bounded with $C^1{\mathrm{Dini}}$ boundary. 
    We are working with the classical Stokes system so $A^{\alpha,\beta}=[-\delta_{\alpha\beta}\delta_{ij}]^3_{i,j=1}$. 
    This is stated on \cite[page 2]{ChoiDong}, however we have a extra minus sign because they operate under the convention $A G(\cdot,y)+\nabla\Pi(\cdot,y)=\delta_yI$. 
    Since the coefficients are constant, any assumptions on their regularity is automatically satisfied. 
    To apply the theorem, we must check that Assumption 3.1 is satisfied.
    The exact statement of \cite[Assumption 3.1]{ChoiDong} is:
    \begin{assumption}
        There exists a constant $K_0>0$ such that the following holds: For any $g\in\tilde{L}^2(\Omega)$, there exists $u\in Y_0^{1,2}(\Omega)^d$ satisfying 
        \begin{align*}
            \nabla\cdot u=g\;\mathrm{in}\;\Omega,\qquad \|\nabla u\|_{L^2(\Omega)}\leq K_0\|g\|_{L^2(\Omega)}.
        \end{align*}
    \end{assumption}
    Assumption 3.1 is satisfied by the following excerpt of \cite[Remark 3.2]{ChoiDong}:
    \begin{remark}
        It is known that Assumption 3.1 holds in a bounded John domain; see [2, Theorem 4.1]. 
        Here and throughout the paper, a domain is said to be bounded if it has finite diameter.
        Note that a bounded domain $\Omega$ having a $C^{1,\mathrm{Dini}}$ boundary as in Definition 2.2 is a John domain as in [2, Definition 2.1] with respect to $(x_0,L)=(x_0,L)(d,R_0,\rho_0)$. 
        Thus by [2, Theorem 4.1], $\Omega$ satisfies Assumption 3.1 with $K_0=K_0(d,R_0,\rho_0,\mathrm{diam}(\Omega))$.
    \end{remark}
    So we are free to apply \cite[Theorem 7.3]{ChoiDong} to our situation. 
    On \cite[Page 5]{ChoiDong}, Choi and Dong define 
    \begin{align*}
        \cL u=D_\alpha(A^{\alpha\beta}D_\beta u)
    \end{align*}
    and
    \begin{align*}
        \cL^*u=D_\alpha((A^{\beta\alpha})^\top D_\beta u).
    \end{align*}
    In our situation, 
    \begin{align*}
        (A^{\beta\alpha})^\top=\left([-\delta_{\beta\alpha}\delta_{ij}]^3_{i,j=1}\right)^\top=[-\delta_{\alpha\beta}\delta_{ij}]^3_{i,j=1}=A^{\alpha\beta}.
    \end{align*}
    So $\cL^*$ also satisfies the condition of the theorem. 
    So (by \cite[7.5,7.6]{ChoiDong}) there exists a constant $C_1>0$ such that
    \begin{align*}
        |\nabla_y G_\Omega(x,y)|=|\nabla_y G_\Omega(y,x)^\top|\leq C_1|x-y|^{-2}.
    \end{align*}
    whenever $0<|x-y|\leq 1$. 
    Let $D\subset \bar{\Omega}\cap\bar{\Omega}$ denote the diagonal. 
    Take a open neighborhood $\cO$ of $D$ in $\bar{\Omega}\times\bar{\Omega}$ which satisfies $|x-y|<1$ for every $(x,y)\in\cO$.
    Then $\nabla_y G_\Omega(x,y)=\nabla_yG_\Omega(y,x)^\top$ is a continuous function (by \cite[7.2]{ChoiDong}) on the compact domain $(\bar{\Omega}\times\bar{\Omega})\setminus \cO$.
    By the extreme value theorem, there exists a constant $M>0$ such that
    \begin{align*}
        |\nabla_y G_\Omega(x,y)\leq M\leq \frac{M(\mathrm{diam}(\Omega))^2}{|x-y|^2}. 
    \end{align*}
    Let $C=\max\{C_1,M(\mathrm{diam}(\Omega))^2\}$. 
    Then $G_\Omega(x,y)$ satisfies
    \begin{align*}
        |\nabla_y G_\Omega(x,y)|\leq \frac{C}{|x-y|^2}
    \end{align*}
    for all $x,y\in\Omega$ with $x\neq y$.
\end{details}

\begin{proof}
    Since $|\nabla_y G_\Omega(x,y)|\leq \frac{C}{|x-y|^2}$, we may use Lorentz-Young to obtain
    \begin{align*}
        \|B_\Omega(u,v)\|_{L^{3,\infty}(\Omega)}&\leq C\left\|\int_\Omega \frac{|u(y)||v(y)|}{|x-y|^2}dy\right\|_{L^{3,\infty}(\Omega)}\\
        &\leq C\||\cdot|^{-2}*F\|_{L^{3,\infty}(\R^3)}\\
        &\leq C\||\cdot|^{-2}\|_{L^{3/2,\infty}(\R^3)}\|F\|_{L^{3/2,\infty}(\R^3)},
    \end{align*}
    where $F$ is the extension by zero of $|u(y)||v(y)|$. 
    Since $|\cdot|^{-2}\in L^{3/2,\infty}(\R^3)$ we obtain 
    \begin{align*}
        \|B_\Omega(u,v)\|_{L^{3,\infty}(\Omega)}\leq C\|F\|_{L^{3/2,\infty}(\Omega)}\leq C\|u\|_{L^{3,\infty}(\Omega)}\|v\|_{L^{3,\infty}(\Omega)},
    \end{align*}
    which proves the lemma.
\end{proof}

\begin{theorem}
    \label{thm::ExistenceOfSolutionsWithFiniteNumberOfPointSingularities}
    For $\Omega\subseteq\R^3$, let $(G_\Omega,\Pi_\Omega)$ be the Green tensor for the Stokes equation with zero boundary conditions.
    Let $\Omega\subseteq\R^3$ be a domain with $|G_\Omega(x,y)|\leq \frac{C}{|x-y|^2}$.
    Let $x_1,...,x_N\in\Omega$ be a finite collection of interior points and $b_1,...,b_N\in\R^3$ be the force vectors at each $x_i$.
    Let $\delta_{x_i}$ denote the Dirac delta distribution supported at $x_i$ and define 
    \begin{align*}
        V^{(0)}(x):=\sum_{i=1}^N G_\Omega(x,x_i)b_i
        \qquad 
        q^{(0)}:=\sum_{i=1}^N \Pi_\Omega(x,x_i)b_i. 
    \end{align*}
    There exists $\epsilon>0$ such that there exists a distributional solution $(V,q)\in L^{3,\infty}(\Omega)\times L^{3/2,\infty}_{\mathrm{loc}}(\Omega)$ to 
    \begin{equation}
        \label{eq::StationaryNavierStokesEquations}
        -\Delta V+(V\cdot\nabla)V+\nabla q=\sum_{i=1}^N b_i\delta_{x_i}=:f,
    \end{equation}
    in $L^{3,\infty}(\Omega)$ satisfying zero boundary conditions and $\|V\|_{L^{3,\infty}(\Omega)}<2\epsilon$ whenever $\|V^{(0)}\|_{L^{3,\infty}(\Omega)}<\epsilon$.
\end{theorem}

\begin{figure}[ht]
    \centering
    \begin{subfigure}{0.48\textwidth}
        \centering
        \includegraphics[width=\linewidth]{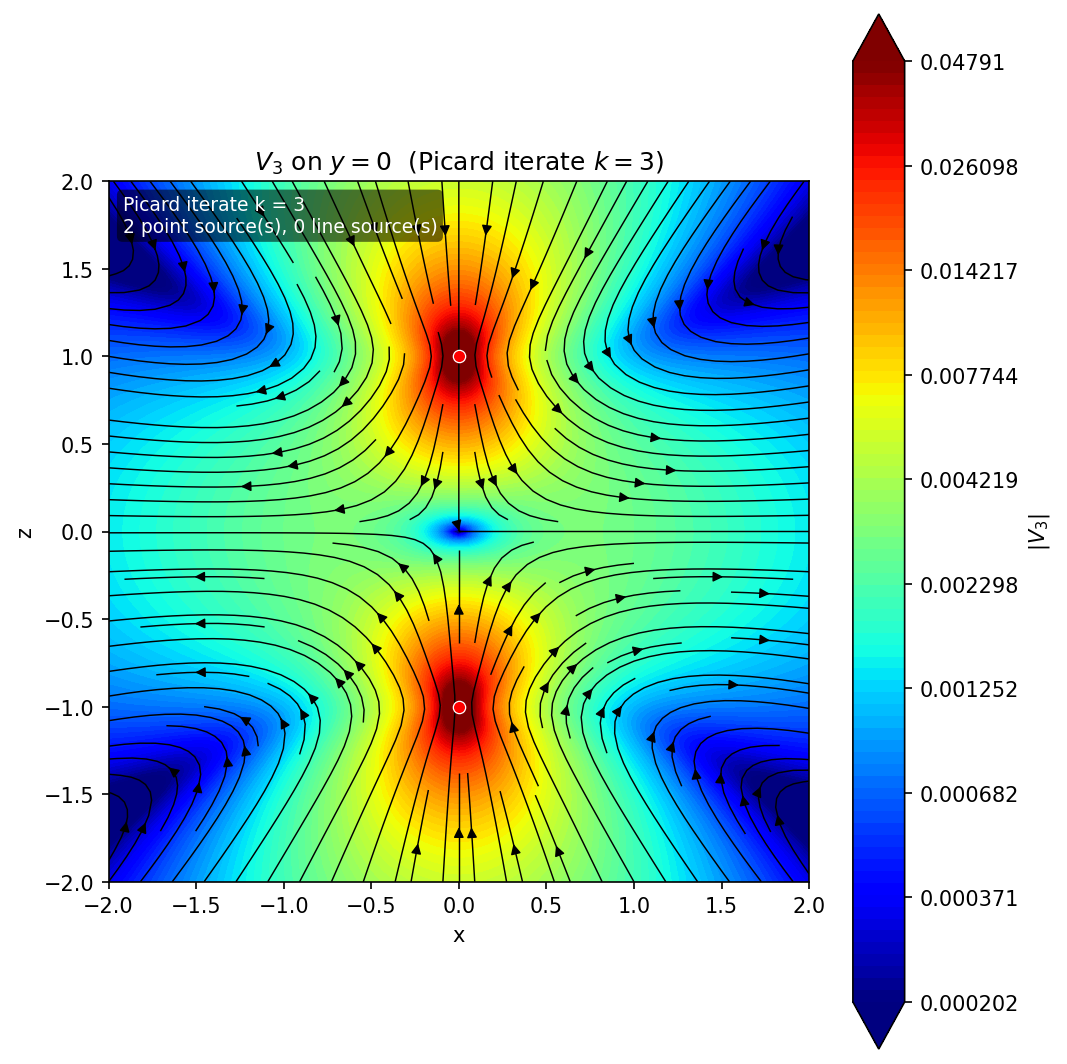}
        \caption{Opposing point sources.}
    \end{subfigure}
    \begin{subfigure}{0.48\textwidth}
        \centering
        \includegraphics[width=\linewidth]{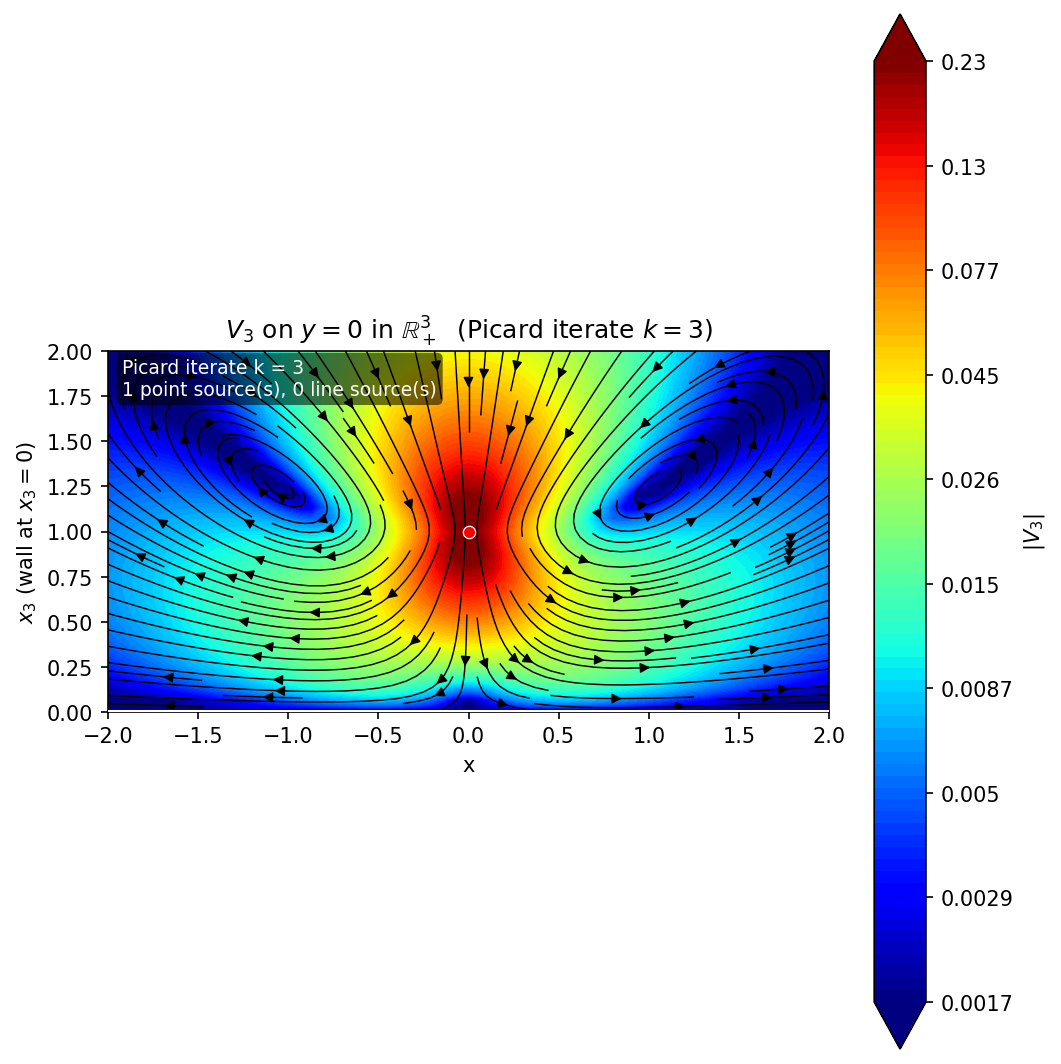}
        \caption{Half-space point source.}
    \end{subfigure}
    \caption{The opposing point-sources in (A) generate a flow satisfying the Euler boundary conditions at the $z=0$ plane whereas in (B) a no-slip boundary condition is enforced.}
    \label{fig::OpposingPoint}
\end{figure}

\begin{figure}[ht]
    \centering
    \begin{subfigure}{0.48\textwidth}
        \centering
        \includegraphics[width=\linewidth]{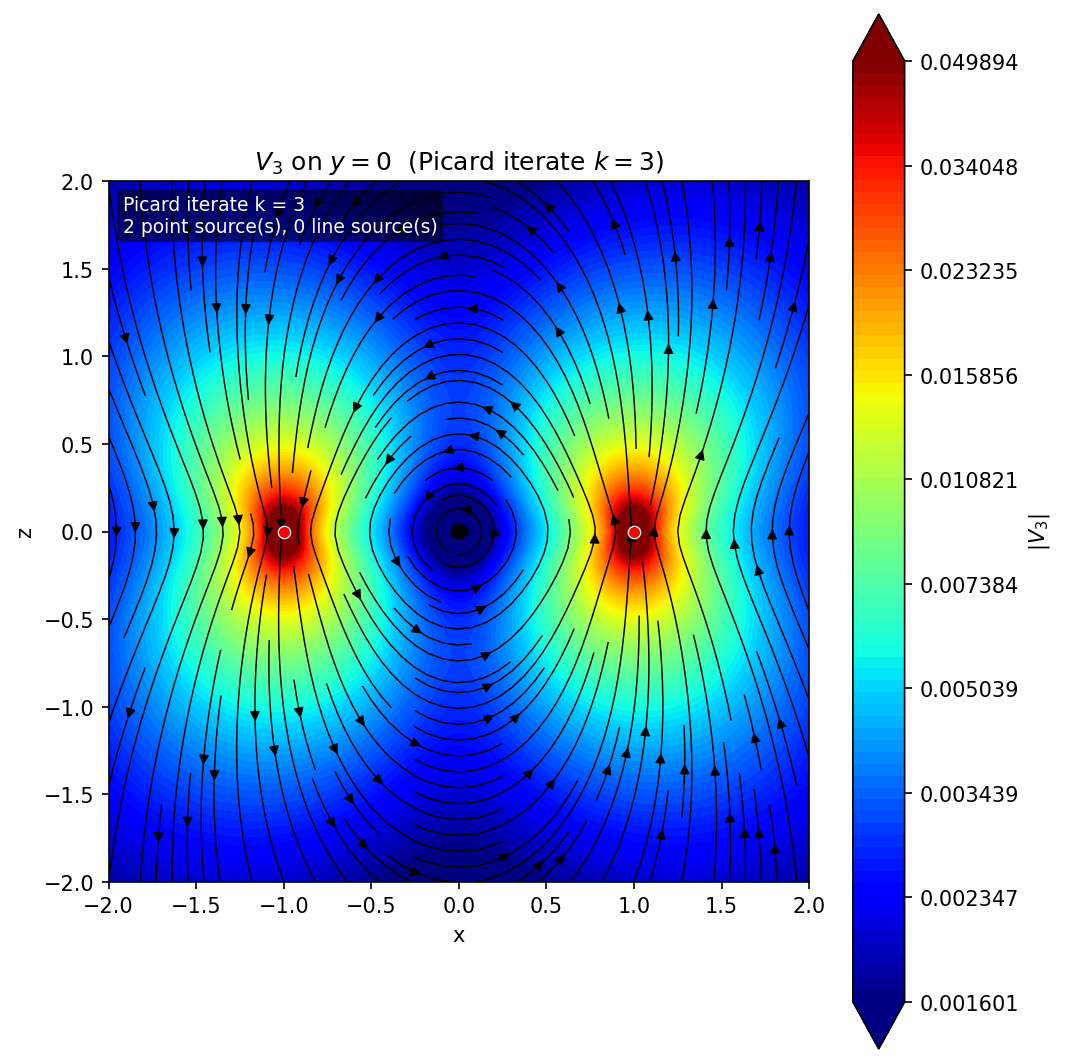}
        \caption{Swirl generated by two point-sources.}
    \end{subfigure}
    \begin{subfigure}{0.48\textwidth}
        \centering
        \includegraphics[width=\linewidth]{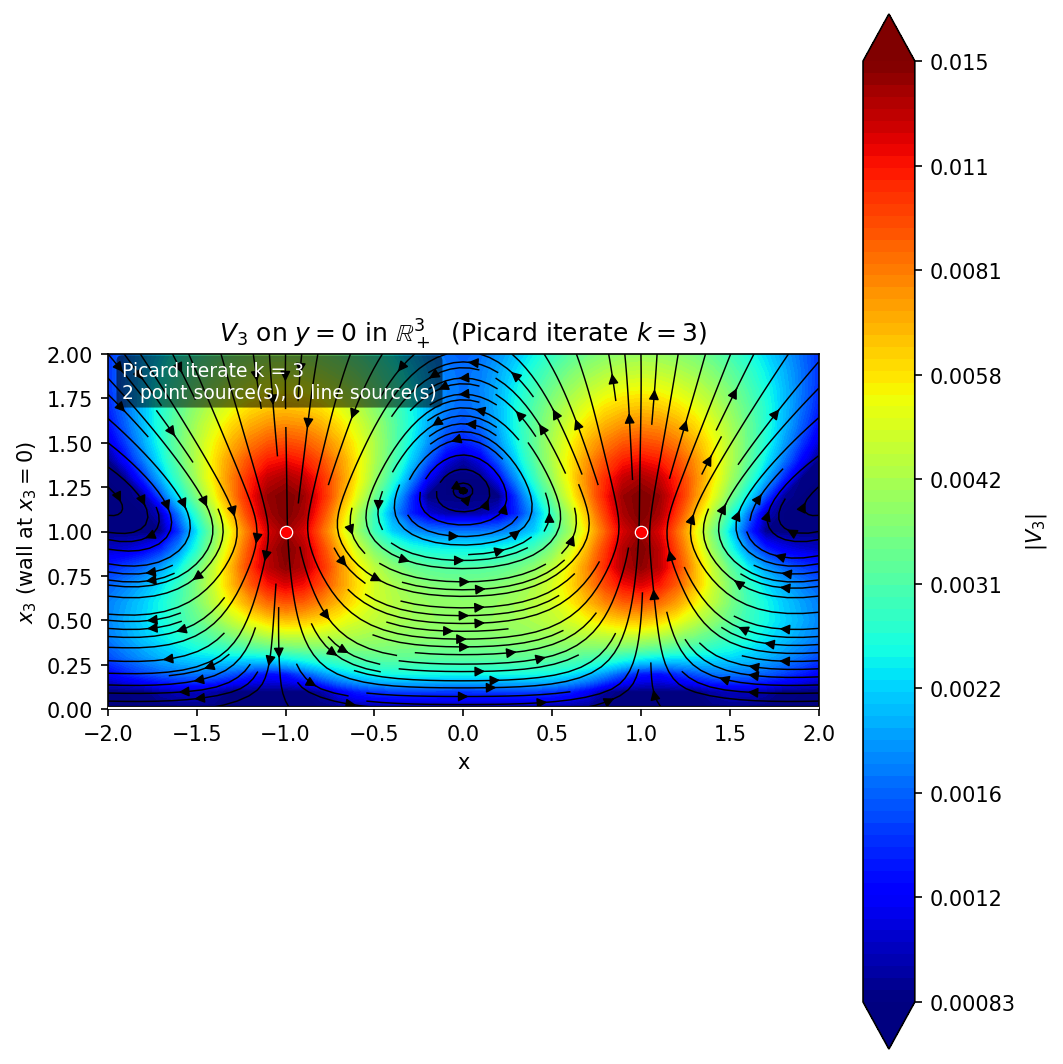}
        \caption{The same  point-sources  but near a boundary.}
    \end{subfigure}
    \caption{The point-sources in (A) generate a stable rotation. In (B) the effect of a boundary is to deform this rotation and push it away from the boundary.}
    \label{fig::SwirlGen2}
\end{figure}

\begin{figure}[ht]
    \centering
    \begin{subfigure}{0.48\textwidth}
        \centering
        \includegraphics[width=\linewidth]{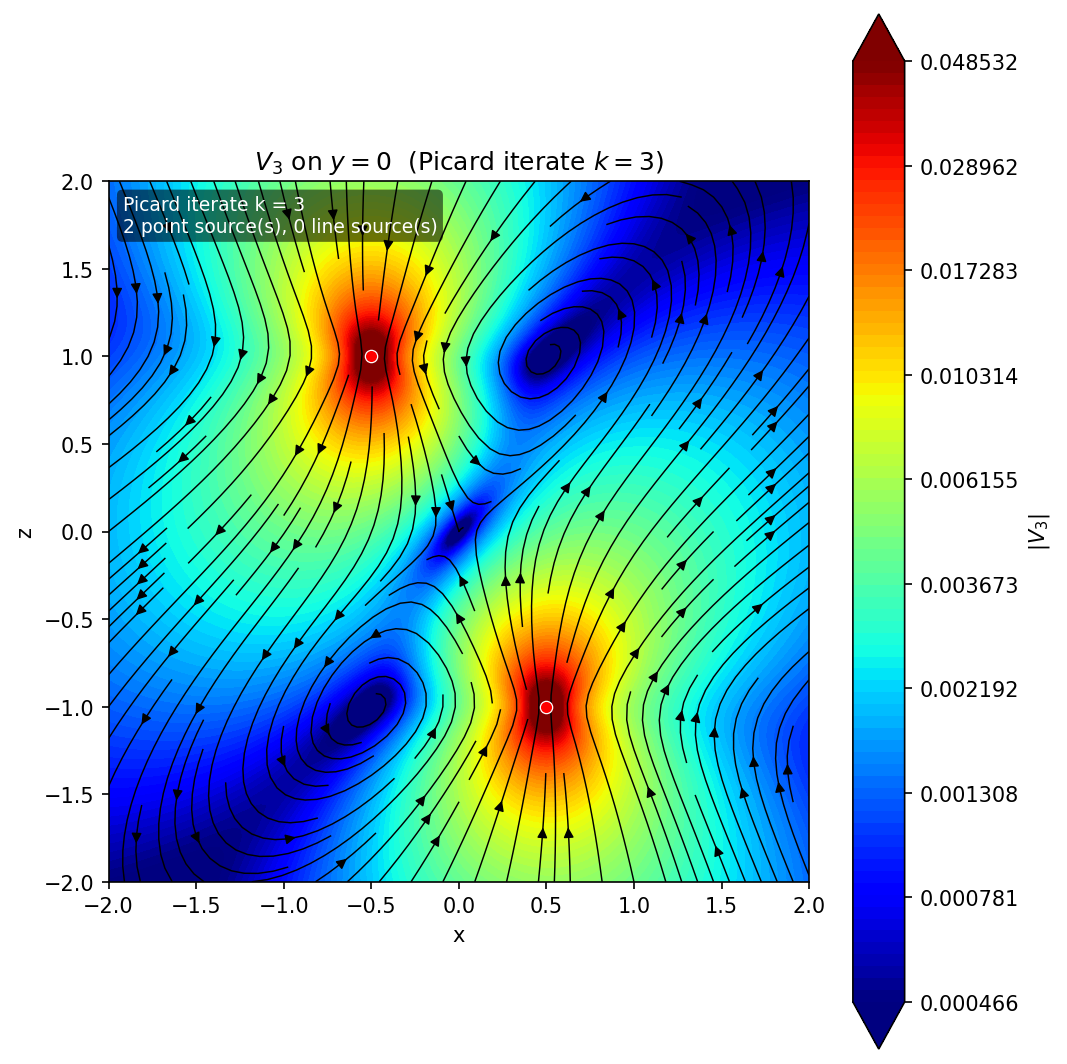}
        \caption{Swirls generated by two point-sources.}
    \end{subfigure}
    \begin{subfigure}{0.48\textwidth}
        \centering
        \includegraphics[width=\linewidth]{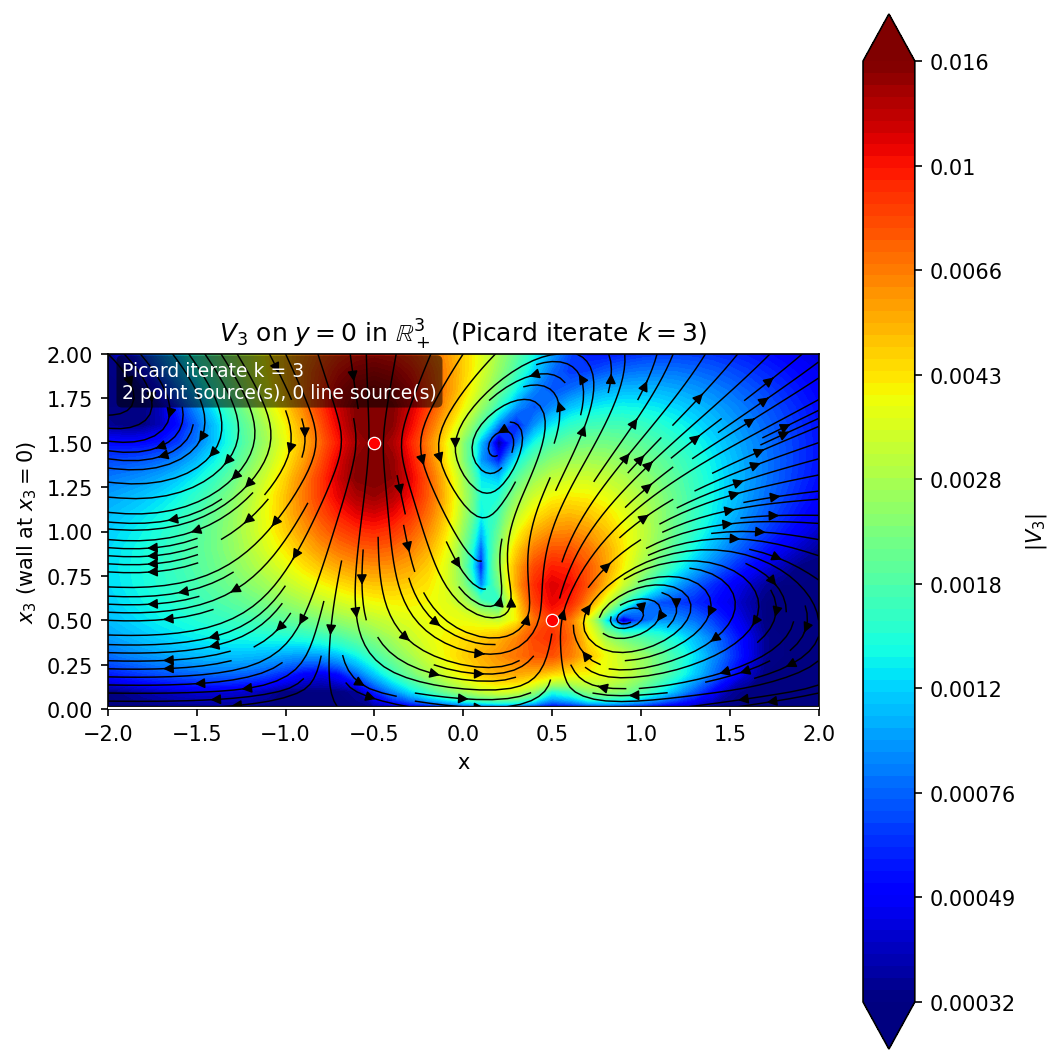}
        \caption{The same-point sources but near a boundary.}
    \end{subfigure}
    \caption{The point-sources in (A) generate multiple stable rotations. In (B), the boundary has destroyed one rotation but has generated another.}
    \label{fig::SwirlGen1}
\end{figure}

Figures \ref{fig::OpposingPoint}, \ref{fig::SwirlGen2} and \ref{fig::SwirlGen1} illustrate the effect of multiple point-sources on the flow as well as its interaction with a boundary. Note that in the small-regime, the interaction of a Stokeslet with the boundary \cite{StokesletNearBoundary} essentially describes these dynamics at the macroscale because the higher order contributions from the nonlinearity are subordinate to the leading linear behavior. In the figure we approximate the nonlinear contribution up to the third Picard iterate.  Our observation is that these features are inherited by the nonlinear problem for many geometries. This is also the case for the dynamics of linear combinations of Stokeslets. We observe that different arrangements of the point-sources can generate different structures such as vortexes. Our asymptotic stability results provide conditions under which these structures are asymptotically stable.

\begin{remark}
    \label{rem::ExistenceOfSolutionsWithFiniteNumberOfPointSingularitiesRemark}
    As will be evident in the following proof, this argument actually provides a construction of a solution to
    \begin{align*}
        -\Delta V+(V\cdot\nabla)V+\nabla q=f
        \qquad 
        \nabla\cdot V=0,
    \end{align*}
    where $f$ is such that there exists a solution $(V^{(0)},q^{(0)})$ to the Stokes equations forced by $f$ and satisfies
    \begin{align*}
        \|V^{(0)}\|_{L^{3,\infty}(\Omega)}\leq\epsilon,
    \end{align*}
    and $q^{(0)}\in L^{3/2,\infty}_{\operatorname{loc}}(\Omega).$
    We will make use of this fact later on when we prove Theorem \ref{thm::WholeSpaceExistenceOfSolutionsWithFiniteNumberOfLineSingularities}.

    Note that when $\Omega=\R^3$, Theorem \ref{thm::ExistenceOfSolutionsWithFiniteNumberOfPointSingularities} is almost contained in \cite{DeIf} which constructs solutions for small $(-3)$-homogeneous source terms that are localized in some sense.
    The proof in \cite{DeIf} constructs solutions in the $O(|x|^{-1})$ class and is therefore too restrictive to apply to Theorem \ref{thm::WholeSpaceExistenceOfSolutionsWithFiniteNumberOfLineSingularities}, which necessitates our work here.
    The relationship between \cite{DeIf} and this remark is that one can expect to generalize \cite{DeIf} by constructing solutions anytime the distributional force generates a solution to the Stokes equations which is small in $L^{3,\infty}(\R^3)$, regardless of its homogeneity.
\end{remark}

\begin{proof}
    By linearity of the Stokes equation, we know that $(V^{(0)},q^{(0)})$ satisfy
    \begin{align*}
        -\Delta V^{(0)}+\nabla q^{(0)}=f,
        \qquad
        \nabla\cdot V^{(0)}=0.
    \end{align*}
    We aim to solve 
    \begin{align*}
        -\Delta V+\nabla \cdot (V\otimes V)+\nabla q=f,
        \qquad 
        \nabla\cdot V=0,
    \end{align*}
    which is equivalent to
    \begin{align*}
        V=B_\Omega(V,V)+V^{(0)}.
    \end{align*}
    We will prove $\Phi(V):=B_\Omega(V,V)+V^{(0)}$ is a contraction on the ball $B(2\epsilon)$.

    Let $C$ be as in Lemma \ref{lem::GreenTensorConvolutionEstimate}. 
    Suppose $\epsilon<\frac{1}{4C}$. 
    If $\|V\|_{L^{3,\infty}(\Omega)}\leq 2\epsilon$ then
    \begin{align*}
        \|\Phi(V)\|_{L^{3,\infty}(\Omega)}
        \leq C\|V\|^2_{L^{3,\infty}(\Omega)}+\|V^{(0)}\|_{L^{3,\infty}(\Omega)}
        \leq 4C\epsilon^2+\epsilon
        \leq 2\epsilon.
    \end{align*}
    Now suppose that $U,V\in B(2\epsilon)$. 
    Then 
    \begin{align*}
        \|\Phi(U)-\Phi(V)\|_{L^{3,\infty}(\Omega)}&=\|B_\Omega(U,U)-B_\Omega(V,V)\|_{L^{3,\infty}(\Omega)}\\
                                                  &=\|B_\Omega(U,U-V)+B_\Omega(U-V,V)\|_{L^{3,\infty}(\Omega)}\\
                                                  &\leq C_\Omega(\|U\|_{L^{3,\infty}(\Omega)}+\|V\|_{L^{3,\infty}(\Omega)})\|U-V\|_{L^{3,\infty}(\Omega)},
    \end{align*}
    shows $\Phi$ is a contraction because $4C\epsilon<1$.
    Let $V_{k+1}=\Phi(V^{(k)})$.
    By the Banach fixed-point theorem there exists a $V\in L^{3,\infty}(\Omega)$ with $\|V\|_{L^{3,\infty}(\Omega)}<2\epsilon$.

    \detail{
        In components,
        \begin{align*}
            B_\Omega(V,V)_i(x)=\int_\Omega \partial_{y_k}G_{ij}(x,y)V_j(y)V_k(y)dy
        \end{align*}
        and hence
        \begin{align*}
            \langle B_\Omega(V,V),\Delta\varphi\rangle&=\int_\Omega\int_\Omega \partial_{y_k}G_{ij}(x,y)V_j(y)V_k(y)dy\Delta\varphi_i(x)dx\\
            &=\int_\Omega V_j(y)V_k(y)\int_\Omega \partial_{y_k}G_{ij}(x,y)\Delta\varphi_i(x)dxdy.
        \end{align*}
        Here we used the Fubini-Tonelli theorem which requires the integrand to be measurable in the product space and the iterated integral of its absolute value to be finite.
        Measurability is clear, to check  
        \begin{align*}
            \int_\Omega \int_\Omega |V_j(y)V_k(y)\partial_{y_k}G_{ij}(x,y)\Delta\varphi_i(x)|dxdy<\infty.
        \end{align*}
        which is true by the Lemma \eqref{lem::GreenTensorConvolutionEstimate} and the uniform bound $\|\Delta\varphi_i(x)\|_{L^\infty(\Omega)}<\infty$.}

    \detail{
        Now, we need to show
        \begin{align*}
            \int_\Omega \partial_{y_k}G_{ij}(x,y)\Delta\varphi_i(x)dx
            =\partial_{y_k}\int_\Omega G_{ij}(x,y)\Delta\varphi_i(x)dx.
        \end{align*}
        Fix $y_0\in\Omega$. 
        Fix $\delta_0<\mathrm{dist}(y_0,\partial\Omega)$.
        For a.e. $x\in\Omega$ the function $h\mapsto G_{ij}(x,y_0+he_k)\Delta\varphi_i(x)$ is continuous on $[-\delta_0,\delta_0]$ and $C^1$ on $(-\delta_0,\delta_0)$.
        So we can apply the fundamental theorem of calculus to obtain the bound
        \begin{align*}
            \left|\frac{G_{ij}(x,y_0+he_k)-G_{ij}(x,y_0)}{h}\right|
            =\left|\int_0^1 \partial_{y_k}G_{ij}(x,y_0+she_k)ds\right|\leq \int_0^1\frac{C}{|x-y_0-she_k|^2}ds
        \end{align*}
        For each $h\in (-\delta_0,\delta_0)$ let
        \begin{align*}
            f_h(x):=\frac{G_{ij}(x,y_0+he_k)-G_{ij}(x,y_0)}{h}\Delta\varphi_i(x).
        \end{align*}
        Let $K=\mathrm{supp}(\varphi)$.
        We now want to show $\cF=\{f_h:0<h<\delta_0\}$ is uniformly integrable in $|x|$.
        That is (see Royden page 93), we need to show for each $\epsilon>0$ there is a $\delta>0$ such that for each $f_h\in\cF$ if $A\subseteq K$ is measurable with $|A|<\delta$, then $\int_A |f|<\epsilon$.
        Fix $\epsilon>0$. 
        Let $\delta=C\|\Delta\varphi\|_{L^\infty(K)} 4\pi\left(\frac{3\epsilon}{4\pi}\right)^{1/3}$ and suppose $|A|<\delta$.
        Then
        \begin{align*}
            \int_A |f_h|dx
            \leq C\|\Delta\varphi\|_{L^\infty(K)} \int_A \int_0^1\frac{1}{|x-y_0-she_k|^2}dsdx
            \leq C\|\Delta\varphi\|_{L^\infty(K)}\sup_{y\in\Omega} \int_A \frac{1}{|x-y|^2}dx.
        \end{align*}
        Let $f=\chi_A$ denote the indicator function for $A$ and $g(x)=\frac{1}{|x-y|^2}$. 
        Then, by Wikipedia, the symmetric decreasing rearrangement of $f$, denoted $f^*$, is the indicator function for the ball $B_R(0)$ with $R=\left(\frac{3|A|}{4\pi}\right)^{1/3}$. 
        Furthermore, the symmetric decreasing arrangement of $g$ is defined by 
        \begin{align*}
            g^*(x)=\int_0^\infty \chi_{\{z:f(z)>t\}^*}(x)dt.
        \end{align*}
        To compute this, note that $\{z:|z-y|^{-2}>t\}=B_{t^{-1/2}}(y)$. 
        So $\{z:f(z)>t\}^*=B_{t^{-1/2}}(0)$ and hence, 
        \begin{align*} |x|<t^{-1/2} t<|x|^{-2}
            g^*(x)
            =\int_0^\infty \chi_{\{B_{t^{-1/2}(0)}}(x)dt
            =\int_0^{|x|^{-2}} dt
            =|x|^{-2}.
        \end{align*}
        By the Hardy-Littlewood inequality,
        \begin{align*}
            \sup_{y\in\Omega} \int_A \frac{1}{|x-y|^2}dx
            \leq \int_{B_R} \frac{1}{|x|^2}dx
            =4\pi\left(\frac{3|A|}{4\pi}\right)^{1/3}.
        \end{align*}
        Thus
        \begin{align*}
            \int_A |f_h|dx
            < C\|\Delta\varphi\|_{L^\infty(K)} 4\pi\left(\frac{3\epsilon}{4\pi}\right)^{1/3}
            =\delta
        \end{align*}
        By the Vitali convergence theorem
        \begin{align*}
            \partial_{y_k}\int_\Omega G_{ij}(x,y_0)\Delta\varphi_i(x)dx
            =\lim_{h\to 0}\int_\Omega f_h(x)dx
            =\int_\Omega \partial_{y_k}G_{ij}(x,y_0)\Delta\varphi_i(x)dx.
        \end{align*}
        Since $y_0$ was arbitrary, this verifies
        \begin{align*}
            \int_\Omega \partial_{y_k}G_{ij}(x,y)\Delta\varphi_i(x)dx
            =\partial_{y_k}\int_\Omega G_{ij}(x,y)\Delta\varphi_i(x)dx.
        \end{align*}
        }

    It remains to construct the pressure.
    Fix $\varphi\in C_{c,\sigma}^\infty(\Omega;\R^3)$. 
    Denote the components of $G_\Omega(x,y)$ by $G_{ij}(x,y)$.
    We first verify the identity
    \begin{equation}
    \label{eq::IdentityNeededForDistributionalPressureCalculation}
        \int_\Omega \partial_{y_k}G_{ij}(x,y)\Delta\varphi_i(x)\,dx
        =\partial_{y_k}\int_\Omega G_{ij}(x,y)\Delta\varphi_i(x)\,dx.
    \end{equation}
    Fix $y_0\in\Omega$. 
    Fix $\delta_0<\mathrm{dist}(y_0,\partial\Omega)$.
    For a.e. $x$ and $h\in(-\delta_0,\delta_0)$, the fundamental theorem of calculus gives 
    \begin{align*}
        \left|\frac{G_{ij}(x,y_0+he_k)-G_{ij}(x,y_0)}{h}\right|
        =\left|\int_0^1 \partial_{y_k}G_{ij}(x,y_0+she_k)\,ds\right|\leq \int_0^1\frac{C}{|x-y_0-she_k|^2}\,ds.
    \end{align*}
     For each $h\in (-\delta_0,\delta_0)$ let
    \begin{align*}
        f_h(x):=\frac{G_{ij}(x,y_0+he_k)-G_{ij}(x,y_0)}{h}\Delta\varphi_i(x),
    \end{align*}
    and set $\cF=\{f_h:0<h<\delta_0\}$.
    \begin{align*}
        \int_A |f_h|\,dx
        &\leq C\|\Delta\varphi\|_{L^\infty(K)} \int_A \int_0^1\frac{1}{|x-y_0-she_k|^2}\,ds\,dx\\
        &\leq C\|\Delta\varphi\|_{L^\infty(K)}\sup_{y\in\Omega} \int_A \frac{1}{|x-y|^2}\,dx.
    \end{align*}
    Let $f=\chi_A$ denote the indicator function for $A$ and $g(x)=\frac{1}{|x-y|^2}$. 
    Then, the symmetric decreasing rearrangements, of $f$ and $g$ are, respectively $f^*=\chi_{B(0,R)}$ where $R=(3|A|/4\pi)^{1/3}$ and $g^*=|x|^{-2}$.
    The Hardy-Littlewood inequality implies that 
    \begin{align*}
        \int_A \frac 1 {|x-y|^2}\,dx \leq \int_{B(0,R)} \frac 1 {|x|^2}\,dx\leq 4\pi R. 
    \end{align*}
    Hence
    \begin{align*}
        \int_A |f_h|\,dx
        &\leq  C\|\Delta\varphi\|_{L^\infty(K)} 4\pi\left(\frac{3|A|}{4\pi}\right)^{1/3}\to 0.
    \end{align*}
    as $|A|\to 0$.   
    So the family $\cF$ is uniformly integrable in $x$.
    By the Vitali convergence theorem \cite[Sec. 4.6]{Royden}
    \begin{align*}
        \partial_{y_k}\int_\Omega G_{ij}(x,y_0)\Delta\varphi_i(x)\,dx
        =\lim_{h\to 0}\int_\Omega f_h(x)\,dx
        =\int_\Omega \partial_{y_k}G_{ij}(x,y_0)\Delta\varphi_i(x)\,dx.
    \end{align*}
    Since $y_0$ was arbitrary, this verifies \eqref{eq::IdentityNeededForDistributionalPressureCalculation}.

    Now, for each $\varphi\in C_{c,\sigma}^\infty(\Omega;\R^3)$,
    \begin{align*}
        \langle B_\Omega(V,V),\Delta\varphi\rangle &=\int_\Omega \int_\Omega \nabla_y G_\Omega(x,y):(V\otimes V)(y)\,dy\cdot\Delta\varphi(x)\,dx\\
        &=\int_\Omega (V\otimes V)(y):\nabla_y \int_\Omega G_\Omega(x,y)^\top(-\Delta\varphi(x))\,dx\,dy\\
        &=\int_\Omega (V\otimes V)(y):\nabla_y \varphi(y)\, dy=-\langle \nabla\cdot(V\otimes V),\varphi\rangle.
    \end{align*}
    So
    \begin{align*}
        \langle f+\Delta V-\nabla\cdot(V\otimes V),\varphi\rangle&=\langle f,\varphi\rangle+\langle V,\Delta\varphi\rangle+\langle V\otimes V,\nabla\varphi\rangle\\
        &=\langle f,\varphi\rangle +\langle V_0,\Delta\varphi\rangle+\langle B_\Omega(V,V),\Delta\varphi\rangle+\langle V\otimes V,\nabla\varphi\rangle=0,
    \end{align*}
    for all $\varphi\in C_{c,\sigma}^\infty(\Omega;\R^3)$. 
    By \cite[Theorem 17']{DeRham} there exists a distribution $q\in (C_{c,\sigma}^\infty(\Omega;\R^3))'$ satisfying 
    \begin{equation}
        \label{eq::PressureEquation}
        \nabla q=f+\Delta V-\nabla\cdot(V\otimes V),
    \end{equation}
    as distributions.

    It remains to show $q\in L^{3/2,\infty}_{\mathrm{loc}}(\Omega)$. 
    As distributions, $(q-q^{(0)})$ satisfy
    \begin{align*}
        -\Delta (q-q^{(0)})=\partial_i\partial_j(V_iV_j),
        \qquad 
        \text{in }\;(C_c^\infty(\R^3))'.
    \end{align*}
    By Lorentz-Holder, $V_iV_j\in L^{3/2,\infty}(\Omega)$. 
    Let $\Phi_{ij}\in L^{3/2,\infty}(\R^3)$ be the extension by zero of $V_iV_j$.
    Define the tempered distribution $g=\sum_{i,j=1}^3 R_iR_j\Phi_{ij}$ where the Riesz transform $R_i$ are interpreted by duality.
    Then $g$ satisfies 
    \begin{align*}
        -\Delta g=\partial_i\partial_j\Phi_{ij}, 
        \qquad 
        \text{in }\;(C_c^\infty(\R^3))'
    \end{align*}
    and $\|g\|_{L^{3/2,\infty}(\Omega)}$ by boundedness of Riesz transforms on $L^{3/2,\infty}(\Omega)$.
    So $h:=(q-q^{(0)})-g$ satisfies 
    \begin{align*}
        -\Delta h=0, 
        \qquad 
        \text{in }\;(C_c^\infty(\Omega))'.
    \end{align*} 
    By the generalization of Weyl's lemma to distributions $h$ is a smooth function. So $h$ is locally bounded and more importantly, $h\in L^{3/2,\infty}(\Omega)_{\mathrm{loc}}$. 
    Therefore $q=h+g+q^{(0)}\in L^{3/2,\infty}_{\mathrm{loc}}(\Omega)$.
\end{proof}

\subsection{Li-Li-Yan Type II solutions}
\label{subsec::Examples::LiLiYanTypeII}

In \cite{LiLiYanII}, Li Li, Yan Yan Li and Xukai Yan classify solutions of \eqref{eq::RicattiODE} in the interval $(-1,1)$ which correspond to $(-1)$-homogeneous, axisymmetric, no-swirl solutions to \eqref{sns} which are smooth in $\R^3\setminus\{x_h=0\}$.
Compared to \cite{LiLiYanI}, which studies solutions that are singular on a half-axis, the solutions in \cite{LiLiYanII} may be $o(|x_h|^{-1})$. 
These solutions, which are called Type II, are asymptotically stable under a smallness condition \cite{LiYan}.
In this section, we make a remark about these solutions and different approaches to proving they are asymptotically stable.
We then identify a subset of solutions which satisfy a physically meaningful boundary condition when viewed as solutions in $\R^3_+$ or $\R^3_-$.

The parameter space in \cite{LiLiYanII} is four dimensional, three of which correspond to the triple $c=(c_1,c_2,c_3)$ which satisfies 
\begin{align*}
    c_1\geq -1,
    \qquad 
    c_2\geq -1,
    \qquad 
    c_3\geq \bar{c_3}(c_1,c_2):=-\frac{1}{2}\left(\sqrt{1+c_1}+\sqrt{1+c_2}\right)\left(\sqrt{1+c_1}+\sqrt{1+c_2}+2\right),
\end{align*}
and the fourth a closed interval $I(c)$ depending on $c$. Define
\begin{align*}
    \gamma^+(c):=v_\theta^+(c)(0),
    \qquad 
    \gamma^-(c):=v_\theta^-(c)(0),
\end{align*}
where $v_\theta^+$ and $v_\theta^-$ are the unique solutions of \eqref{eq::RicattiODE} attaining the largest and smallest values at $x=0$ respectively.
We have that $I(c):=[\gamma^-(c),\gamma^+(c)]$.

In \cite{LiYan}, Yan Yan Li and Xukai Yan prove asymptotic stability for the solutions $u^{c,\gamma}$ with 
\begin{align*}
    (c,\gamma)\in M:=\{(c,\gamma):c_1=c_2=0,c_3>-4,\gamma^-(c)<\gamma<\gamma^+(c)\},
\end{align*}
under a smallness condition.
For $(c,\gamma)\in M$, the associated solution $V^{c,\gamma}$ to \eqref{sns} satisfies the point-wise bound
\begin{align*}
    |V^{c,\gamma}(x)|\leq \frac{C(|c|,|\gamma|)}{\sqrt{|x||x_h|}},
\end{align*}
which follows from \cite[Corollary 2.1]{LiYan}.
To be more precise, \cite[Corollary 2.1]{LiYan} directly implies 
\begin{align*}
    |V^{c,\gamma}(x)|\leq \frac{C}{|x|}\left|\ln\left(\frac{|x_h|}{|x|}\right)\right|.
\end{align*}
Because $-1/(\sqrt{a})\leq\ln a\leq 1/(\sqrt{a})$ for $a\in (0,1)$ we infer that 
\begin{align*}
    \frac{1}{|x|}\left|\ln\left(\frac{|x_h|}{|x|}\right)\right|\leq \frac{1}{|x|}\frac{|x|^{1/2}}{|x_h|^{1/2}}.
\end{align*}
Asymptotic stability was proven directly in \cite{LiYan} using a new, anisotropic Caffarelli-Kohn-Nirenberg inequality which is of independent interest, see also \cite{LiYanII}.
We note that asymptotic stability also follows directly as a consequence of \cite{KarchPilSchonbek} using the fact that $|x|^{-1/2}|x_h|^{-1/2}\in L^{3,\infty}(\R^3)$. 
We were surprised to realize this, as well as the fact that slightly more singular profiles also belong to $L^{3,\infty}(\R^3)$, which is the content of the following lemma.

\begin{lemma}
    \label{lem::TypeIISolutionsAreInWeakL3}
    We have that $\frac{1}{|x|^a|x_h|^{1-a}}\in L^{3,\infty}(\R^3)$ for $a\in(1/3,\infty)$.
    Consequently, $V^{c,\gamma}\in L^{3,\infty}(\R^3)$ when $(c,\gamma)\in M$.
\end{lemma}

\begin{proof}
    In cylindrical coordinates $(\rho,\phi,z)$, $|x|=\sqrt{\rho^2+z^2}$ and $|x_h|=\rho$. 
    If $\rho>0$ then 
    \begin{align*}
        (\rho^2+z^2)^{-a/2}\rho^{a-1}>\lambda,
    \end{align*}
    is equivalent to\detail{ $(\rho^2+z^2)^{-a/2}\rho^{a-1}>\lambda$ iff $(\rho^2+z^2)^{-a/2}>\lambda \rho^{1-a}$ iff $\rho^2+z^2<(\lambda \rho^{1-a})^{-2/a}$ iff $|z|<\sqrt{(\lambda \rho^{1-a})^{-2/a}-\rho^2}$}
    \begin{align*}
        |z|<\sqrt{\lambda^{-2/a}\rho^{-2/a+2}-\rho^2}=:A(\rho).
    \end{align*}
    But $A(\rho)$ is a real number only when $\rho<\lambda^{-1}$.\detail{ $A(\rho)$ is real iff $\lambda^{-2/a}\rho^{-2/a+2}-\rho^2>0$ iff $\lambda^{-2/a}>\rho^{2/a}$ iff $\lambda^{-1}>\rho$.}
    Thus
    \begin{align*}
        |\{x\in\R^3:|x|^{-a}|x_h|^{a-1}>\lambda\}|
        &=2\int_0^{\lambda^{-1}}\int_0^{2\pi}\int_0^{A(\rho)}\rho dzd\phi d\rho\\
        &=4\pi\lambda^{-3}\int_0^1 t^{2-1/a}\sqrt{1-t^{2/a}}dt
        =:I\lambda^{-3}.
    \end{align*}
    The integral $I$ is finite exactly when $a>1/3$. 
    Thus, $\||x|^{-a}|x_h|^{a-1}\|_{L^{3,\infty}(\R^3)}=I^{1/3}$.
    \detail{
    The missing details in the above integral computation are:
    \begin{align*}
        &=4\pi\int_0^{\lambda^{-1}}\rho\sqrt{\lambda^{-2/a}\rho^{-2/a+2}-\rho^2} d\rho\\
        &=4\pi\int_0^1 \lambda^{-1}t\sqrt{\lambda^{-2/a}(t\lambda^{-1})^{-2/a+2}-(t\lambda^{-1})^2} \lambda^{-1}dt\\
        &=4\pi\lambda^{-2}\int_0^1 t\sqrt{\lambda^{-2}t^{-2/a+2}-t^2\lambda^{-2}}dt
    \end{align*}
    }

    It is perhaps interesting to give a second proof of this.
    It is implied by \cite[Lemma 3.1]{BT1} that if $u_0$ is $(-1)$-homogeneous, then $u_0\in L^{3,\infty}(\R^3)\iff u_0\in L^3_{\mathrm{loc}}(\R^3\setminus\{0\})$. 
    Let $K\Subset\R^3\setminus\{0\}$. 
    Let $x_K$ satisfy $|x_K|=\inf\{|x|:x\in K\}$.
    Then
    \begin{align*}
        \int_K \frac{1}{|x|^{3a}|x_h|^{3(1-a)}}\, dx
        \leq |x_K|^{-3a}\int_K \frac{1}{|x_h|^{3(1-a)}}\, dx.
    \end{align*}
    Viewing the above as an iterated integral and noting that, when $a>1/3$, we have $3(1-a)<2$ and, therefore, $\frac{1}{|x_h|^{3(1-a)}}\in L^2_{\mathrm{loc}}(\R^2)$. 
    This proves $u_0\in L^3_{\mathrm{loc}}(\R^3\setminus\{0\})$.
\end{proof}

We can also find solutions which satisfy a no-penetration boundary condition when restricted to $\R^3_+$ or $\R^3_-$.
For each $\gamma\in [\gamma^-(c),\gamma^+(c)]$ the corresponding solution $v_\theta^{c,\gamma}$ satisfies $v_\theta^{c,\gamma}(0)=\gamma$.
So if we want to find solutions in $M$ which restrict to solutions in the half-space with a no-penetration boundary condition, it suffices to show that $\gamma^-(c)\leq 0\leq \gamma^+(c)$ for each $c$ that shows up in $M$.

\begin{lemma}
    \label{lem::LemmaForHalfSpaceTypeII}
    Suppose $c_1=c_2>-1$ and $c_3>\bar{c}_3(c_1,c_2)$. Then $\gamma^-(c)=-\gamma^+(c)$.
\end{lemma}

\begin{proof}
    First note that if $c_1=c_2$, then \eqref{eq::RicattiODE} is invariant under the involution $v_\theta(x)\mapsto -v_\theta(-x)$. 
    So $U(x)=-v_\theta^+(c)(-x)$ solves \eqref{eq::RicattiODE} with constants $c$.
    For contradiction, assume $U\neq v_\theta^-(c)$.
    By \cite[Theorem 1.3]{LiLiYanII},
    \begin{align*}
        U(1)=-v_\theta^+(c)(-1)=-2-2\sqrt{1+c_1}=-2-2\sqrt{1+c_2}=v_\theta^-(c)(1).
    \end{align*}
    But, \cite[Theorem 1.3]{LiLiYanII} says $U\neq v_\theta^-(c)$ implies $U(1)=-2+2\sqrt{1+c_2}$.
    This is a contradiction since no $c_2>-1$ satisfies $-2+2\sqrt{1+c_2}=-2-2\sqrt{1+c_2}$.
\end{proof}

So for each $c_3>-4$, we have the pair $(c,0)\in M$.
This implies we have a curve of solutions $V^{c,\gamma}$ in the parameter space $M$ which satisfy a no-penetration boundary condition at $x_3=0$.
This represents two, one-parameter families of solutions which can be made small in $L^{3,\infty}(\R^3_+)$ which are associated to the restrictions of $V^{c,\gamma}$ to the upper and lower half-space respectively.

The preceding work justifies the following proposition.

\begin{proposition}
    \label{prop::HalfSpaceTypeII}
    There exist solutions to \eqref{sns} on $\R^3_+\setminus\{x_h=0\}$ satisfying no-penetration boundary conditions and belonging to $L^{3,\infty}(\R^3_+)$. 
    The set of solutions contains elements with arbitrarily small norm.
\end{proposition}

Under a smallness assumption, our asymptotic stability result Theorem \ref{mainthm1} implies these solutions are stable under perturbations which satisfy the same boundary condition and which possess the same interior singularity.
This singularity can be formulated as a source term comprised of a line of Dirac functions \cite{LiYan}.
Because the boundary vector field of the solutions in Proposition \ref{prop::HalfSpaceTypeII} is purely radial and homogeneous, it has a physical interpretation that is elaborated on in the next section which considers solutions with the same boundary condition but no interior singularity.

\subsection{Squire's surface discharge flows}
\label{subsec::Examples::SquireSurfaceDischargeFlows}

Squire discovered a one-parameter family of solutions which converge in $\R^3_+$ with no-penetration, radial boundary conditions \cite{Squire2}. 
Wang \cite{Wang} later proposes these solutions as a model for the "frequent accidental spillings of oil on the surface of water."
He goes on to describe the problem, stating that,
\begin{quote}
...due to the balance of gravity and surface tension, the oil spreads with more or less constant thickness covering a large area on the water surface. 
Since the oil is much more viscous than water, we assume the water has little dynamic effect on the layer of oil. 
Our problem thus considers only the effect of the spreading oil on the water underneath. \cite{Wang} 
\end{quote}
Based on this description we refer to these as \textbf{surface discharge flows}.
These solutions were obtained using the same reduction that led to Landau solutions but where the constants $c_i$ are chosen to ensure the solution converges in $\R^3_+$ and is purely tangential at the boundary.
If one pushes this further and attempts to satisfy a no-slip boundary condition, then a singularity necessarily develops along the axis of symmetry.
The necessity of this is proven in \cite[Appendix]{KMT}.

In this section we will show that every sufficiently regular, $(-1)$-homogeneous, axisymmetric, no-swirl solution to \eqref{sns} in the upper half space $\R^3_+$ (or other cone-domains) with radial, no-penetration boundary conditions is a surface discharge flow.
These solutions have purely radial, $(-1)$-homogeneous boundary conditions and belong to $L^{3,\infty}(\Omega)$.
Under a smallness hypothesis, it is reasonable to expect they are unique.
Our conclusion here implies that if a non-uniqueness generating bifurcation occurs as a smallness parameter is increased, then it must invoke some curl creation or symmetry breaking, either homogeneity or axisymmetry.
We will also see that the surface discharge flows trace out a path in a two-parameter space of solutions discovered by Li, Li, and Yan \cite{LiLiYanI}, thereby identifying a physically meaningful scenario within the framework of \cite{LiLiYanI} and making a connection between the mathematical literature and the physics literature.
Let us mention that in contrast to the Type II solutions of \cite{LiLiYanII} we discuss in Section \ref{subsec::Examples::LiLiYanTypeII}, the solutions in \cite{LiLiYanI} are singular on a half-axis and their singularities are more severe.

For generality we consider solutions on cones around the $x_3$ axis instead of just $\R^3_+$, which is the case discussed in \cite{Wang}.
For each angle $\theta_0\in (0,\pi)$ let $\Omega_{\theta_0}$ be the cone
\begin{align*}
    \Omega_{\theta_0}=\{(r,\theta,\phi)\in\R^3:\theta<\theta_0\}.
\end{align*}
The domain $\Omega_{\theta_0}$ is a cone-like domain and it makes sense to talk about $(-1)$-homogeneous solutions in this domain with no-penetration boundary conditions.

As seen in \cite{LiLiYanI} and \cite{Wang}, there is a one-to-one correspondence between smooth, $(-1)$-homogeneous, axisymmetric, no-swirl solutions to \eqref{sns} in $\Omega_{\theta_0}$ and solutions to 
\begin{equation}
    \label{eq::RicattiODECones}
    (1-x^2)v_\theta'+2xv_\theta+\frac{1}{2}v_\theta^2=\beta(1-x)^2,
\end{equation}
in $[x_0,1]$ with zero boundary conditions at $x_0=\cos(\theta_0)$ and satisfy $v_\theta=O(1-x^2)$ near $1$.
We can recover $V=(V_r,V_\theta,V_\phi)$ from $v_\theta$ by the change of coordinates
\begin{equation}
    \label{eq::ThetaXCoordinateChange}
    x=\cos(\theta),
    \qquad 
    V_r=\frac{v_\theta'(x)}{r},
    \qquad 
    V_\theta=\frac{v_\theta(x)}{r\sin(\theta)},
    \qquad 
    V_\phi\equiv 0,
\end{equation}
where $'$ denotes differentiation in $x$.
Solutions to \eqref{eq::RicattiODECones} together with their maximal interval of existence were classified in \cite[Theorem 3.1]{LiLiYanI}, which we presently recall.

\begin{theorem}
    \cite[Theorem 3.1]{LiLiYanI}. Let $\delta^*$ be as in \cite[Eq. 3.4]{LiLiYanI}. 
    For every $(\beta,\gamma)\in\R^2$, there exists a unique $v_\theta:=v_\theta(\beta,\gamma;\cdot)\in C^\infty(\delta^*,1)$ satisfying \eqref{eq::RicattiODECones} in $(\delta^*,1)$ and 
    \begin{align*}
        \lim_{x\to 1^-}v_\theta'(x)=\gamma.
    \end{align*}
    The interval $(\delta^*,1)$ is the maximal interval of existence for $v_\theta$ and in particular, 
    \begin{align*}
        \lim_{x\to\delta^{*+}}|v_\theta(x)|=\infty,
        \qquad 
        \text{if }\;\delta^*>-1.
    \end{align*}
    Moreover, $v_\theta$ is explicitly given by
    \begin{equation}
        \label{eq::LiLiYan3.1ThetaComponentFormula}
        v_\theta(x)=
        \begin{cases}
            (1-x)\left(1-c-\frac{2c(\gamma+1-c)}{(\gamma+1+c)\left(\frac{1+x}{2}\right)^{-c}-\gamma-1+c}\right), & \beta>-\frac{1}{2},\\
            (1-x)\left(1+\frac{2(\gamma+1)}{(\gamma+1)\ln\left(\frac{1+x}{2}\right)-2}\right), & \beta=-\frac{1}{2},\\
            (1-x)\left(1+\frac{c\left(c\tan\left(\frac{b(x)}{2}\right)+\gamma+1\right)}{(\gamma+1)\tan\left(\frac{b(x)}{2}\right)-c}\right), & \beta<-\frac{1}{2},
        \end{cases}
    \end{equation}
    where $c:=\sqrt{|1+2\beta|}$, and $b(x):=c\ln\frac{1+x}{2}$.
\end{theorem}

Roughly, the following proposition asserts that the only axisymmetric, no-swirl, $(-1)$-homogeneous solutions to \eqref{sns} on $\R^3_+$ (or cones) with no-penetration boundary conditions are surface discharge flows (or generalizations of these solutions to cones).
The proof is essentially contained in \cite{LiLiYanI,Squire2,Wang}, although our observation connecting these is new.
The parameter space and the position of surface discharge flows are illustrated in Figure \ref{fig}.

\begin{proposition}
    \label{prop::SquireSolutionsClassification}
    Fix $\theta_0\in (0,\pi)$. 
    There exists $\beta_{\mathrm{min}}(\theta_0)<0$ such that every $(-1)$-homogeneous, axisymmetric, no-swirl solution to \eqref{sns} which belongs to $C^2(\overline{\Omega_{\theta_0}})\times C^1(\overline{\Omega_{\theta_0}})$ and which satisfies a no-penetration boundary condition must agree with a vector field $V^{\beta,\theta_0}$ for some $\beta\in(\beta_{min}(\theta_0),\infty)$ where
    \begin{align}
    \label{eq::WangsSolutionsTheta}
            V_\theta^{\beta,\theta_0}
            &=\begin{cases}
            \frac{1-\cos(\theta)}{r\sin(\theta)}\cdot\frac{2\beta(R(\theta)^c-1)}{(c+1)+R(\theta)^c (c-1)} & \beta>-\frac{1}{2}\\
            \frac{1-\cos(\theta)}{r\sin(\theta)}\cdot\left[1+\frac{2}{\ln(R(\theta))-2}\right] & \beta=-\frac{1}{2}\\
            \frac{1-\cos(\theta)}{r\sin(\theta)}\cdot \frac{-2\beta}{1-c\cot\left(\frac{c}{2}\ln(R(\theta))\right)} & \beta<-\frac{1}{2}
            \end{cases}\\
            V_r^{\beta,\theta_0}&=-\frac{dV^\beta_\theta}{d\theta}-\cot(\theta)V^\beta_\theta
     \end{align}
     with $R(\theta)=\frac{1+\cos(\theta)}{1+\cos(\theta_0)}$ and $c=\sqrt{|1+2\beta|}$. We also have that $V_r^{\beta,\theta_0}|_{\partial\Omega_{\theta_0}}=\frac{\beta}{r}\cdot\frac{1-\cos(\theta_0)}{1+\cos(\theta_0)}$
\end{proposition}

Under a smallness assumption, our asymptotic stability result Theorem \ref{mainthm1} implies these solutions are stable under perturbations which satisfy the same boundary condition.

\begin{proof}
    Fix $\theta_0\in (0,\pi)$.
    Let $x_0:=\cos(\theta_0)$.
    Then $x_0\in (-1,1)$.
    If $V$ has no-penetration boundary conditions on $\Omega_{\theta_0}$, then, due to uniqueness, it must correspond to a solution described in \cite[Theorem 3.1]{LiLiYanI} with $v_\theta(x_0)=0$.
    So it suffices to compute what $\gamma(\beta)$ must be for each $\beta\in\R$.
    Suppose $\beta>-\frac{1}{2}$. 
    Then $v_\theta(x_0)=0$ is equivalent to 
    \begin{align*}
        1-c=\frac{2c(\gamma+1-c)}{(\gamma+1+c)a^{-c}-\gamma-1+c}
    \end{align*}
    where $a=\frac{1+x_0}{2}$. 
    So
    \begin{align*}
        \gamma[(1-c)(a^{-c}-1)-2c]=2c(1-c)+2\beta(a^{-c}+1)+2(1-c)=2\beta(a^{-c}-1)
    \end{align*}
    and hence
    \begin{align*}
        \gamma 
        =\frac{2\beta(a^{-c}-1)}{(1-c)(a^{-c}-1)-2c}
        =\frac{2\beta(1-R(0)^c)}{R(0)^c(c-1)+c+1}.
    \end{align*}
    If $\beta=-\frac{1}{2}$ then, $v_\theta(x_0)=0$ if and only if 
    \begin{align*}
        \gamma(-\ln(a)-2)=\ln(a)
    \end{align*}
    and hence 
    \begin{align*}
        \gamma(\beta)=\frac{\ln(R(0))}{2-\ln(R(0))}.
    \end{align*}
    For the final case suppose $\beta<-\frac{1}{2}$.
    So
    \begin{align*}
        c-(\gamma+1)d=c(cd+\gamma+1)
    \end{align*}
    where $d=\tan\left(\frac{c\ln(a)}{2}\right)$.
    Since $c^2=-1-2\beta$ for $\beta<-\frac{1}{2}$,
    \begin{align*}
        \gamma(-d-c)=(c^2+1)d=-2\beta d
    \end{align*}
    and hence 
    \begin{align*}
        \gamma=\frac{2\beta}{1+cd^{-1}}=\frac{2\beta}{1-c\cot(\frac{c}{2}\ln(R(0)))},
    \end{align*}
    because $\cot(\theta)$ is an odd function.
    So the solutions described in \cite[Theorem 3.1]{LiLiYanI} which satisfy a zero boundary condtion at $x_0$ are the one-parameter subfamily given by $(\beta,\gamma(\beta))$ where
    \begin{align*}
        \gamma(\beta)
        =\begin{cases}
        \frac{2\beta(1-R(0)^c)}{R(0)^c(c-1)+c+1}       & \beta>-\frac{1}{2}\\
        \frac{\ln(R(0)}{2-\ln(R(0))}                  & \beta=-\frac{1}{2}\\
        \frac{2\beta}{1-c\cot(\frac{c}{2}\ln(R(0)))} & \beta<-\frac{1}{2}
        \end{cases}
    \end{align*}
    and $\beta\in(\beta_{\mathrm{min}},\infty)$ where $\beta_{\mathrm{min}}<0$ is the largest real number such that $\lim_{\beta\to\beta^+_{\mathrm{min}}}\gamma(\beta)$ diverges.
    
    Using a symbolic calculator, substituting $\gamma(\beta)$ for $\gamma$ into equation \eqref{eq::LiLiYan3.1ThetaComponentFormula} yields
    \begin{align*}
        v_\theta(x)
        =\begin{cases}
            (1-x)\cdot \frac{2\beta(r(x)^c-1)}{(c+1)+(c-1)r(x)^c}               & \beta>-\frac{1}{2}\\
            (1-x)\left(1+\frac{2}{\ln(r(x))-2}\right)                           & \beta=-\frac{1}{2}\\
            (1-x)\cdot \frac{-2\beta}{1-c\cot\left(\frac{c}{2}\ln(r(x))\right)} & \beta<-\frac{1}{2}
        \end{cases}
    \end{align*}
    where $r(x)=\frac{1+x}{1+x_0}$. 
    The explicit formulas \eqref{eq::WangsSolutionsTheta} for $V_\theta^{\beta,\theta_0}$ follow from \eqref{eq::ThetaXCoordinateChange}.
\end{proof}

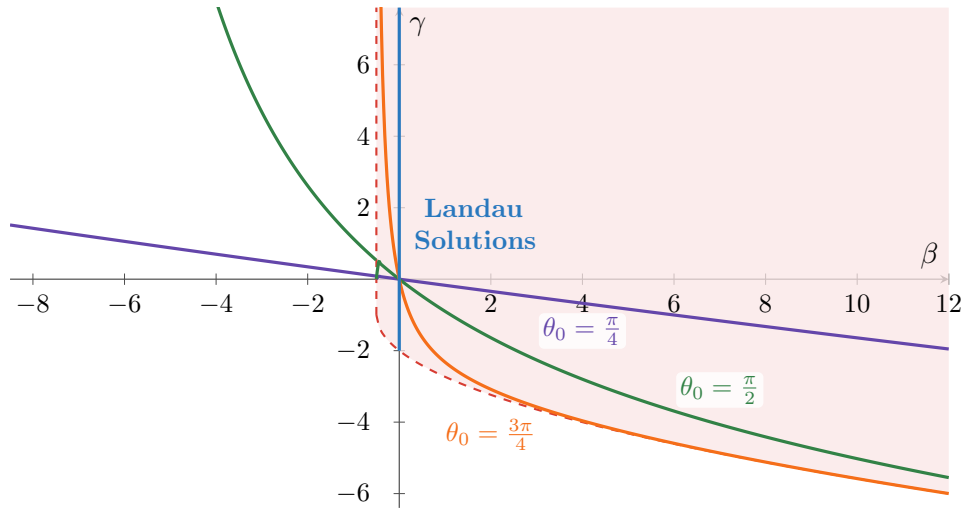
\begin{figure}[ht]
    \centering
    \begin{tikzpicture}
        \begin{axis}[
            width = 14cm,
            height = 8.2cm,
            xmin = -8.5,
            xmax = 12,
            ymin = -6.4,
            ymax = 7.6,
            axis lines = middle,
            axis line style = {black!70, line width = .45},
            tick style = {black!70},
            xtick = {-8, -6, -4, -2, 0, 2, 4, 6, 8, 10, 12},
            ytick = {-6, -4, -2, 0, 2, 4, 6},
            xlabel = {$\beta$},
            ylabel = {$\gamma$},
            clip = true,
            tick label style = {font = \small},
            xticklabel style = {font = \small},
            yticklabel style = {font = \small, xshift = -5pt},
        ]

        \path[name path = top] (axis cs:-.5, 7.6) -- (axis cs:12, 7.6);

        \addplot[
            name path = lower,
            draw = none,
            domain = -.5:12,
            samples = 250
        ]
        {-1 - sqrt(1 + 2 * x)};

        \addplot[
            draw = none,
            fill = ParamPink,
            fill opacity = .75
        ]
        fill between[of = top and lower];

        \addplot[
            Boundary,
            domain = -.5:12
        ]
        {-1 - sqrt(1 + 2 * x)};

        \addplot[
            BoundaryRed,
            dashed,
            thick
        ]
        coordinates {(-.5,-1) (-.5,7.6)};


        \addplot[
            Landau,
            CurvePurple,
            domain = -0.5:12
        ]
        {2 * x * (1 - pow(2 / (1 + cos(45)), sqrt(1 + 2 * x))) / (pow(2 / (1 + cos(45)), sqrt(1 + 2 * x)) * (sqrt(1 + 2 * x) - 1) + sqrt ( 1 + 2 * x) + 1)};

        \addplot[
            Landau,
            CurvePurple,
            domain = -8.5: -0.505
        ]
        {2 * x / (1 - sqrt(abs(1 + 2 * x)) * cot(deg(sqrt(abs(1 + 2 * x))/2 * ln(2 /( 1 + cos(45))))))};


        \addplot[
            Landau,
            CurveGreen,
            domain = -0.5:12
        ]
        {2 * x * (1 - pow(2 / (1 + cos(90)), sqrt(1 + 2 * x))) / (pow(2 / (1 + cos(90)), sqrt(1 + 2 * x)) * (sqrt(1 + 2 * x) - 1) + sqrt ( 1 + 2 * x) + 1)};

        \addplot[
            Landau,
            CurveGreen,
            domain = -5.2: -0.505
        ]
        {2 * x / (1 - sqrt(abs(1 + 2 * x)) * cot(deg(sqrt(abs(1 + 2 * x))/2 * ln(2 /( 1 + cos(90))))))};


        \addplot[
            Landau,
            CurveOrange,
            domain = -0.485:12
        ]
        {2 * x * (1 - pow(2 / (1 + cos(135)), sqrt(1 + 2 * x))) / (pow(2 / (1 + cos(135)), sqrt(1 + 2 * x)) * (sqrt(1 + 2 * x) - 1) + sqrt ( 1 + 2 * x) + 1)};

        \addplot[
            LandauBlue,
            very thick
        ]
        coordinates {(0,-2) (0,7.6)};

        \node[LabelStyle, text = CurvePurple] at (axis cs:4, -1.45) {$\theta_0=\frac{\pi}{4}$};
        \node[LabelStyle, text = CurveGreen] at (axis cs:7, -3.1) {$\theta_0=\frac{\pi}{2}$};
        \node[LabelStyle, text = CurveOrange] at (axis cs:2, -4.35) {$\theta_0=\frac{3\pi}{4}$};
        \node[text = LandauBlue, align = left, font = \bfseries\small, anchor = west] at (axis cs:.1,1.55) {\,\,Landau\\Solutions};
        
        \end{axis}
    \end{tikzpicture}
    \caption{
    The curve in the $(\beta,\gamma)$-parameter space corresponding to surface-discharge solutions in cones of angles $\theta_0=\frac{\pi}{4}, \frac{\pi}{2}, \frac{3\pi}{4}$. 
    The Landau solutions correspond to $\beta=0$.
    The shaded region is the Li-Li-Yan parameter region in \cite{LiLiYanI}.
    }
    \label{fig}
\end{figure}

As a final remark, note that \cite{RW} discusses self-similar solutions with a centered, interior singular ray on conical domains.
We expect that these solutions agree with the solutions discussed here when viewed on the other side of the physical boundary.

\subsection{Squire's surface discharge flows with swirl}
\label{subsec::Examples::SquireSurfaceDischargeFlowsWithSwirl}

In \cite{LiLiYanI}, Li, Li and Yan classify all classical solutions to \eqref{sns} which are $(-1)$-homogeneous, axisymmetric, with no-swirl in $\R^3\setminus S$ where 
\begin{align*}
    S=\{(0,0,x_3):x_3\leq 0\}.
\end{align*}
They further prove existence results for $(-1)$-homogeneous, axisymmetric solutions with non-zero swirl near the solutions in their parameter space.
Upon examining their proof, we see they in fact prove existence of a path of solutions extending from each surface discharge flow consisting of axisymmetric, $(-1)$-homogeneous solutions to \eqref{sns} in $\R^3_+$ with non-zero swirl and a no penetration boundary condition.
If surface discharge flows model fluids entrained to an oil slick emerging from a pipe, then these rotating solutions may model fluids entrained to an oil slick that is emerging from a \textit{rotating} pipe. 
As such, they represent a new class of physically motivated solutions.
In this section we confirm the existence of solutions in $\R^3_+$ with non-zero swirl and which satisfy a no-penetration boundary condition.

Solutions to \eqref{sns} which are $(-1)$-homogeneous and axisymmetric but with possibly nonzero swirl are described by the system
\begin{equation}
    \label{eq::RicattiODENonzeroSwirl}
    \begin{aligned}
        &(1-x^2)v_\theta'+2xv_\theta+\frac{1}{2}v_\theta^2-\int_x^1\int_\ell^1\int_t^1\frac{2v_\phi(s)v_\phi'(s)}{1-s^2}\,ds\,dt\,d\ell=\hat{\beta}(1-x)^2,\\
        &(1-x^2)v_\phi''+v_\phi v_\phi'=0.
    \end{aligned}
\end{equation}
We aim to find solutions to \eqref{eq::RicattiODENonzeroSwirl} with $v_\phi\neq 0$ which are close to a surface discharge flow $V^\beta:=V^{\beta,\pi/2}$ for all $\beta\geq -\frac{1}{2}$.
The result for $\beta<-\frac{1}{2}$ is currently unknown.
This has already been done in the whole space in \cite{LiLiYanI}.
Our observation in this section is that the results in \cite{LiLiYanI} also imply existence in the half-space with no-penetration boundary conditions as well.
To explain further, we need to recall a theorem from \cite{LiLiYanI}.

For $\beta\geq -\frac{1}{2}$, define the function $U_\beta:(-1,1)\to\R^2$ by,
\begin{align*}
    (U_\beta)_\theta(x):=0,
    \qquad 
    (U_\beta)_\phi(x):= \int_x^1 e^{-\int_0^t \frac{v^\beta_\theta(s)}{1-s^2}ds}dt,
\end{align*}
and thus only contributes to the swirl component of $v$.

\begin{theorem}
    \cite[Theorem 4.1]{LiLiYanI}. 
    Let $v^\beta$ be a solution from Proposition \ref{prop::SquireSolutionsClassification} where $\beta>-1/2$ and $\theta_0=\pi/2$. 
    Fix $K\subseteq(-1/2,\infty)$ compact and let $\epsilon$ be such that $v_\theta^\beta(-1)<2\epsilon<2$ for every $\beta\in K$.
    Then there exists a $\delta(K)>0$ and $F\in C^\infty(K\times (-\delta,\delta),X_1(\epsilon))$ satisfying $F(\beta,0)=0$ and $\frac{\partial F}{\partial \zeta}=0$, such that
    \begin{align*}
        v=v^\beta+\zeta U_\beta+F(\beta,\zeta),
    \end{align*}
    satisfies \eqref{eq::RicattiODENonzeroSwirl} with
    \begin{align*}
        \hat{\beta}=\beta-\frac{1}{4}\int_{-1}^1\int_\ell^1\int_t^1\frac{2v_\phi(s)v_\phi'(s)}{1-s^2}\,ds\,dt\,d\ell.
    \end{align*}
\end{theorem}

Here $X_1(\epsilon)$ is the Banach space 
\begin{align*}
    X_1(\epsilon):=\{(F_\theta,F_\phi):F_\theta\in M_1(\epsilon),F_\phi\in M_2(\epsilon),F_\theta(0)=F_\phi(0)=0\},
\end{align*}
where $M_1(\epsilon)$ and $M_2(\epsilon)$ are subsets of $C^1((-1,1];\R)$ with weighted $L^\infty$ norms.
The exact definition of $X_1(\epsilon)$ can be found in \cite[(4.14)]{LiLiYanI}, but it is quite complicated.
The important thing for us is that, $F_\theta(0)=0$. 
So $v_\theta(0)=0$ and $v_\phi\not\equiv 0$.

The next theorem represents a version of the prior theorem but for $\beta=-1/2$.

\begin{theorem}
    \cite[Theorem 4.2]{LiLiYanI}. 
    Let $\epsilon$ be such that $v_\theta^{-1/2}(-1)<2\epsilon<2$. 
    Then there exists a $\delta>0$ and $F\in C^\infty((-\delta,\delta),X_2(\epsilon))$ satisfying $F(0)=0$ and $\frac{dF}{d\zeta}=0$, such that 
    \begin{align*}
        v=v^{-1/2}+\zeta U_{-1/2}+F(\zeta),
    \end{align*}
    satisfies \eqref{eq::RicattiODENonzeroSwirl} with
    \begin{align*}
        \hat{\beta}=-\frac{1}{2}-\frac{1}{4}\int_{-1}^1\int_\ell^1\int_t^1\frac{2v_\phi(s)v_\phi'(s)}{1-s^2}\,ds\,dt\,d\ell.
    \end{align*}
\end{theorem}

Similarly $X_2(\epsilon)$ is the Banach space, defined in \cite[(4.36)]{LiLiYanI},
\begin{align*}
    X_2(\epsilon):=\{(F_\theta,F_\phi):F_\theta\in M_3(\epsilon),F_\phi\in M_4(\epsilon),F_\theta(0)=F_\phi(0)=0\},
\end{align*}
where $M_3(\epsilon)$ and $M_4(\epsilon)$ are subsets of $C^1((-1,1];\R)$ with different weighted $L^\infty$ norms. 
The case with $\beta<-\frac{1}{2}$ is not explored in \cite{LiLiYanI}.

These theorems imply that there are swirling solutions around any non-swirling solution from \cite{Squire2}.
We formalize this as a proposition.

\begin{proposition}
    \label{prop::ExistenceOfSquireSolutionsWithNonzeroSwirl}
    For $\beta\in[-1/2,\infty)$ there exists a $\delta(\beta)>0$ such that, for every $\zeta\in (-\delta,\delta)$ there exists a $(-1)$-homogeneous, axisymmetric solution $V^{\beta,\zeta}$ to \eqref{sns} in $\R^3_+$ with no-penetration boundary conditions.
    If $\zeta=0$ then $V^{\beta,\zeta}$ is the Squire solution $V^{\beta,\pi/2}$.
    If $\zeta\neq 0$ then $V^{\beta,\zeta}$ has non-zero swirl.
\end{proposition}

Under a smallness assumption, our asymptotic stability result Theorem \ref{mainthm1} implies these solutions are stable under perturbations which satisfy the same boundary condition.

\begin{proof}
    Fix $\beta$.
    Let $v^\beta$ be the solution to \eqref{eq::RicattiODENonzeroSwirl} corresponding to the Squire solution $V^{\beta,\pi/2}$.
    If $\beta\neq-1/2$ we use \cite[Theorem 4.1]{LiLiYanI}, otherwise we use \cite[Theorem 4.2]{LiLiYanI}.
    In either case, we have that
    \begin{align*}
        v^{\beta,\zeta}=v^\beta+\zeta U_\beta+F(\zeta)
        \qquad
        \text{where}
        \qquad 
        (v^\beta)_\theta(0)=(U_\beta)_\theta(0)=(F(\zeta))_\theta(0)=0.
    \end{align*}
    If $\zeta=0$ then $F(\zeta)\equiv 0$ and hence $v^{\beta,\zeta}=v^\beta$.
    As $(U_\beta)_\phi(0)\neq 0$ but $(F(\zeta))_\phi(0)=0$, if $\zeta\neq 0$ then $v^{\beta,\zeta}$ has non-zero $\phi$ component on the boundary.
\end{proof}

The boundary conditions for these solutions are no penetration and have radial component identical to that of the Squire surface discharge flow that we are perturbing around as well as the new swirl component coming only from $U_\beta$, which is given explicitly.
Although we have found these solutions within the framework of \cite{LiLiYanI}, they do not seem to have been discussed in connection with a physical application before.
The non-swirling solutions do not exist for all parameters $\beta$.
We can view the extant solutions as vector fields on the boundary plane.
The parameter $\beta$ represents their radial components.
The theorem above guarantees existence for small angular components for $\beta\geq-1/2$, but the smallness is not quantified.
An interesting open problem is to quantify this smallness.
More comprehensively, because any $(-1)$-homogeneous, axisymmetric vector field in the plane can be determined as a two-parameter family representing radial flow and swirl, it would be interesting to classify which $(-1)$-homogeneous and axisymmetric vector fields on the punctured plane have associated $(-1)$-homogeneous, axisymmetric solutions in $\R^3_+$.

\subsection{Type II.5 solutions}
\label{subsec::Examples::TypeII.5Solutions}

In \cite{Serrin}, Serrin constructed $(-1)$-homogeneous, axisymmetric solutions to \eqref{sns} on $\R^3_+$ which vanish on the boundary, see also \cite{Gusarov}.
Compared to Squire's solution, the tradeoff is that Serrin's solutions are singular on the axis of symmetry.
Serrin's solutions can possess spin and also exhibit multi-cellular behavior and, therefore, were originally proposed as tornado models.

In spherical coordinates (where $r=|x|$ is the radial variable, $\phi$ is the azimuthal variable and $\theta$ is the angle around the $x_3$ axis),
these solutions have the following asymptotics on the sphere near the north pole,
\begin{align*}
    |u_\theta|=O(|x_h|\ln|x_h|),
    \qquad 
    |u_r|=O(\ln|x_h|),
    \qquad 
    0\leq c=\lim_{|x|=1,x\to (0,0,1)}|x_h||u_\phi|(x)<\infty,
\end{align*}
where $x_h$ is the projection of $x$ onto its first two components.
Within the language of \cite{LiLiYanI,LiLiYanII,LiLiYanIII}, this means the solutions have a Type III singularity at the axis.
To be more precise, a Type III singularity is one in which
\begin{align*}
    0<\limsup_{|x|=1,|x_h|\to 0}|x_h||u(x_h)|<\infty.
\end{align*}
Asymptotic stability has not been proven for Type III solutions to \eqref{sns}, as is discussed in \cite{LiYan}. 
Our main insight is that the Serrin vortex is not Type III in all directions.
This motivates the following definition which is copied from Definition \ref{def::TypeII.5} in Section \ref{sec::Introduction}.

\begin{definition}[Type II.5 solutions]
    Assume $V$ is a self-similar, axisymmetric solution to \eqref{sns} which belongs to $L^\infty_{\mathrm{loc}}(\Omega\setminus\{x:x_h=0\})$ where $\Omega=\R^3$ or $\Omega=\R^3_+$.
    If $\Omega=\R^3$, then we say that $V$ is a Type II.5 solution (or has a Type II.5 singularity) if 
    \begin{align*}
        |V_{3}(x_h,\pm 1)|=O(|x_h|^{-2/3+}),
    \end{align*}
    while $V_{i}(x_h,\pm 1)=O(|x_h|^{-1})$ for $i=1,2$ and satisfy in a neighborhood of $x_2=0$ the asymptotic expansions 
    \begin{align*}
        V_i(x_1,0,\pm 1)=\frac{c_{i,\pm}}{x_1}+O(|x_1|^{-2/3+}),
    \end{align*}
    where $c_{i,\pm}$ are constants and $x_1>0$. If $\Omega=\R^3_+$, then we require these conditions are satisfied at $x_3=1$ for some constants $c_i$.
\end{definition}

Above, $-2/3+$ indicates the statement is true for any value slightly larger than $-2/3$.
Note that by self-similarity and axisymmetry, the asymptotic expansions at $x_3=\pm 1$ and $x_2=0$ extend to complete knowledge of $V$ in a cone around the $x_3$-axis.

Serrin's vortex is interesting as it was proposed as a tornado model and motivated us to introduce Type II.5 solutions.
However, we have not confirmed that these solutions satisfy the asymptotic expansion exactly, although we expect them to.
We presently identify two classes of solutions which are known to satisfy the definition of Type II.5 solutions:

\begin{itemize}
    \item The first example is from \cite{LiLiYanI} and is defined everywhere except the negative $x_3$ axis.
    In particular, let $v_\theta$ be the non-swirl, angular component of a solution from \cite[Theorem 1.3]{LiLiYanI}.
    Let $V_\theta(x:=\cos(\theta))=v_\theta(\theta)\sin(\theta)$.
    Then, \cite[Theorem 1.3]{LiLiYanI} states that, for $ V_\theta(-1)<2$,  $V_\theta(-1)\neq 0$ and $\alpha_0=1-\frac{V_\theta(-1)}{2}$,
    \begin{align*}
        V_\theta(x)=V_\theta(-1)+a_1(1+x)^{\alpha_0}+\text{lower order terms}.
    \end{align*}
    We therefore have 
    \begin{align*}
        v_r(\theta)=V_\theta'(x)=a_1\alpha_0(1+\cos(\theta))^{\alpha_0-1}+\text{lower order terms},
    \end{align*}
    where, in Cartesian coordinates, $v_r$ on the $x_3$-axis becomes $v_{x_3}$.
    Note that $1+\cos(\theta)$ is equivalent asymptotically speaking to $|x_h|$ near the south pole of the unit sphere, so this gets converted to $|x_h|^{\alpha_0-1}$.
    This gives us a Type II.5 solution whenever 
    \begin{align*}
        { - \frac {V_\theta(-1)} 2 > -\frac 2 3  },
    \end{align*}
    as this implies $\alpha_0-1>-2/3$.
    Note that $V_\phi$ has the same asymptotics as $V_\theta$, so the horizontal components of the velocity in Cartesian coordinates also satisfy the definition of Type II.5 solutions when $V_\theta(-1)<4/3$ and $V_\theta(-1)\neq 0$.  { Note that due to the sign of $\sin(\theta)$ as $\theta \to \pi^-$, $V_\theta (-1)<0$ corresponds to a source of radial momentum at the singular axis while $V_\theta(-1)>0$ corresponds to a sink.}
    A summary of results from Li, Li and Yan's papers is \cite[Theorem 2.4]{LiYanSurvey}. { A numerical plot of a solution with a source can be found in \cite[Fig.~6; p.~1161]{LiLiYanI}. 
    }

    \item The second example of Type II.5 self-similar solutions are those that also satisfy the Euler equations, namely, 
    \begin{align*}
        V_\theta=\frac{a\cos(\theta)+b}{\sin(\theta)},
        \qquad 
        V_r=a,
        \qquad 
        V_\phi=0,
    \end{align*}
    where $a,b\in\R$. We learned of these from \cite{LiYanSurvey}.
\end{itemize}

Under $(-1)$-homogeneity, the asymptotics asserted in this definition and Lemma \ref{lem::TypeIISolutionsAreInWeakL3} imply any Type II.5 solution $V$ satisfies 
\begin{align*}
    V=V_0+V_1,
\end{align*}
where $V_0\notin L^{3,\infty}(\Omega)$ but primarily depends on $x_h$ (its vertical derivatives are nice in some respect) while $V_1\in L^{3,\infty}(\Omega)$ and can depend on all variables.
These facts will be used to prove Theorem \ref{mainthm2} which states, provided we restrict our attention to \textbf{axisymmetric} perturbations, small Type II.5 solutions are asymptotically stable for perturbations { that satisfy \eqref{eq::EnergyInequalityModified} and \eqref{eq::GeneralizedEnergyInequality}. Such perturbations are expected to exist when the Type II.5 solution has a small source of momentum at the singular line. }

\subsection{Solutions with multiple line sources}
\label{subsec::Examples::SolutionsWithMultipleLineSources}

As seen in \cite[Proposition 1.1]{LiYan}, the solutions $u^{c,\gamma}$ for $(c,\gamma)\in M$ satisfy
\begin{align*}
    -\Delta u^{c,\gamma}+\nabla \cdot (u^{c,\gamma}\otimes u^{c,\gamma})+\nabla p^{c,\gamma}=(4\pi c_3\ln |x_3|\partial_{x_3}\delta_{(0,0,x_3)}-b^{c,\gamma}\delta_0)\vec{e_3}; \qquad \nabla \cdot u^{c,\gamma}=0;
\end{align*} 
on $\R^3$ in the sense of distributions. 
As a distribution, 
\begin{align*}
    \langle 4\pi c_3\ln|x_3|\partial_{x_3}\delta_{0,0,x_3},\varphi\rangle
    =-4\pi c_3\int_\R \ln|t|\partial_t \varphi(0,0,t)dt,
\end{align*}
for any scalar field $\varphi\in C_c^\infty(\R^3;\R)$ \cite[(8)]{LiYan}. 
It is natural to ask if one can solve the Stokes equation with this force and then apply a contraction argument as in the point singularity case to generate Navier-Stokes flows with infinite sources along multiple lines. 
In this section we answer this question affirmatively for sources that span a full line but are unable to answer the question for half-line forces.
The forces can be imposed on any combination of lines, not just the vertical axis. 

This is related to a problem is posed in the recent survey paper \cite{LiYanSurvey}:
\textit{Does there exist a non-axisymmetric $(-1)$-homogeneous solution exhibiting Type $2$ behavior?}  
Technically our present result answers this question, although the context in which the question is asked in \cite{LiYanSurvey} suggests an interest in non-axisymmetric $(-1)$-homogeneous solutions with Type $2$ behavior that are singular only on the $x_3$-axis. 
Our solutions, on the other hand, are non-axisymmetric and $(-1)$-homogeneous but have multiple singular axes, along which the solutions exhibit a Type $2$ singularity. 

For real numbers $a,b\in\R$ define the distributions $F_{a,b}^{\mathrm{ax}}$ and $F_{a,b}^+$ by
\begin{align*}
    \langle F_{a,b}^{\mathrm{ax}},\varphi\rangle
    & :=-a\int_\R \ln|t|\partial_t \varphi_3(0,0,t)dt-b\varphi_3(0,0,0), \\
    \langle F_{a,b}^+,\varphi\rangle
    & :=-a\int_0^\infty \ln|t|\partial_t \varphi_3(0,0,t)dt-b\varphi_3(0,0,0),
\end{align*}
for all vector fields $\varphi=(\varphi_1,\varphi_2,\varphi_3)\in C_c^\infty(\R^3;\R^3)$. 
The force $F_{a,b}^{\mathrm{ax}}$ is the force satisfied by $u^{c,\gamma}$ when $a=4\pi c_3$ and $b=b^{c,\gamma}$. 
The distribution $F_{a,b}^+$ is the corresponding force for the half axis $x_3>0$. Our goal now is to solve the Stokes problems  
\begin{equation}
    \label{eq::StokesSingularOnAxis}
    \left\{
    \begin{aligned}
    -\Delta V_{a,b}^{\mathrm{ax}} + \nabla q_{a,b}^{\mathrm{ax}} &= F_{a,b}^{\mathrm{ax}} && \text{in } \R^3\\
    \nabla \cdot V_{a,b}^{\mathrm{ax}} &= 0 && \text{in } \R^3
    \end{aligned}
    \right.\qquad\text{and}\qquad \left\{
    \begin{aligned}
    -\Delta V_{a,b}^+ + \nabla q_{a,b}^+ &= F_{a,b}^+ && \text{in } \R^3\\
    \nabla \cdot V_{a,b}^+ &= 0 && \text{in } \R^3
    \end{aligned}.
    \right.
\end{equation}
For the contraction argument in Section \ref{subsec::Examples::LandauSolutions} to go through, we would need the solutions $V_{a,b}^{\mathrm{ax}}$ and $V_{a,b}^+$ to be in $L^{3,\infty}$. 
We will find this holds for the former but not the latter.\par

Recall that a distribution in $\R^n$ is $(a)$-homogeneous if  
\begin{align*}
    \lambda^{-n-a}\langle T,\phi(\cdot/\lambda)\rangle=\langle T,\phi\rangle.
\end{align*} 
This agrees with the definition of homogeneity of functions when $T$ is associated with a function.
In $\R^3$ the Dirac delta function is $(-3)$-homogeneous. 
On the other hand, $\delta_{(0,0,x_3)}$ is effectively a two dimensional Dirac delta function and so should be $(-2)$-homogeneous. 
Thus, $\partial_{x_3}\delta_{(0,0,x_3)}$ should be $(-3)$-homogeneous. 
The following lemma clarifies this heuristic argument, which is valid on the full axis but, interestingly, not on the half-axis.

\begin{lemma}
    \label{lem::ForceScalingSingularAxis} 
    The full axis force is $(-3)$-homogeneous while the half-axis force satisfies
    \begin{align*}
        \left\langle F_{a,b}^+ ,\varphi(\cdot/\lambda)\right\rangle
        =\langle F_{a,b}^+,\varphi\rangle + \langle a \ln(\lambda)\delta_0 \vec{e_3},\varphi \rangle,
    \end{align*}
    and, in particular, is not $(-3)$-homogeneous.
\end{lemma}

This lemma is not used in the construction. 
It is illuminating, however, as the failure of $F_{a,b}^+$ to be $(-3)$-homogeneous is the obstruction to our method of constructing Type $2$ solutions on a half-axis. 
Note that the solutions with a half-line source in \cite{LiLiYanI} are Type III whereas the solutions that we would be generating here would be Type II, which cannot exist by the classification in \cite{LiLiYanI}. 
This lemma provides an intuition for why (small) Type II self-similar, axisymmetric, no-swirl solutions with  a half-line source do not exist. 

\begin{proof}
    Fix a test function $\varphi=(\varphi_1,\varphi_2,\varphi_3)\in C_c^\infty(\Omega)$. 
    Then, because $-3-(-3)=0$, the $(-3)$-homogeneity of $F_{a,b}^{\mathrm{ax}}$ is follows from the following calculation,
    \begin{align*}
             \left\langle  F_{a,b}^{\mathrm{ax}} ,\varphi (\cdot/\lambda)\right\rangle
        & =- a\int_\R \ln|t|\partial_t \varphi_3(0,0,t/\lambda)dt- b\varphi_3(0,0,0)\\
        & = -a \int_\R \ln|\lambda s|\partial_s \varphi_3(0,0,s)ds- b\varphi_3(0,0,0)\\
        & = -a \left(\ln|\lambda|\int_\R \partial_s \varphi_3(0,0,s)ds+\int_\R \ln|s|\partial_s \varphi_3(0,0,s)ds\right)- b\varphi_3(0,0,0)\\
        & = \langle F_{a,b}^{\mathrm{ax}}, \varphi \rangle,
    \end{align*}
    because
    \begin{align*}
        \int_\R \partial_s \varphi_3(0,0,s)ds=\int_{-L}^L \partial_s\varphi_3(0,0,s)ds=\varphi_3(0,0,L)-\varphi_3(0,0,-L)=0,
    \end{align*}
    for some $L$ sufficiently large. 
    However,
    \begin{align*}
            \left\langle F_{a,b}^+ ,\varphi(\cdot/\lambda)\right\rangle
        & = -a \left(\ln|\lambda|\int_0^\infty \partial_s \varphi_3(0,0,s)ds
          + \int_0^\infty \ln|s|\partial_s \varphi_3(0,0,s)ds\right)- b\varphi_3(0,0,0) \\
        & = \langle F_{a,b}^+,\varphi\rangle  
          + \langle a \ln(\lambda)\delta_0 \vec{e_3},\varphi \rangle,
    \end{align*}
    because 
    \begin{align*}
        \int_0^\infty \partial_s \varphi_3(0,0,s)ds
        = \int_0^L\partial_s \varphi_3(0,0,s)
        = \varphi_3(0,0,L)-\varphi_3(0,0,0)
        = -\varphi_3(0,0,0),
    \end{align*}
    for some $L$ sufficiently large. 
\end{proof} 

Let $(G,\Pi)$ be the Oseen tensor. That is, 
\begin{align*}
    G_{ij}(x,y)=\frac{1}{8\pi}\left(\frac{\delta_{ij}}{|x-y|}
    +\frac{(x_i-y_i)(x_j-y_j)}{|x-y|^3}\right),\qquad \Pi_j(x,y)
    =\frac{(x_j-y_j)}{4\pi|x-y|^3}.
\end{align*} 

We now give formulas for the solutions to \eqref{eq::StokesSingularOnAxis}. 
We will show the functions $(V_{a,b}^{\mathrm{ax}},q_{a,b}^{\mathrm{ax}})$ with components 
\begin{align*}
      (V_{a,b}^{\mathrm{ax}})_i(x)
    & =-a\int_\R \ln|t|\partial_t G_{i3}(x,t\vec{e_3})dt-bG_{i3}(x), \\
      q_{a,b}^{\mathrm{ax}}(x)
    & =-a\int_\R \ln|t|\partial_t \Pi_3(x,t\vec{e_3})dt-b\Pi_3(x),
\end{align*}
for $x\neq s \vec{e_3}$ where $s\in \R$ with $t\in\R$ and $(V_{a,b}^+,q_{a,b}^+)$ with components 
\begin{align*}
    (V_{a,b}^+)_i(x)
    & =-a\int_0^\infty \ln|t|\partial_t G_{i3}(x,t\vec{e_3})dt-bG_{i3}(x),\\
    q_{a,b}^+(x)
    & =-a\int_0^\infty \ln|t|\partial_t \Pi_3(x,t\vec{e_3})dt-b\Pi_3(x),
\end{align*}
for $x\neq s \vec{e_3}$ with $s\in\R^+$ satisfy \eqref{eq::StokesSingularOnAxis} in the sense of distributions in Lemma \ref{lem::StokesSolutionSingularAxis}.
Formally, these are the convolutions $(G*F_{a,b}^{\mathrm{ax}},\Pi*F_{a,b}^{\mathrm{ax}})$ and $(G*F_{a,b}^+,\Pi*F_{a,b}^+)$. 
However, neither $G$, $\Pi$, $F_{a,b}^{\mathrm{ax}}$, or $F_{a,b}^+$ have compact support. 

We know the Stokeslet $G$ is $(-1)$-homogeneous. 
An argument similar to the one in Lemma \ref{lem::ForceScalingSingularAxis} proves the following.

\begin{lemma}
    \label{lem::StokesSolutionHomogeneity}
    The full-axis singularity solutions $(V_{a,b}^{\mathrm{ax}},q_{a,b}^{\mathrm{ax}})$ satisfy 
    \begin{align*}
        V_{a,b}^{\mathrm{ax}}(\lambda x)=\lambda^{-1}V_{a,b}^{\mathrm{ax}}(x),
        \qquad 
        q_{a,b}^{\mathrm{ax}}(\lambda x)=\lambda^{-2}q_{a,b}^{\mathrm{ax}}.
    \end{align*}
    However, the half-axis singularity solutions $(V_{a,b}^+,q_{a,b}^+)$ satisfy 
    \begin{align*}
        V_{a,b}^+(\lambda x) &= \lambda^{-1}V_{a,b}^+(x)+a\lambda^{-1}\ln(\lambda)G_{\cdot,3}(x),\\
        q_{a,b}^+(\lambda x) &= \lambda^{-2}q_{a,b}^+(x)+a\lambda^{-2}\ln(\lambda)\Pi_3(x).
    \end{align*}
\end{lemma}

Since $V_{a,b}^{\mathrm{ax}}$ is $(-1)$-homogeneous it is reasonable to postulate $V_{a,b}^{\mathrm{ax}}\in L^{3,\infty}$. Furthermore, the manner of the failure of $V_{a,b}^+$ to be $(-1)$-homogeneous indicates its failure to be in $L^{3,\infty}$ and is consistent with the non-existence of Type II self-similar, axisymmetric no-swirl solutions in \cite{LiLiYanI}.

\begin{lemma}
    We have that, $V_{a,b}^+\notin L^{3,\infty}$.
\end{lemma}

\begin{proof}
    Fix $\lambda>0$. 
    For contradiction, assume $V_{a,b}^+\in L^{3,\infty}$. 
    By Lemma \ref{lem::StokesSolutionHomogeneity}, 
    \begin{align*}
        a\ln(\lambda) G_{\cdot,3}(x)=\lambda V_{a,b}^+(\lambda x)-V_{a,b}^+(x).
    \end{align*}
    Taking the $L^{3,\infty}$ quasi-norm on both sides gives 
    \begin{equation}
        \label{eq::WeakL3Contradiction}
        |a| |\ln(\lambda)| \| G_{\cdot,3} \|_{L^{3,\infty}}
        \leq C\left(\|\lambda V_{a,b}^+(\lambda x)\|_{L^{3,\infty}}
        +\|V_{a,b}^+(x)\|_{L^{3,\infty}}\right)
        \leq 2C\|V_{a,b}^+\|_{L^{3,\infty}},
    \end{equation}
    by the quasi-triangle inequality and criticality of $L^{3,\infty}$. 
    The inequality \eqref{eq::WeakL3Contradiction} holds for $\lambda>0$. 
    But, $|a| |\ln(\lambda)| \| G_{\cdot,3} \|_{L^{3,\infty}}\to\infty$ as $\lambda\to\infty$. 
    This is a contradiction since $2C\|V_{a,b}^+\|_{L^{3,\infty}}$ is a finite number independent of $\lambda$.
\end{proof}

A consequence of this lemma is that the argument in Section \ref{subsec::Examples::LandauSolutions} fails for solutions with force supported on a half-axis. This is consistent with the fact that the solutions studied in \cite{LiLiYanI} are all  Type III singular. For the full axis source, on the other hand, we have $V_{a,b}^{\mathrm{ax}}\in L^{3,\infty}$. 

\begin{lemma}\label{lem::VinLorentz}
    We have that, $V_{a,b}^{\mathrm{ax}}\in L^{3,\infty}$ and, furthermore, $\| V_{a,b}^{\mathrm{ax}} \|_{L^{3,\infty}}\leq C(|a|+|b|)$.
\end{lemma}

\begin{proof}
That $b G_{i3}\in L^{3,\infty}$ is obvious.

    For simplicity of notation, set $G(x-y):=G_{\cdot,3}(x,y)$.   Fix $x$ off of the vertical axis so that $|x_h|\neq 0$.  This means that $G(x-t\vec{e_3})$ is bounded in $t$.  We first confirm the  integration by parts formula, 
      \begin{align}\label{IBP}
        \int_\R \ln|t|\partial_t G (x-t\vec{e_3})\,dt =- \int_\R \frac 1 t G (x-t\vec e_3)\,dt,
    \end{align} 
    is valid.  
    For each $\epsilon>0$, 
    \begin{equation}
        \label{a::eq::IntegrationByParts}
        \begin{aligned}
            -\int_{\epsilon<|t|<\frac{1}{\epsilon}} \ln|t|\partial_t G(x-t\vec{e_3})\,dt&=\int_{\epsilon<|t|<\frac{1}{\epsilon}} \frac{G(x-t\vec{e_3})}{t}\,dt
            +\ln|\epsilon|\left(G(x-\epsilon \vec{e_3})-G(x+\epsilon \vec{e_3})\right)\\
            &\quad +\ln|\epsilon^{-1}|(\left(G(x+\epsilon^{-1}\vec{e_3})-G(x-\epsilon^{-1}\vec{e_3})\right).
        \end{aligned}
    \end{equation}
    The two boundary terms are $O(\ln|\epsilon|\epsilon)$ which can be seen using the mean value theorem and the fact that we are avoiding any singularities in $G$ or its derivatives because $|x_h|\neq 0$. Sending $\epsilon\to 0$ justifies  \eqref{IBP} .


    We consider the integrals over $|t|\leq |x|/2$ and $|t|\geq |x|/2$. We have
    \begin{align*}
        \int_{-|x|/2}^{|x|/2} \frac 1 t G(x-t\vec e_3) \,dt 
        = \int_{-|x|/2}^{|x|/2} \frac 1 t (G(x-t\vec e_3)-G(x)) \,dt,
    \end{align*}
    because $1/t$ is odd. Using the mean value theorem we have 
    \begin{align*}
        \bigg| \int_{-|x|/2}^{|x|/2} \frac 1 t G(x-t\vec e_3) \,dt \bigg|
        \leq \int_{-|x|/2}^{|x|/2} | \partial_{x_3}G(y)|\,dt,
    \end{align*}
    where $y = x-s\vec e_3$ for some $s\in [x_3-t,x_3]$ when $t>0$ and $s\in [x_3,x_3-t]$ when $t<0$. In either case, $|y|\sim |x|$ and so 
    \begin{align*}
    | \partial_{x_3}G(y)| \sim |x|^{-2}.
    \end{align*}
    As this no longer depends on $t$ we can move it outside the integral to obtain
    \begin{align*}
       \bigg| \int_{-|x|/2}^{|x|/2} \frac 1 t G(x-t\vec e_3) \,dt\bigg|\leq \frac C {|x|}\in L^{3,\infty}.
    \end{align*}
    For the far-field integral we have 
    \begin{align*}
    \bigg|    \int_{|x|/2}^\infty   \frac 1 t G(x-t\vec e_3) \,dt \bigg|
        \leq \| \frac 1 t \|_{L^2((|x|/2,\infty)} \| G(x-(\cdot) \vec e_3)\|_{L^2(\R)}
        \leq \frac 1 {|x|^{1/2}} \| G(x-(\cdot ) \vec e_3)\|_{L^2(\R)}.
    \end{align*}
    We have 
    \begin{align*}
        \int_\R |G(x-t \vec e_3)|^2\,dt
        \leq \int_\R \frac 1 {|x_h|^2 + s^2}\,ds.
    \end{align*}
    where $s=x_3-t$. We have 
    \begin{align*}
        \int_\R \frac 1 {|x_h|^2 + s^2}\,ds 
        = \frac 1 {|x_h|^2} \int_\R \frac 1 {1+ (s/|x_h|)^2}\,ds.
    \end{align*}
    Let $\tau = s/|x_h|$ so that 
    \begin{align*}
        \int_\R \frac 1 {|x_h|^2 + s^2}\,ds
        = \frac 1 {|x_h|} \int_\R \frac 1 {1+ (\tau)^2}\,d\tau,
    \end{align*}
    where the above integral is bounded independently of $x$.
    Altogether this implies
    \begin{align*}
     \bigg|   \int_{|x|/2}^\infty   \frac 1 t G(x-t\vec e_3) \,dt \bigg|
        \leq \frac C {|x|^{1/2}|x_h|^{1/2}}\in L^{3,\infty}.
    \end{align*}
    The integral over $(-\infty, -|x|/2 ]$ is bounded identically. 
    This proves the inclusion in $L^{3,\infty}$. The upper bound comes from multiplying by $|a|$ throughout the above inequalities.
\end{proof}

\begin{lemma}\label{lem::qinLorentz}
    We have $q_{a,b}^{\mathrm{ax}} \in L^{3/2,\infty}$ and, furthermore, $\|q_{a,b}^{\mathrm{ax}} \|_{L^{3/2,\infty}}\leq C (|a|+|b|)$.
\end{lemma}
\begin{proof} The proof is similar to the proof of Lemma \ref{lem::VinLorentz} and we try to be brief.
Recall that $q_{a,b}^{\mathrm{ax}}(x)
     =-a\int_\R \ln|t|\partial_t \Pi_3(x,t\vec{e_3})dt-b\Pi_3(x)$. The $b\Pi_3(x)$ part is in $L^{3/2,\infty}$ because it has an explicit formula as the pressure associated with a Stokeslet. 
     We adopt the abbreviation $ \Pi(x-y)=\Pi_3(x,y)$. 
     For the integral term, we first observe that via an identical argument which supported \eqref{IBP}, we have 
     \[
        \int_\R \ln|t|\partial_t \Pi(x-t\vec{e_3})dt =- \int_\R \frac 1 {t} \Pi(x-t\vec{e_3})dt.
     \]
     We consider the integrals over $|t|\leq |x|/2$ and $|t|\geq |x|/2$. We have
    \begin{align*}
        \int_{-|x|/2}^{|x|/2} \frac 1 t \Pi (x-t\vec e_3) \,dt 
        = \int_{-|x|/2}^{|x|/2} \frac 1 t (\Pi(x-t\vec e_3)-\Pi(x)) \,dt.
    \end{align*}
    Using the mean value theorem and the fact that in this case $|x|\sim |x\pm t\vec e_3|$ implies 
    \begin{align*}
       \bigg| \int_{-|x|/2}^{|x|/2} \frac 1 t \Pi (x-t\vec e_3) \,dt \bigg|
        \leq C \int_{-|x|/2}^{|x|/2} | \partial_{x_3} \Pi(y)|\,dt\leq C|x|\frac 1 {|x|^3}\in L^{3/2,\infty},
    \end{align*}
    where above $y$ is some point with $|y|\sim |x|$. 

    For the far-field integral we have 
    \begin{align*}
        \bigg|\int_{|x|/2}^\infty   \frac 1 t \Pi (x-t\vec e_3) \,dt \bigg|
        \leq \frac 2 {|x|} \int_{|x|/2}^\infty |\Pi(x-(\cdot)\vec e_3)|\,dt,
    \end{align*}
    We have for $s=x_3-t$ that
    \begin{align*}
    \int_\R   |\Pi(x-t \vec e_3)|\,dt
        \leq \int_\R \frac C {|x_h|^{2} + s^2}\,ds\leq \frac C {|x_h|}.
    \end{align*}
    Hence,
    \begin{align*}
        \bigg|\int_{|x|/2}^\infty   \frac 1 t \Pi (x-t\vec e_3) \,dt \bigg|
        \leq \frac C {|x||x_h|}.
    \end{align*}
    Because $(|x_h||x|)^{-1/2}\in L^{3,\infty}$, we also have $(|x_h||x|)^{-1} \in L^{3/2,\infty}$. The integral over $(-\infty, -|x|/2)$ is bounded identically.
\end{proof}
 
\begin{lemma}
    \label{lem::StokesSolutionSingularAxis}
    The pair $(V_{a,b}^{\mathrm{ax}},q_{a,b}^{\mathrm{ax}})$ solves \eqref{eq::StokesSingularOnAxis} in the sense of distributions.
\end{lemma}

\begin{proof}
    Fix a test function $\varphi\in C_c^\infty(\R^3;\R^3)$.
    Let $V_{a,b}$ denote $V_{a,b}^\mathrm{ax}$ and adopt the same convention for $q_{a,b}$ and $F_{a,b}$. 
    For $(V_{a,b},q_{a,b})$ to satisfy \eqref{eq::StokesSingularOnAxis} in the sense of distributions we need to show that
    \begin{align}\label{identityForDistSol}
        -\langle V_{a,b},\Delta\varphi\rangle-\langle q_{a,b},\nabla\cdot\varphi\rangle=F_{a,b}.
    \end{align} 
    Implicit in this is that the left-hand sides of the above expressions are finite. This follows from Lemmas \ref{lem::VinLorentz} and \ref{lem::qinLorentz}.

    We can calculate the action of the vector field  
    \begin{align*}
        \int_\R \ln|t|\partial_t G_{\cdot 3}(x,t\vec{e}_3)\,dt,
    \end{align*} 
    on a test function $\varphi$ as follows,
    \begin{align*}
        \left\langle \int_I \ln|t|\partial_t G_{\cdot 3}(x,t\vec{e}_3)\,dt,\varphi\right\rangle
        &=\int_\R \ln|t| \int_{\R^3}\partial_t G_{\cdot 3}(x,t\vec{e}_3)\cdot \varphi(x)\,dx\,dt\\
        &=\int_\R \ln|t| \int_{\R^3}-\partial_{x_3} G_{\cdot 3}(x,t\vec{e}_3)\cdot \varphi(x)\,dx\,dt\\
        &=\int_\R \ln|t|\int_{\R^3}G_{\cdot 3}(x,t\vec{e}_3)\cdot \partial_{x_3}\varphi(x)\,dx\,dt\\
        &=\int_\R \ln|t|\langle G_{\cdot 3}(x,te_3),\partial_{x_3}\varphi\rangle \,dt,
    \end{align*}
    where we used the fact that $\partial_{y_3}G(x,y)=-\partial_{x_3}G(x,y)$.
    For the pressure and test functions $\psi\in C_c^\infty(\R^3;\R)$ we have by a similar calculation that
    \begin{align*}
        \left\langle\int_\R \ln|t|\partial_t \Pi_3(x,t\vec{e}_3)\,dt,\psi\right\rangle
        =\int_\R \ln|t|\langle \Pi_3 (x,t\vec{e}_3),\partial_{x_3}\psi\rangle \,dt.
    \end{align*}
    So, 
    \begin{align*}
          - \langle V_{a,b},\Delta\varphi\rangle-\langle q_{a,b},\nabla\cdot \varphi\rangle
        & = a\int_\R \ln|t|\left(\langle G_{\cdot 3}(x,te_3),\Delta\partial_{x_3}\varphi\rangle 
          + \langle \Pi_3 (x,t\vec{e}_3),\nabla\cdot \partial_{x_3}\varphi\rangle\right)\,dt \\
        & + b(\langle G_{\cdot 3},\Delta\varphi\rangle
          + \langle \Pi_3,\nabla\cdot \varphi\rangle)
          = -a\int_\R\ln|t|\partial_{x_3}\varphi_3(0,0,t)\,dt
          - b\varphi_3(0),
    \end{align*}
    because  
    \begin{align*}
          \langle G_{\cdot 3}(x,t\vec{e_3}),\Delta \varphi_i\rangle
        + \langle\Pi_3(x-t\vec{e_3}),\nabla\cdot \varphi\rangle
        = -\varphi_3(0,0,t),
    \end{align*}
    for each $t\in I$ and test function $\varphi$.
\end{proof}

Once one knows $V_{a,b}^{\mathrm{ax}}\in L^{3,\infty}$ can be made small in the parameter space and that $q_{a,b}^{\mathrm{ax}}\in L^{3/2,\infty}$, the argument which constructs solutions to \eqref{sns} whose forcing is made up of translations of rotations of $F_{a,b}^{\mathrm{ax}}$ is identical to the corresponding argument given in Section \ref{subsec::Examples::LandauSolutions}---see Remark \ref{rem::ExistenceOfSolutionsWithFiniteNumberOfPointSingularitiesRemark}. 
To make this precise we introduce the following notation. 
Let $R\in SO(3)$ be a rotation matrix and $z\in\R^3$ be a vector we will translate by. 
Define the distribution $F_{a,b}^{R,z}$ by 
\begin{align*}
    \langle F_{a,b}^{R,z},\varphi\rangle
    :=\langle F_{a,b}^{\mathrm{ax}},R^T\varphi(R\cdot+z)\rangle
    =-a\int_\R \ln|t|\partial_t Re_3\cdot \varphi(tRe_3+z)dt-b Re_3\cdot \varphi(z).
\end{align*}
This is the translation by $z$ of the rotation by $R$ of $F_{a,b}^{\mathrm{ax}}$. 
It is supported on the axis 
\begin{align*}
    \ell^{R,z}=\{z+tRe_3:t\in\R\}.
\end{align*} 
The pair $(V_{a,b}^{R,z},q_{a,b}^{R,z})$ defined by 
\begin{align*}
    V_{a,b}^{R,z}(y)=RV_{a,b}^{\mathrm{ax}}(R^T(y-z))
    \quad 
    \text{and} 
    \quad 
    q_{a,b}^{R,z}(y)=q_{a,b}^{\mathrm{ax}}(R^T(y-z)),
\end{align*} 
solves the stokes equation 
\begin{equation}
    \label{eq::RotatedStokes}
    - \Delta V_{a,b}^{R,z}+\nabla q_{a,b}^{R,z}
    = F_{a,b}^{R,z};\qquad \nabla\cdot V_{a,b}^{R,z}
    = 0,
\end{equation}
in the sense of distributions.

\begin{figure}[ht]
    \centering
    \begin{subfigure}{0.48\textwidth}
        \centering
        \includegraphics[width=\linewidth]{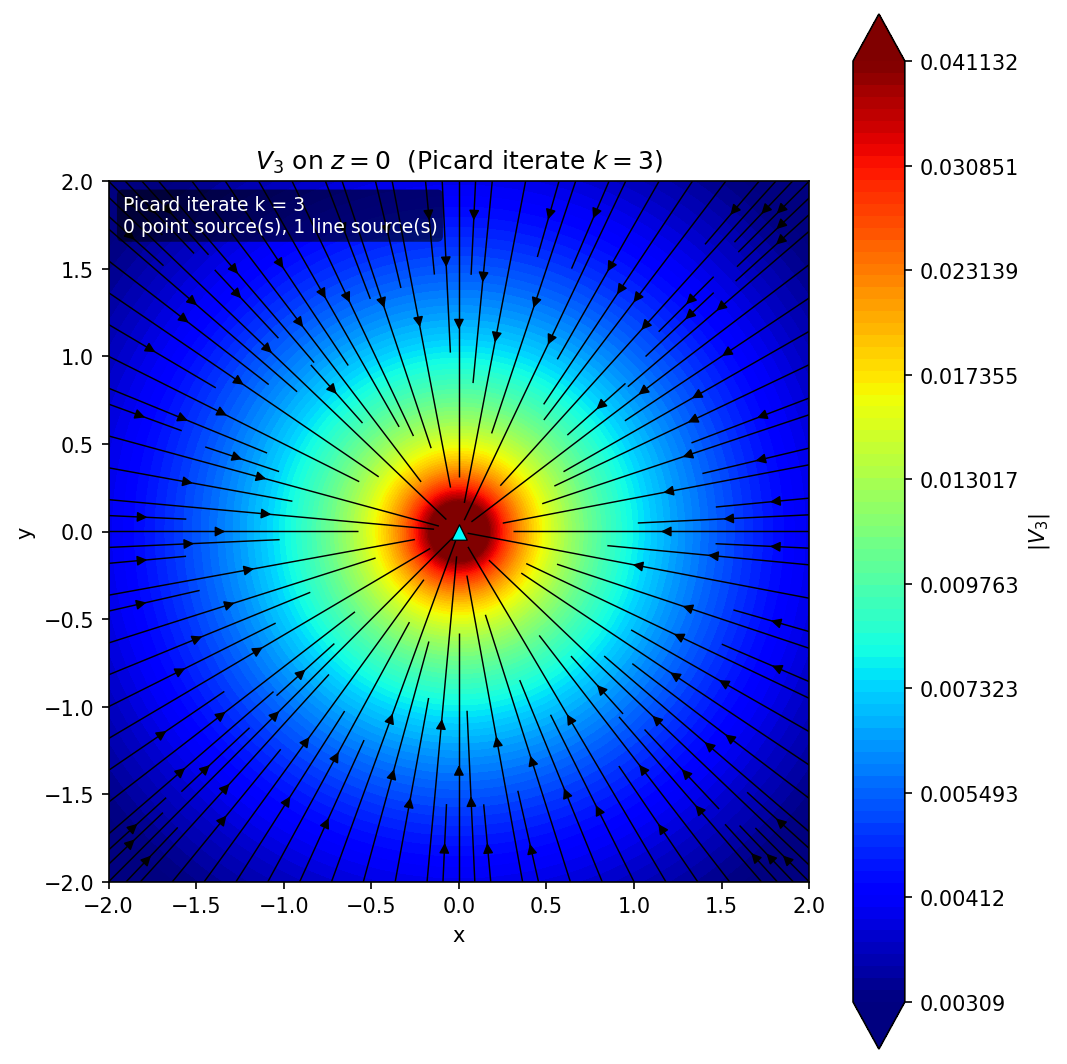}
        \caption{The $xy$-plane}
    \end{subfigure}
    \begin{subfigure}{0.48\textwidth}
        \centering
        \includegraphics[width=\linewidth]{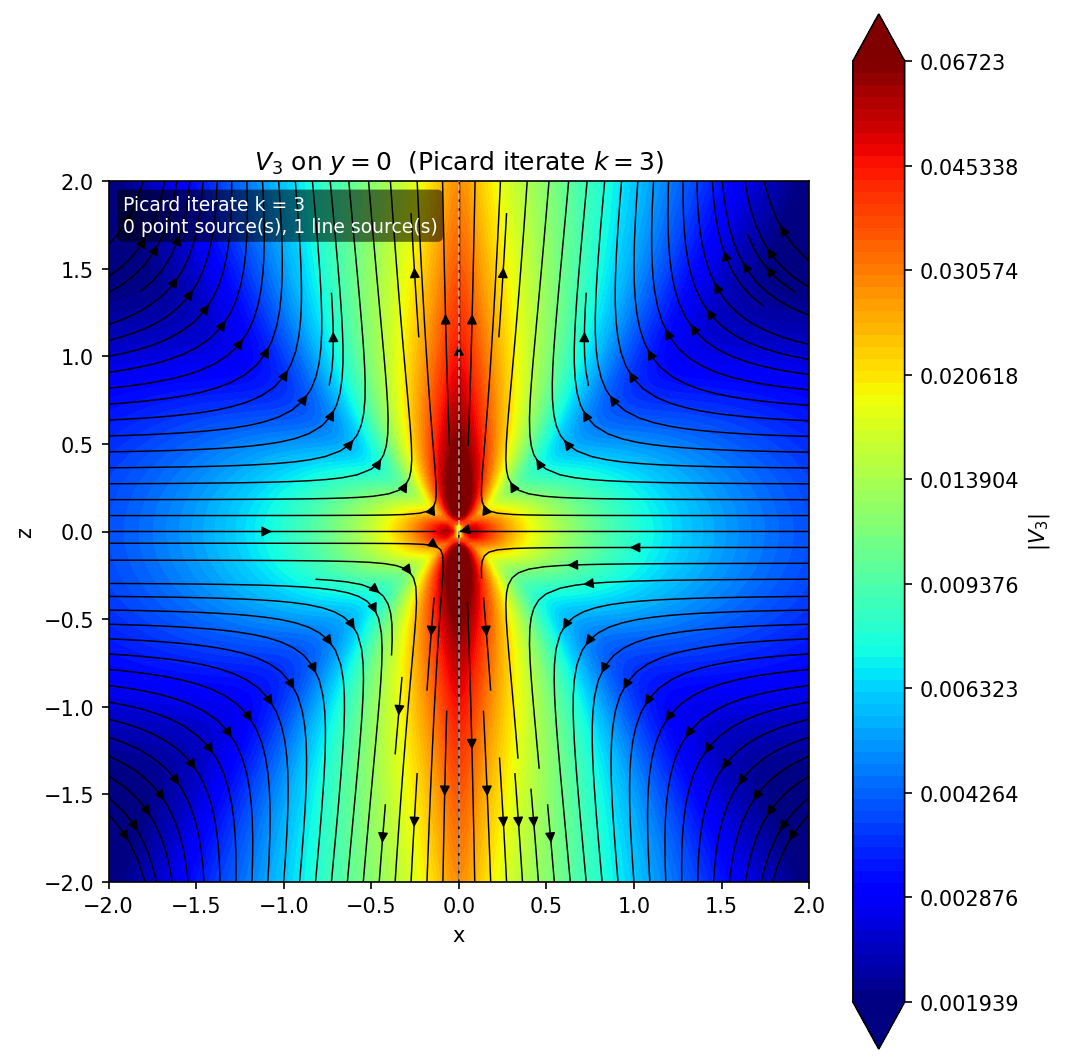}
        \caption{The $xz$-plane}
    \end{subfigure}
    \caption{Illustrations of flows driven by the Stokes seed $V_{.1,0}^{\mathrm{ax}}$. }
    \label{fig::1l}
\end{figure}
\begin{figure}[ht]
    \centering
    \begin{subfigure}{0.48\textwidth}
        \centering
        \includegraphics[width=\linewidth]{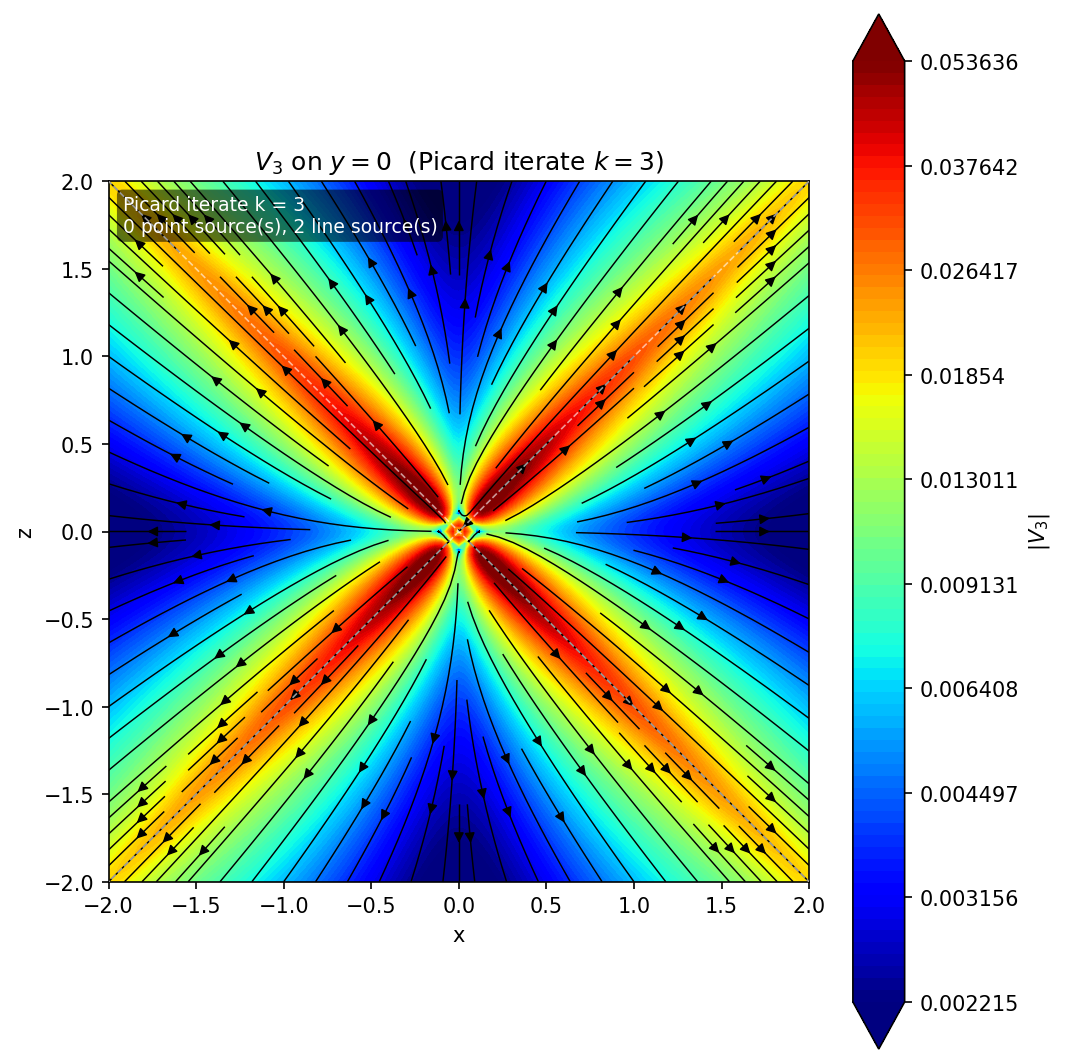}
        \caption{The $xy$-plane}
    \end{subfigure}
    \begin{subfigure}{0.48\textwidth}
        \centering
        \includegraphics[width=\linewidth]{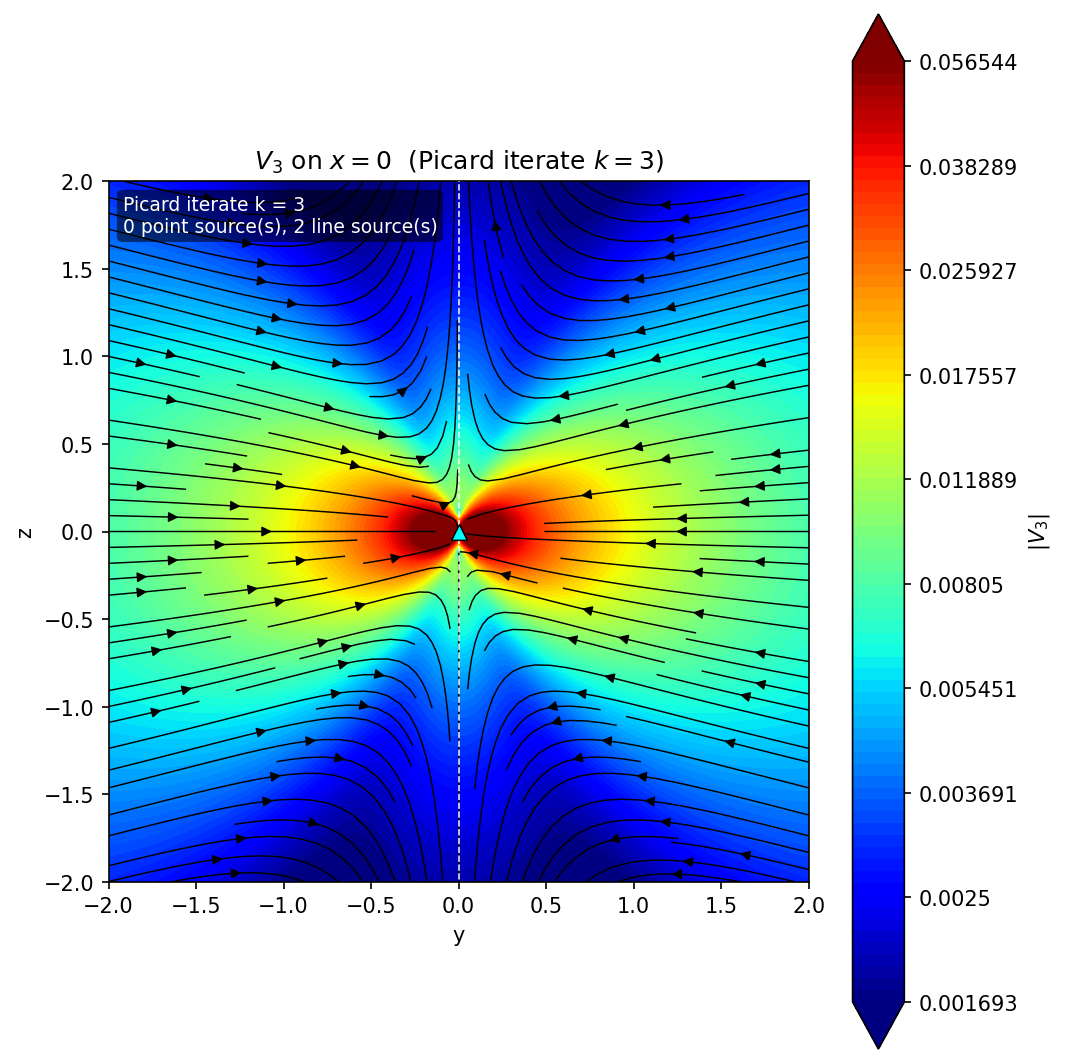}
        \caption{The $yz$-plane}
    \end{subfigure}
    \caption{Illustrations of flows driven by two perpendicular line-sources which overlap at their centers. }
    \label{fig::2l}
\end{figure}

\begin{theorem}
    \label{thm::WholeSpaceExistenceOfSolutionsWithFiniteNumberOfLineSingularities}
    Fix a positive integer $n$. 
    For $1\leq i\leq n$ fix real numbers $a_i,b_i\in\R$, rotation matrices $R_i\in SO(3)$, and vectors $z_i\in\R^3$ such that $\ell^{R_i,z_i}\cap \ell^{R_j,z_j}$ is at most one point for all $i\neq j$. 
    Let 
    \begin{align*}
        F_i = F_{a_i,b_i}^{R_i,z_i}
        \qquad
        \text
        {and}
        \qquad
        (V_i,P_i) = (V_{a_i,b_i}^{R_i,z_i},P_{a_i,b_i}^{R_i,z_i}),\end{align*} 
    be the corresponding line sources and solutions to \eqref{eq::RotatedStokes}. 
    Let 
    \begin{align*}
        F=\sum_{i=1}^n F_i, 
        \qquad V_0=\sum_{i=1}^n V_i,
        \qquad P_0=\sum_{i=1}^n P_i.
    \end{align*}
    For every $\epsilon>0$ there exists $\delta>0$ such that there exists a vector field $V\in L^{3,\infty}(\R^3)$ and a scalar field $q\in L^{3/2,\infty}_{\mathrm{loc}}(\R^3)$ satisfying
    \begin{equation}
        \label{eq::RotatedSNS}
        -\Delta V+(V\cdot\nabla)V+\nabla P=F\qquad \nabla\cdot V=0,
    \end{equation}
    and $\|V\|_{L^{3,\infty}}<\epsilon$ whenever $\|V_0\|_{L^{3,\infty}}<\delta$. 
    Furthermore, if $z_i=0$ for all $1\leq i\leq n$ then $V$ is $(-1)$-homogeneous.
\end{theorem}

Figures \ref{fig::1l} and \ref{fig::2l} illustrate the behavior of two Type II solutions each viewed from two directions. The solution depicted in \ref{fig::1l} is generated by the Stokes seed $V^{\mathrm{ax}}_{.1,0}$. In \ref{fig::1l}a and \ref{fig::1l}b we depict the $xy$-plane and $xz$-plane respectively. The solution in \ref{fig::2l} consists of two perpendicular line singularities exerting the same magnitude of force. In \ref{fig::2l}a and \ref{fig::2l}b we depict the $xy$-plane and $yz$-plane respectively. The non-linear interactions are approximated up to the third Picard iterate.

\begin{proof} Once we note that $\|V_0\|_{L^{3,\infty}}$ can be made small by Lemma \ref{lem::VinLorentz} and that $P_0\in L^{3/2,\infty}$,
    the proof for the first part is identical to the proof of Theorem \ref{thm::ExistenceOfSolutionsWithFiniteNumberOfPointSingularities}---see Remark \ref{rem::ExistenceOfSolutionsWithFiniteNumberOfPointSingularitiesRemark}. 
    To see that $V$ is $(-1)$-homogeneous in the case that all forces $F_i$ are centered at the origin, define the subspace 
    \begin{align*}
        L^{3,\infty}_{ss} := \left\{V\in L^{3,\infty}:V(x)=\lambda V(\lambda x) \text{ for all }\lambda\in\R\right\}.
    \end{align*} For each fixed $\lambda>0$ the operator $\Lambda_\lambda:L^{3,\infty}\to L^{3,\infty}$ mapping $V$ to $\lambda V(\lambda \cdot)$ has operator norm one and is hence continuous. 
    So for any sequence $V_n\in L^{3,\infty}_{ss}$ with limit $V\in L^{3,\infty}$ satisfies 
    \begin{align*}
        \Lambda_\lambda V=\Lambda_\lambda\lim_{n\to\infty} V_n= \lim_{n\to\infty}\Lambda_\lambda V_n=\lim_{n\to\infty}V_n=V.
    \end{align*} 
    Thus $L^{3,\infty}_{ss}$ is a Banach space. 
    We know $V_0\in L^{3,\infty}$. 
    It remains to show the contraction, 
    \begin{align*}
        \Phi(V):=G*\nabla\cdot(V\otimes V)+V_0,
    \end{align*} 
    on $L^{3,\infty}$ maps $L^{3,\infty}_{ss}$ to $L^{3,\infty}_{ss}$. 
    Suppose $V$ is $(-1)$-homogeneous. 
    Then $\nabla \cdot (V\otimes V)$ is $(-3)$-homogeneous. 
    The Oseen tensor $G$ is $(-1)$-homogeneous and hence so is $G*\nabla\cdot (V\otimes V)$.
\end{proof}

\section{Asymptotic stability for stationary flows}
\label{sec::StationaryStability}

Fix a domain $\Omega\subseteq\R^3$ and a function $V\in L^{3,\infty}(\Omega)$.  
Recall that 
\begin{align*}
    L^2_\sigma(\Omega) = \overline{C_{c,\sigma}^\infty(\Omega)}^{L^2(\Omega)},
    \qquad 
    H^1_{0,\sigma}(\Omega) = \overline{C_{c,\sigma}^\infty(\Omega)}^{H^1(\Omega)}.
\end{align*} 
Define the sesquilinear form $a_{\cL}(\cdot,\cdot):H_{0,\sigma}^1(\Omega)\times H_{0,\sigma}^1(\Omega)\to\R$ by
\begin{align*}
    a_\cL(u,w) = \int_\Omega \nabla u\cdot\nabla w\,dx - \int_\Omega (u\cdot\nabla)w\cdot V\,dx + \int_\Omega (V\cdot\nabla)u\cdot w\,dx
\end{align*}
We will prove that $a_\cL$'s associated linear operator generates an analytic contraction semigroup on $L^2_\sigma(\Omega)$ using the following proposition.

\begin{proposition}\label{prop::analyticsemigroup}
    \cite[Propositions 1.51 and 1.52]{Ouhabaz}.
    Let $H_1$ and $H_2$ be real Hilbert spaces. 
    Let $a(\cdot,\cdot):H_1\times H_1\to \R$ be a sesquilinear form. 
    Suppose that the following conditions are met:
    \begin{itemize}
        \item $a$ is densely defined, i.e., $H_1$ is dense in $H_2$.
        \item $a$ is accretive, i.e., $a(u,u)\geq 0$ for all $u\in H_1$.
        \item $a$ is continuous, i.e., there exists a constant $M>0$ such that 
        \begin{align*}
            |a(u,w)|\leq M\|u\|_a\|w\|_a,
        \end{align*}
        for all $u,w\in H_1$ where $\|u\|_a:=\sqrt{a(u,u)+\|u\|_{H_1}^2}$.
        \item $a$ is closed, i.e., $(H_1,\|\cdot\|_a)$ is a Banach space.
    \end{itemize}
    Let $\cL$ be the operator associated to $a(\cdot,\cdot)$. 
    That is, $\cL:D(\cL)\to H_2$ where $D(\cL)$ is the set of $u\in H_1$ such that there exists a $v\in H_2$ satisfying $a(u,\varphi)=\langle v,\varphi\rangle_{H_2}$ for all $\varphi\in H_1$ and $\cL u=v$. Then $-\cL$ is the generator of an analytic contraction semigroup on $H_2$.
\end{proposition}

This proposition is quite robust with respect to the domains allowed. 
Our setup is identical to the approach taken in \cite{KarchPil}. 
Because we are working with generic fields $V\in L^{3,\infty}(\Omega)$ our proof of the next theorem differs compared to \cite{KarchPil} where they use Hardy's inequality. 

\begin{theorem}
    Let $\Omega\subseteq\R^3$ be any domain. 
    Assume $V\in L^{3,\infty}(\Omega)$ is divergence-free in the distributional sense.
    There exists $\epsilon_*>0$ so that, if
    \begin{align*}
        \|V\|_{L^{3,\infty}(\Omega)} < \epsilon_*,
    \end{align*}
    then the sesquilinear form $a_\cL(u,w)$ and its adjoint $a^*_\cL(u,w)=a_\cL(w,u)$ are densely defined accretive continuous closed forms on $L^2_\sigma(\Omega)$. 
    Thus the associated linear operators $\cL$ and $\cL^*$ generate strongly continuous contractive semigroups on $L^2_\sigma(\Omega)$ which are analytic on the complexification of $L^2_\sigma(\Omega)$. 
    We will notate these as $e^{-t\cL}$ and $e^{-t\cL^*}$ respectively.
\end{theorem}

\begin{proof}
    We apply Proposition \ref{prop::analyticsemigroup} with $H_1=H_{0,\sigma}^1(\Omega)$ and $H_2=L^2_\sigma(\Omega)$. 
    Density is clear. 
    Since $V$ is divergence free and $u$ and $w$ are limits of compactly supported test functions,
    \begin{align*}
        \int_{\Omega} (V\cdot\nabla)u\cdot w\,dx
        =-\int_{\Omega}(V\cdot\nabla)w\cdot u\,dx.
    \end{align*}
    So if $u=w$, then, 
    \begin{align*}
        a_\cL(u,u) = \int_\Omega |\nabla u|^2\,dx - \int_\Omega (u\cdot\nabla)u\cdot V\,dx.
    \end{align*}
    We have the estimate
    \begin{equation}
        \label{eq::SesqulinearEstimate}
        \|(u\cdot \nabla)u\cdot V\|_{L^1(\Omega)}
        \leq  \|u\|_{L^{6,2}(\Omega)} \|V\|_{L^{3,\infty}(\Omega)} \|\nabla u\|_{L^2(\Omega)}
        \leq C\|\nabla u\|_{L^2(\Omega)}^2\|V\|_{L^{3,\infty}(\Omega)},
    \end{equation} 
    by the Lorentz-Sobolev embedding $H^1_{0}\hookrightarrow L^{6,2}$ stated in Theorem \ref{thm::LorentzInequalities}c. 
    So 
    \begin{align*}
        -\int_\Omega (u\cdot\nabla)u\cdot V\,dx
        \geq -\|(u\cdot\nabla)V\cdot u\|_{L^1(\Omega)}
        \geq - C\|\nabla u\|_{L^2(\Omega)}^2\|V\|_{L^{3,\infty}(\Omega)}.
    \end{align*}
    Let $\epsilon_*<\frac{1}{C}$. 
    Then 
    \begin{equation}
        \label{eq::SesquilinearBound}
        a_\cL(u,u) 
        = \|\nabla u\|_{L^2(\Omega)}^2-\int_\Omega (u\cdot\nabla)u\cdot V\,dx
        \geq (1-C\|V\|_{L^{3,\infty}(\Omega)})\|\nabla u\|_{L^2(\Omega)}^2.
    \end{equation}
    Assuming $\epsilon_*<\frac{1}{C}$ we have proven $a_\cL(u,w)$ is accretive. 
    Continuity follows from 
    \begin{align*}
        |a_\cL(u,w)|
        &\leq \|\nabla u\cdot \nabla w\|_{L^1(\Omega)} + \|(u\cdot \nabla)w\cdot V\|_{L^1(\Omega)} + \|(V\cdot\nabla)u\cdot w\|_{L^1(\Omega)}\\
        &\leq \|\nabla u\|_{L^2(\Omega)}\|\nabla w\|_{L^2(\Omega)} + 2C\|V\|_{L^{3,\infty}(\Omega)}\|\nabla u\|_{L^2(\Omega)}\|\nabla w\|_{L^2(\Omega)}
    \end{align*}
    by estimate \eqref{eq::SesqulinearEstimate} and an analogous estimate for the third term. 
    By \eqref{eq::SesquilinearBound} the sesquilinear norm $\|\cdot\|_a$ is equivalent to the usual norm on $H^1_{0,\sigma}$ and hence completeness is automatic.  
    Applying Proposition \ref{prop::analyticsemigroup} proves the theorem for $a_\cL$.
    As $a^*_\cL(u,w)= a_{\cL}(w,u)$, the preceding calculations apply to $a^*_\cL$.
    Applying Proposition \ref{prop::analyticsemigroup} completes the proof.
\end{proof}

The next lemmas are contained in \cite{KarchPil} and are essentially domain independent properties of the semigroups in view. 
We have included complete proofs to make this clear, but do not claim any new insights compared to \cite{KarchPil}.

\begin{lemma}
    \label{lem::L2-L2Estimate}
    There exists a constant $C>0$ such that 
    \begin{align*}
        \|\nabla e^{-t\cL^*}\psi\|_{L^2(\Omega)} \leq Ct^{-1/2}\|\psi\|_{L^2(\Omega)},
    \end{align*}
    for any $\psi\in L^2_\sigma(\Omega)$.
\end{lemma}

\begin{proof}
    Applying the adjoint version of \eqref{eq::SesquilinearBound} with $u=e^{-t\cL^*}\psi$ gives 
    \begin{align*}
        \|\nabla e^{-t\cL^*}\psi\|_{{L^2(\Omega)}}^2
        \leq \frac{1}{1-C\|V\|_{L^{3,\infty}}} a_{\cL^*}(e^{-t\cL^*}\psi,e^{-t\cL^*}\psi).
    \end{align*}
    Since analytic semigroups map into the domain of the infinitesimal generator for $t>0$, see, e.g. \cite[Theorem 4.6]{EngelNagel}, we have $e^{-t\cL^*}\psi\in D(\cL^*)$ for all $t>0$ and, furthermore,  
    \begin{align*}
        \|\cL^* e^{-t\cL^*}\psi\|_{L^2(\Omega)} 
        \leq 
        Ct^{-1}\|\psi\|_{L^2(\Omega)}.
    \end{align*} 
    By Cauchy-Schwarz 
    \begin{align*}
        a_{\cL^*}(e^{-t\cL^*}\psi,e^{-t\cL^*}\psi)
        =(\cL^* e^{-t\cL^*}\psi,e^{-t\cL^*}\psi)
        \leq \|\cL^* e^{-t\cL^*}\psi\|_{L^2(\Omega)}\|e^{-t\cL^*}\psi\|_{L^2(\Omega)}
        \leq Ct^{-1}\|\psi\|_{L^2(\Omega)}^2.
    \end{align*}
    This shows
    \begin{align*}
        \|\nabla e^{-t\cL^*}\psi\|_{{L^2(\Omega)}}^2
        \leq Ct^{-1}\|\psi\|_{L^2(\Omega)}^2.
    \end{align*}
    Taking the square root on the above inequality completes the proof.
\end{proof}

\begin{lemma}
    \label{lem::decayOfAverage}
    We have that 
    \begin{align*}
        \lim_{t\to\infty} \frac{1}{t}\int_0^t \|e^{-s\cL}\psi\|_{L^2(\Omega)}ds = 0,
    \end{align*}
    for all $\psi\in L^2_\sigma(\Omega)$.
\end{lemma}

\begin{proof}
    By simple reasoning, it suffices to show that 
    \begin{align*}
        \lim_{t\to\infty}\|e^{-t\cL}\psi\|_{L^2(\Omega)}=0.
    \end{align*} 
    We know that $a_\cL^*(u,u)\geq \alpha \|\nabla u\|_{L^2(\Omega)}^2$. 
    So $\cL^* u$ has trivial kernel. 
    Indeed, if $\cL^* u=0$ then $a_{\cL^*}(u,u)=\langle 0,u\rangle=0$. 
    But, $a(u,u)\geq \alpha\|\nabla u\|_{L^2(\Omega)}$. 
    Thus $\nabla u=0$ and hence $u=0$. 
    A standard fact from Hilbert space theory says 
    \begin{align*}
        \overline{\text{Range}(\cL)}=\text{Ker}(\cL^*)^\perp=L^2_\sigma(\Omega).
    \end{align*} 
    Fix $\varphi\in D(\cL)$. 
    Since $e^{-t\cL}$ is a contraction semigroup,
    \begin{align*}
        \|e^{-t\cL}\psi\|_{L^2(\Omega)} 
        \leq \|e^{-t\cL}(\psi-\cL\varphi)\|_{L^2(\Omega)} + \|\cL e^{-t\cL}\varphi\|_{L^2(\Omega)}
        \leq \|\psi-\cL\varphi\|_{L^2(\Omega)} + Ct^{-1}\|\varphi\|_{L^2(\Omega)}
    \end{align*}
    Choosing $t$ arbitrarily large and, by density, $\|\psi-\cL\varphi\|_{L^2(\Omega)}$ arbitrarily small completes the proof.
\end{proof}

We are now ready to prove Theorem \ref{mainthm1}.  Our proof follows the logic of \cite[Proof of Theorem 2.2]{KarchPil}, but we include some additional details for the sake of completeness. 
  
\begin{proof}[Proof of Theorem \ref{mainthm1}]
    We will first show 
    \begin{align*}
        \lim_{t\to\infty} \frac{1}{t} \int_0^t \|u(s)\|_{L^2(\Omega)} \,ds = 0.
    \end{align*} 
    We prove this by first establishing\footnote{This is essentially asserting that $u$ is a mild solution. One can prove this for every $s$ but an approximation argument seems necessary which we omit for the sake of brevity.} for almost every $s$
    \begin{equation}
        \label{eq::samesies}
        \int_\Omega u(s)\cdot \psi dx=\int_\Omega e^{-s\cL}u_0\cdot \psi dx-\int_0^s\int_\Omega (u\cdot\nabla )u\cdot e^{-(s-\tau)\cL^*}\psi dxd\tau.
    \end{equation}
    We establish \eqref{eq::samesies} by testing \eqref{pns} against backwards propagating test functions. 
    We cannot test against $\varphi(\tau)=e^{-(s-\tau)\cL^*}\psi$ directly since
    \begin{equation}
        \label{eq::samesiesBadTerm}
        \int_0^s \int_\Omega u(\tau)\partial_\tau \varphi(\tau)\,dx \,d\tau,
    \end{equation}
    might diverge because the bound for the semi-group term is $O(1/\tau)$.
    For each $\epsilon>0$ and $\psi\in L^2_\sigma(\Omega)$, define the function $\varphi_\epsilon(\tau)=e^{-(s-\tau+\epsilon)\cL^*}\psi$. 
    We will eventually remove \eqref{eq::samesiesBadTerm} from view and then send $\epsilon$ to zero.
    By \cite[Lemma II.1.3]{EngelNagel}
    \begin{align*}
        a_\cL(u(\tau),\varphi_\epsilon(\tau))
        =\int_\Omega u(\tau)\cL^*\varphi_\epsilon(\tau) \,dx
        =\int_\Omega u(\tau)\partial_\tau \varphi_\epsilon(\tau) \,dx,
    \end{align*}
    for almost every $\tau\in [0,s]$, where the ``almost every'' caveat reflects the fact that $u\in H^1$ for almost every $\tau\in [0,s]$. 
    Here we used that $\varphi_\epsilon\in D(\cL^*)$ which requires $s-\tau+\epsilon>0$. 
    So \eqref{eq::pnsDistributionalSolution} with $s=0$ and $t=s$, i.e.,
    \begin{align*}
           \int_\Omega u(s)\cdot \varphi_\epsilon(s) \,dx 
        &+ \int_0^s\int_\Omega \nabla u: \nabla \varphi_\epsilon \,dx \,d\tau 
         + \int_0^s\int_\Omega (u\cdot\nabla)u\cdot\varphi_\epsilon \,dx \,d\tau
         - \int_0^s\int_\Omega (u\cdot\nabla)\varphi_\epsilon \cdot V \,dx \,d\tau\\
        &+ \int_0^s\int_\Omega (V\cdot\nabla)u\cdot\varphi_\epsilon \,dx \,d\tau
         = \int_\Omega u_0\cdot \varphi_\epsilon(0) \,dx
         + \int_0^s \int_\Omega u\cdot \partial_\tau \varphi_\epsilon \,dx \,d\tau,
    \end{align*}
    becomes
    \begin{align*}
          \int_\Omega u(s)\cdot e^{-\epsilon\cL^*}\psi \,dx
        + \int_0^s\int_\Omega (u\cdot\nabla )u\cdot e^{-(s-\tau+\epsilon)\cL^*}\psi \,dx \,d\tau
        = \int_\Omega u_0\cdot e^{-(s+\epsilon)\cL^*}\psi \,dx,
    \end{align*}
    where we used the identity
    \begin{align*}
          \int_0^s a_\cL(u,\varphi_\epsilon) \,d\tau
        = \int_0^s\int_\Omega \nabla u:\nabla \varphi_\epsilon \,dx \,d\tau
        - \int_0^s\int_\Omega (u\cdot\nabla)\varphi_\epsilon\cdot V \,dx \,d\tau
        + \int_0^s\int_\Omega (V\cdot\nabla)u\cdot \varphi_\epsilon \,dx \,d\tau.
    \end{align*}
    By adjointness, this is the same as
    \begin{align*}
          \int_\Omega u(s)\cdot e^{-\epsilon\cL^*}\psi \,dx
        = \int_\Omega e^{-s\cL}u_0\cdot e^{-\epsilon\cL^*}\psi \,dx
        - \int_0^s\int_\Omega (u\cdot\nabla )u\cdot e^{-(s-\tau+\epsilon)\cL^*}\psi \,dx \,d\tau.
    \end{align*}
    We would like to send $\epsilon$ to zero to obtain \eqref{eq::samesies} for almost every $s>0$ and presently explain why this is justified. 
    For the non-time-integrated terms this follows from the strong continuity properties of the semi-group and the fact that other terms are in $L^2(\Omega)$.
    For the nonlinear term we use the dominated convergence theorem in the time variable noting that,
    \begin{align*}
               \left| \int_\Omega (u\cdot\nabla )u\cdot e^{-(s-\tau+\epsilon)\cL^*}\psi\,dx \right|
        & \leq \|u(\tau)\|_{L^6(\Omega)}\|\nabla u(\tau)\|_{L^2(\Omega)} \|e^{-(s-\tau+\epsilon)\cL^*}\psi\|_{L^3(\Omega)}\\
        & \leq C\|\nabla u(\tau)\|_{L^2(\Omega)}^2 \|e^{-(s-\tau+\epsilon)\cL^*}\psi\|^{1/2}_{L^2(\Omega)} \|e^{-(s-\tau+\epsilon)\cL^*}\psi\|^{1/2}_{L^6(\Omega)}\\
        & \leq C\|\nabla u(\tau)\|_{L^2(\Omega)}^2 \|e^{-(s-\tau+\epsilon)\cL^*}\psi\|^{1/2}_{L^2(\Omega)} \|\nabla e^{-(s-\tau+\epsilon)\cL^*}\psi\|^{1/2}_{L^2(\Omega)}\\
        & \leq C\|\nabla u(\tau)\|^2_{L^2(\Omega)} (s-\tau+\epsilon)^{-1/4} \|\psi\|_{L^2(\Omega)}\\
        & \leq C\|\nabla u(\tau)\|^2_{L^2(\Omega)} (s-\tau)^{-1/4} \|\psi\|_{L^2(\Omega)}.
    \end{align*}
    This is not necessarily in $L^1(0,s)$ for every $s$. 
    However, we know that $\|\nabla u(\tau)\|^2_{L^2(\Omega)}\in L^1(\R^+)$ because $u$ is in the energy class. 
    For each $T>0$, therefore,
    \begin{align*}
            \int_0^T\int_0^s \|\nabla u(\tau)\|^2_{L^2(\Omega)}(s-\tau)^{-1/4} \,d\tau \,ds 
        & = \int_0^T \|\nabla u(\tau)\|^2_{L^2(\Omega)}\int_\tau^T (s-\tau)^{-1/4} \,ds \,d\tau \\
        & = \frac{4}{3}\int_0^T \|\nabla u(\tau)\|^2_{L^2(\Omega)}(T-\tau)^{3/4} \,d\tau 
          < \infty,
    \end{align*}
    and hence $\|\nabla u(\tau)\|^2(s-\tau)^{-1/4}\in L^1(0,s)$ for a.e.~$s$. 
    This justifies the application of the dominated convergence theorem and confirms that \eqref{eq::samesies} holds for a.e.~$s$. 
    
    We now work toward $L^2(\Omega)$ decay. 
    Note that for each $\tau$, 
    \begin{align*}
            \left|\int_\Omega(u\cdot\nabla)u\cdot e^{-(s-\tau)\cL^*}\psi \,dx \right|
        & = \left| \int_\Omega u\otimes  u\cdot \nabla e^{-(s-\tau)\cL^*}\psi \,dx \right|\\
        & \leq \|u\|_{L^4(\Omega)}^2 \|\nabla e^{-(s-\tau)\cL^*}\psi\|_{L^2(\Omega)}\\
        & \leq C(s-\tau)^{-\frac{1}{2}} \|u\|_{L^6(\Omega)}^\frac{3}{2} \|u\|_{L^2(\Omega)}^\frac{1}{2} \|\psi\|_{L^2(\Omega)}\\
        & \leq C(s-\tau)^{-\frac{1}{2}} \|\nabla u\|_{L^2(\Omega)}^\frac{3}{2} \|u\|_{L^2(\Omega)}^\frac{1}{2} \|\psi\|_{L^2(\Omega)},
    \end{align*}
    by the Sobolev embedding $H^1_0(\Omega)\hookrightarrow L^6(\Omega)$ and the standard interpolation inequality.
    The fact that 
    \begin{equation}
        \label{eq::L2SelfDual}
        \|f\|_{L^2(\Omega)} 
        = \sup_{\substack{ \psi\in L^2_\sigma(\Omega)\\ \|\psi\|_{L^2(\Omega)}=1}} \left|\int f\psi \,dx\right|,
    \end{equation}
    for all $f\in L^2_\sigma(\Omega)$ yields,
    \begin{align*}
        \|u(s)\|_{L^2(\Omega)}
        &\leq \|e^{-s\cL}u_0\|_{L^2(\Omega)} + C\int_0^s (s-\tau)^{-\frac{1}{2}}\|\nabla u\|_{L^2(\Omega)}^\frac{3}{2}\|u\|_{L^2(\Omega)}^\frac{1}{2} \,d\tau\\
        &\leq \|e^{-s\cL}u_0\|_{L^2(\Omega)} + C\int_0^s (s-\tau)^{-\frac{1}{2}}\|\nabla u\|_{L^2(\Omega)}^\frac{3}{2} \,d\tau,
    \end{align*}
    because $u$ is in the energy class $L_t^\infty L_x^2$. 
    Here $C=C(\|u_0\|_{L^2})$. 
    So the average of the energy satisfies
    \begin{align*}
        \frac{1}{t}\int_0^t \|u(s)\|_{L^2(\Omega)}ds
        &\leq \frac{1}{t}\int_0^t\|e^{-s\cL}u_0\|_{L^2(\Omega)}ds + \frac{C}{t}\int_0^t\left(|\cdot|^{-\frac{1}{2}}*\|\nabla u(\cdot)\|_{L^2(\Omega)}^\frac{3}{2}\right)(s)ds\\
        &\leq \frac{1}{t}\int_0^t\|e^{-s\cL}u_0\|_{L^2(\Omega)}ds + Ct^{-\frac{1}{4}},
    \end{align*}
    by Young's convolution inequality (see \cite[Lemma 5.1]{KarchPil}). 
    So the average of the energy goes to zero by \eqref{lem::decayOfAverage}, implying there exists a strictly increasing and unbounded sequence of times $t_k$ so that $\|u(t_k)\|_{L^2}\to 0$.
    Monotonicity of the $L^2$-norm proves $L^2$-decay of $u$.
\end{proof}

\section{Asymptotic stability for Type II.5 solutions under axisymmetric perturbations}
\label{sec::TypeII.5Stability}

In this section we will prove some results for $\R^3$ and $\R^3_+$ but some only for $\R^3$.
This is because, while there are barriers to some results in $\R^3_+$, we would ultimately like to study the problem in $\R^3_+$ as well and hope that these remarks are useful in that future endeavor.

Let $\Omega\subseteq\R^3$ be the whole-space $\R^3$ or the half-space $\R^3_+$.
For an element $x\in\R^3$ let $x=(x_1,x_2,x_3)$ and $x_h=(x_1,x_2)$.
Let $I=\{x\in\Omega:x_h=0\}$.
Throughout this section we assume $V$ is a Type II.5 singular solution as defined in Section \ref{subsec::Examples::TypeII.5Solutions}.
Let $V_1$, $V_2$ and $V_3$ be the components of $V$ in cartesian coordinates.
Let $V_h=(V_1,V_2)$.
We mentioned in Section \ref{subsec::Examples::TypeII.5Solutions} that Type II.5 solutions can be decomposed in a certain way.
The following lemma clarifies this.

\begin{lemma}
    \label{lem::DecompositionOfTypeII.5}
    Assume $V$ is a Type II.5 solution to \eqref{sns} on $\R^3\setminus I$. 
    Then, $V_3\in L^{3,\infty}(\R^3)$ and for  $i=1,2$ and $x\in\R^3\setminus I$,
    \begin{align*}
        V_i(x)=a_i(x)+b_i(x),
    \end{align*}
    where $a_i$ satisfies
    \begin{align*}
        |\partial_{x_3}a_i|(x)=O(|x|^{-2}),
    \end{align*}
    and $b_i\in L^{3,\infty}(\R^3)$.
\end{lemma}

While the lemma is stated for $\R^3$, it applies also to $\R^3_+$. The same is true of all arguments in this section except for the proof of Theorem \ref{mainthm2}.

\begin{proof}
    By the definition of Type II.5 solutions, there exists $r>0$ where the assumed expansions hold in the neighborhood $\{x_3=\pm 1,|x_h|<r\}$.
    Without loss of generality we only consider the positive part of the $x_3$-axis.
    Let $\varphi(x_1):\R\to[0,\infty)$ be a smooth, even cut-off function which is $1$ for $|x_1|\leq r/3$, $0$ for $|x_1|\geq 2r/3$ and is decreasing in $|x_1|$.

    Now consider for $x_1\geq 0$
    \begin{align*}
        \tilde{a}_i(x_1,0,x_3)
        =\frac{1}{x_3}V_i(x_1/x_3,0,1)\varphi(x_1/x_3)
        =\frac{c_{i,+}}{x_1}\varphi(x_1/x_3) + \underbrace{\frac{1}{x_3}\varphi(x_1/x_3) O((x_1/x_3)^{-2/3+\delta})}_{=:\text{LOT}(x_1,0,x_3)}.
    \end{align*}
    Axisymmetry then determines $\tilde{a}_i(x)$ and $\text{LOT}(x)$ for all $x\in\R^3_+$.
    By Lemma \ref{lem::TypeIISolutionsAreInWeakL3}, 
    \begin{align*}
        \text{LOT}\in L^{3,\infty}(\R^3)
    \end{align*}
    and will become part of $b_i$. 
    Let
    \begin{align*}
        a_i=\tilde{a}_i-\text{LOT}.
    \end{align*}
    We now have for $x_2=0$ and $x_1>0$ that
    \begin{align*}
        \partial_{x_3}\frac{c_{i,+}}{x_1}\varphi(x_1/x_3)
        =-\frac{c_{i,+}}{x_1}\varphi'(x_1/x_3)\frac{x_1}{x_3^2}.
    \end{align*}
    We have $x_1/x_3\leq 2r/3$ implying $x_1\leq 2rx_3/3$ and therefore $|x|\sim x_3$.
    Hence we have 
    \begin{align*}
        \partial_{x_3}\frac{c_{i,+}}{|x_h|}\varphi(|x_h|/x_3)\lesssim \frac{c_{i,+}}{|x|^2}
    \end{align*}
    for $x=(x_1,0,x_3)$.
    By axisymmetry we can rotate this to obtain a bound for $x\in \R^3_+$. 
    An identical bound follows in $\R^3_-$ with $c_{i,+}$ replaced by $c_{i,-}$.

    It is easy to see that $b_i=V_i-a_i\in L^{3,\infty}(\R^3)$ because it is self-similar and is equal to a part defined above by $\text{LOT}$ and another part which is bounded on the unit sphere.
\end{proof}

The following lemma simply affirms that an endpoint version of the $L^2(\R^2)$-Hardy's inequality holds for axisymmetric vector fields.
The failure of this endpoint case for generic functions is a major obstacle in proving stability of Type III solutions in general.

\begin{lemma}
    \label{lem::EndpointHardysInequality}
    Let $u\in H^1_0(\Omega)$ be an axisymmetric vector field where $\Omega=\R^3$ or $\R^3_+$.
    Then,
    \begin{align*}
        \int_\Omega \frac{|u_h|^2}{|x_h|^2}\,dx\leq \|\nabla u\|^2_{L^2(\Omega)}.
    \end{align*}
\end{lemma}

\begin{proof}
    In cylindrical coordinates $(\rho,\phi,z)$, $|x_h|=\rho$.
    By axisymmetry
    \begin{align*}
        \partial_\phi u_1
        &= \partial_\phi(u_\rho\cos(\phi)-u_\phi\sin(\phi))
        = -u_\rho\sin(\phi)-u_\phi\cos(\phi)
        = -u_2\\
        \partial_\phi u_2
        &= \partial_\phi(u_\rho\sin(\phi)+u_\phi\cos(\phi))
        = u_\rho\cos(\phi)-u_\phi\sin(\phi)
        =u_1.
    \end{align*}
    So
    \begin{align*}
        \int_\Omega \frac{|u_h|^2}{|x_h|^2}\,dx
        = \int_\Omega \frac{|u_1|^2}{|x_h|^2}\,dx + \int_\Omega \frac{|u_2|^2}{|x_h|^2}\,dx 
        = \int_\Omega \frac{|\partial_\phi u_2|^2}{|x_h|^2}\,dx + \int_\Omega \frac{|\partial_\phi u_1|^2}{|x_h|^2}\,dx
        \leq \int_\Omega |\nabla_h u_h|^2\, dx,
    \end{align*}
    where we used 
    \begin{align*}
        \nabla_h f=\partial_\rho fe_\rho+\frac{\partial_\phi f}{\rho}e_\phi,
    \end{align*}
    which implies
    \begin{align*}
        |\nabla_h f|^2=|\partial_\rho f|^2+\frac{|\partial_\phi f|^2}{\rho^2}.
    \end{align*}
    Since $|\nabla_hu_h|\leq |\nabla u|$, the proof is complete.
\end{proof}

\begin{remark}
    \label{rem::OddEndpointHardy}
    Another way to prove this is to use the fact that if $\phi\in C_c^\infty((0,\infty))$ then
    \begin{align*}
        \int_0^\infty \frac{\phi^2}{x^2}\, dx
        \leq 4\int_0^\infty \left|\frac{d\phi}{dx}\right|^2\,dx.
    \end{align*}
    This clearly extends to functions in $H^1_0((0,\infty))$ by a density argument.
    If $f\in H^1_0(\R^3)$ is odd in $x_1$, then, for fixed $x_2$ and $x_3$, $f\in H^1_0((0,\infty))$ and $H^1_0((-\infty,0))$.
    So
    \begin{align*}
        \int_{\R^3}\frac{f^2}{|x_h|^2}\, dx
        &\leq \int_{x_3\in\R}\int_{x_2\in\R}\int_{x_1\in\R}\frac{f^2}{x_1^2}\, dx_1\, dx_2\, dx_3\\
        &\leq 4\int_{x_3\in\R}\int_{x_2\in\R}\int_{x_1\in\R}\left|\frac{\partial f}{\partial x_1}\right|^2\, dx_1\, dx_2\, dx_3
        \leq 4\|f\|^2_{\dot{H}^1}.
    \end{align*}
    Since the horizontal components of axisymmetric vector fields can be decomposed into parts that are odd in $x_1$ or $x_2$, the result follows for axisymmetric fields.
    More importantly, however, this remark also shows it holds for vector fields whose horizontal components have odd parity in either of the horizontal variables.
\end{remark}

We aim to prove asymptotic stability of small Type II.5 solutions.
We introduce the following notion of boundedness based on the decomposition above.

\begin{definition}
    \label{def::TypeII.5Boundedness}
    For a fixed $M>0$, we say $V$ is $M$-bounded in the Type II.5 sense if $|V(x)|\leq M|x_h|^{-1}$, $\|V_3\|_{L^{3,\infty}(\R^3)}<M$, $\sup_{i=1,2}\|b_i\|_{L^{3,\infty}(\R^3)}<M$ and
    \begin{align*}
        \sup_{i,j\in\{1,2\}}\||x_h|^2\partial_ia_j\|_{L^\infty(\R^3)}+\sup_{j\in\{1,2\}}\||x|^2\partial_3a_j\|_{L^\infty(\R^3)}<M.
    \end{align*}
\end{definition}

Plainly Lemma \ref{lem::DecompositionOfTypeII.5} implies any Type II.5 solution is $M$-bounded for some $M>0$. 
This definition can be used to encode the smallness into the problem.

The next two propositions illustrate how this boundedness allows for key integral estimates.

\begin{proposition}
    \label{prop::TypeII.5uVvEstimate} 
    Let $\Omega=\R^3$ or $\Omega=\R^3_+$.
    Suppose that $V$ is a $(-1)$-homogeneous axisymmetric solution to the \eqref{sns} which is Type II.5 singular.  
    For a fixed $M>0$, suppose that $V$ is $M$-bounded in the Type II.5 sense.
    Then,
    \begin{align*}
        {|I_0|:=}\bigg|\int_\Omega (u\cdot\nabla)v\cdot V \,dx \bigg|\leq C M\|\nabla u\|_{L^2(\Omega)}\|\nabla v\|_{L^2(\Omega)},
    \end{align*}
    whenever { $u\in  H^1_0(\Omega)$, $v\in H^2(\Omega)$} and $u$ and $v$ are axisymmetric and divergence free. 
    Here, $C $ is a universal constant.
\end{proposition}

\begin{proof}
    Let{
    \begin{align*}
        I_{ij}=\int_\Omega u_i\partial_i v_j V_j\,dx.
    \end{align*} }
    Then $|I_0|\leq 9 \sup_{i,j\in \{1,2,3\}} |I_{ij}|$.
    
    If $j=3$, then   
    \begin{equation}
        \label{eq::TypeII.5Ii3Estimate}
        \begin{aligned}
            \left|\sum_{i=1}^3 I_{i3}\right|
            &\leq C\|\nabla u\|_{L^2}\|\nabla v\|_{L^2}\|V_3\|_{L^{3,\infty}}
            \\&\leq M C\|\nabla u\|_{L^2}\|\nabla v\|_{L^2},
        \end{aligned}
    \end{equation} 
    by Lemma \ref{lem::DecompositionOfTypeII.5} and Lorentz-Sobolev.
   
    We now bound $I_{3j}$ for $j=1,2$. We know that $V_j=a_j+b_j$ for $b_j\in L^{3,\infty}$ with $\|b_j\|_{L^{3,\infty}}\leq CM$ by Lemma \ref{lem::DecompositionOfTypeII.5}. 
    We need to establish the bound
    \begin{align*}
        |I_{3j}|
        &\leq{\left|\int_\Omega (\partial_3 u_3) a_j v_j\,dx\right|+} \left|\int_\Omega u_3(\partial_3 a_j) v_j\,dx\right|
        + \left|\int_\Omega u_3 b_j \partial_3v_j\,dx\right|,
    \end{align*}
    which amounts to justifying integration by parts for $\int_\Omega u_3 a_j\partial_3 v_j\,dx$. 
    Let $v^{(k)}$ be smooth, compactly supported axisymmetric vector fields that approximate $v$ in {$H^2(\Omega)$. 
    Then 
    \begin{align*}
        \int_\Omega u_3a_j\partial_3 v_j^{(k)}\,dx=-\int_\Omega (\partial_3 u_3)a_jv_j^{(k)}\,dx-\int_\Omega u_3 (\partial_3 a_j) v_j^{(k)}\,dx.
    \end{align*}}
    As for the limit we have 
    \begin{align*}
    \int_\Omega u_3(\partial_3 a_j) (v_j^{(k)}- v_j)\,dx\leq C \| u_3\|_{\dot H^1} \|v_j^{(k)}- v_j\|_{H^2}\to 0,
    \end{align*}
    by the $3D$ Hardy inequality{, 
    \begin{align*}
        \int_\Omega (\partial_3 u_3)a_j(v_j^{(k)}-v_j)\,dx\leq C \| u_3\|_{\dot H^1} \|v_j^{(k)}- v_j\|_{H^2}\to 0
    \end{align*}
    by Lemma \ref{lem::EndpointHardysInequality}, and 
    \begin{align*}
        \int_\Omega u_3a_j \partial_3(v_j^{(k)}-v_j)\,dx\leq C\|u_3\|_{L^2}\|v_j^{(k)}-v_j\|_{H^2}\to 0
    \end{align*}
    by Remark \ref{rem::OddEndpointHardy}. Note that the assumption $v\in H^2$ is only used to justify integration by parts. It does not show up in the bounds of the theorem.}
    Returning to our bound on $|I_{3j}|$, the second term is   bounded by the $3D$ Hardy inequality, 
    \begin{align*}
        \left|\int_\Omega u_3\partial_3 a_j v_j \,dx\right|
        \leq M\bigg(\int_\Omega\frac{|u(x)|^2}{|x|^2}dx\bigg)^{\frac 1 2}\bigg(\int_\Omega\frac{|v(x)|^2}{|x|^2}dx\bigg)^{\frac 1 2}
        \leq CM\|\nabla u\|_{L^2}\|\nabla v\|_{L^2}.
    \end{align*}
    {The first term is similarly bounded by the Hardy inequality using $|a_j|\leq \frac{M}{|x_h|}$. The last term is bounded by the usual trilinear estimate using the fact that $b_j\in L^{3,\infty}$.} 

    The remaining indices are  $i,j\in\{1,2\}$. Observe that
    \begin{align*}
        |I_{ij}|\leq \left|\int_\Omega u_i \frac{M}{|x_h|}\partial_i v_j\,dx\right|+\left|\int_\Omega u_i  b_j\partial_iv_j\,dx\right|.
    \end{align*}
    The second term is bounded by $CM\|\nabla u\|_{L^2}\| \nabla v\|_{L^2}$ and the Hardy inequality for axisymmetric fields, Lemma \ref{lem::EndpointHardysInequality}, is used to obtain the same upper bound for the first term. 
\end{proof}

{The next proposition provides an estimate for a term arising from $\nabla \cdot (V\otimes u)$. 
This term typically vanishes in energy estimates but, because $V$ is not necessarily divergence-free across the singular axis, care needs to be taken. Note first that 
\[
\int_s^t\int  (\nabla \cdot (V\otimes u))\cdot u\,dx,d\tau,
\]
may not even converge. The next proposition will help explain how this issue can be avoided and supports our choice of the energy inequality mentioned in Theorem \ref{mainthm2}. We comment on this further in Remark \ref{rem::SpecExistenceTypeII.5}. 
}

\begin{proposition}
    \label{prop::TypeII.5Vuphi*uEstimate}
    Let $u\in H^2$  be divergence-free and axisymmetric.  
    Assume that  $V$ satisfies the assumptions of Proposition \ref{prop::TypeII.5uVvEstimate} {  and suppose additionally that $c_{1,\pm} \geq 0$.}
    Then
    {
    \begin{align*}
  { - }  \sum_{i,j=1}^3    \int_{\R^3} \partial_{x_i}( V_i u_j )  u_j\,dx
        \leq CM\|\nabla u\|^2_{L^2(\R^3)},
    \end{align*}
    where $C$ is a universal constant.
     }
\end{proposition}

\begin{proof}
    {For test functions $u\in C_c^\infty(\R^3;\R^3)$, we interpret the integral as the distributional pairing 
    $$-\sum_{i,j=1}^3 \int_{\R^3} \partial_{x_i}( V_i u_j )  u_j\,dx=-\langle \nabla\cdot (V\otimes u),u\rangle=\langle V\otimes u, \nabla u\rangle=\sum_{i,j=1}^3 J_{ij},$$
    where $J_{ij}=\int_{\R^3}  { V_i u_j \partial_{x_i} (  u_j)} \, dx$. We extend this identity to general $u$ by density.}   
  For $i=3$ we use the same argument as in \eqref{eq::TypeII.5Ii3Estimate}. 
    If $j=1,2$ we use Lemma \ref{lem::EndpointHardysInequality}.  In both cases the upper bound is $C M \|\nabla u\|_{L^2}^2$.
    The difficult term is when $i=1$ or $2$ and $j=3$ and we focus on it presently.

    { Decompose $V_h$ as $a_h+b_h$ as in Lemma \ref{lem::DecompositionOfTypeII.5}. We know 
    $$a_h=\frac{c_{1,\pm}(x_1,x_2)+c_{2,\pm}(-x_2,x_1)}{|x_h|^2}\varphi\left(\frac{|x_h|}{|x_3|}\right).$$
    and $b_h\in L^{3,\infty}$. So $J_h=J_{1,3}+J_{2,3}$ can be decomposed as 
    $$J_h=\frac{1}{2}\int_{\R^3} V_h\cdot \nabla_h u_3^2\, dx=\frac{1}{2}\int_{\R^3} a_h\cdot \nabla_h u_3^2\, dx+\frac{1}{2}\int_{\R^3} b_h\cdot \nabla_h u_3^2\, dx.$$ 
    The latter integral is bounded by the usual trilinear estimate. 
    For the former, we decompose the integral over cylinder's $T_\epsilon=\{|x_h|\leq\epsilon\}$ and integrate by parts in the exterior. That is,
    \begin{align*}
        \frac{1}{2}\int_{\R^3} a_h\cdot \nabla_h u_3^2\, dx
        &=-\frac{1}{2}\int_{\R^3\setminus T_\epsilon} (\nabla_h\cdot a_h)u_3^2\,dx
        +\frac{1}{2}\int_{T_\epsilon} a_h\cdot \nabla_h u_3^2\, dx\\
        &+\frac{1}{2}\int_{\partial T_\epsilon} a_h\cdot \nu_h\, u_3^2\, dS
        =:I_1+I_2+I_3,
    \end{align*}
    where $\nu_h=-\frac{x_h}{|x_h|}$ is the inward pointing unit normal to $\partial T_\epsilon$ (it is outward with respect to $\R^3\setminus T_\epsilon$).
    
    We bound each of the $I_j$ terms individually. By positivity of $c_{1,\pm}$, the boundary term satisfies 
    \begin{align*}
        \int_{\partial T_\epsilon}a_h\cdot \nu_h u_3^2\,dS=-\frac{c_{1,\pm}}{\epsilon}\int_{\partial T_\epsilon} \varphi\left(\frac{\epsilon}{|x_3|}\right)u_3^2\leq 0,
    \end{align*} 
    because $\varphi\geq 0$, provided the integral converges.  It will converge for a.e.~$\epsilon$ because $u_3\in L^2(\R^3)$. From here on we always assume $\epsilon$ comes from this set. 
    For the divergence term, note that
    \begin{align*}
        \nabla_h\cdot a_h=\frac{c_{1,\pm}(x_1,x_2)+c_{2,\pm}(-x_2,x_1)}{|x_h|^2}\varphi'\left(\frac{|x_h|}{|x_3|}\right)\cdot \frac{x_h}{|x_h||x_3|}=\frac{c_{1,\pm}}{|x_h||x_3|}\varphi'\left(\frac{|x_h|}{|x_3|}\right).
    \end{align*}
    Since $\varphi$ is supported where $|x_h|\sim |x_3|\sim |x|$, we apply the 3D Hardy inequality to bound
    \begin{align*}
        |I_1|\leq CM\int_{\R^3\setminus T_\epsilon} \frac{u_3^2}{|x|^2}\leq CM\|\nabla u\|_{L^2}^2.
    \end{align*}
    Now, because $u_3$ are even in the horizontal variable, $\partial_{x_i}u_3$ is odd. We can therefore use Remark \ref{rem::OddEndpointHardy} to conclude
    \[
        \frac{1}{2}\int_{\R^3} | V\cdot \nabla_h   u_3^2| \,dx \leq CM\| u_3 \|_{L^2}\|D^2  u_3\|_{L^2}<\infty.
    \]
    As this is finite by assumption, we see that $V_i ( u_3) \partial_{x_i}  u_3\in L^1(\R^3)$ and, therefore, 
    \[
    \frac{1}{2}\int_{T_{\epsilon}} V\cdot \nabla_h  u_3^2\, dx \to 0,
    \] 
    as $\epsilon\to 0$. Altogether, since the bound on $I_1$ is independent of $\epsilon$,
    \begin{align*}
        J_h\leq \frac{1}{2}\int_{\R^3}b_h\cdot\nabla_h u_3^2\,dx+ \lim_{\epsilon\to 0}(I_1+I_2)\leq CM\|\nabla u\|_{L^2}^2
    \end{align*}
    which proves the result.
    }
\end{proof}


In Theorem \ref{mainthm2} we require that the solution to \eqref{pns} satisfies the following generalized energy inequality which is based on \cite[(5.7)]{KarchPilSchonbek}. 
Let $E\in C^1([0,\infty))$ and $\psi(t)\in C^1([0,\infty);\mathcal S(\R^3))$ be arbitrary. 
A weak solution to \eqref{pns} is said to satisfy the \textbf{generalized energy inequality} provided
\begin{equation}
    \label{eq::GeneralizedEnergyInequality}
    \begin{aligned}
        E(t)\| \psi(t)*u(t)\|_{L^2}^2 
        & \leq  E(s) \| \psi(s)*u(s)\|_{L^2}^2 
          + \int_s^t E'(\tau)\|\psi(\tau) *u(\tau)\|_{L^2}^2 \,d\tau \\ 
        &\quad 
          +2 \int_s^t E(\tau )\bigg[ \langle \psi'(\tau)*u(\tau), \psi(\tau)*u(\tau)\rangle - \|\psi(\tau)*\nabla u(\tau)\|_{L^2}^2\bigg]\,d\tau \\
        &\quad 
          -2 \int_s^t E(\tau)  \int_{\R^3}\big( ((u\cdot \nabla u) \cdot \psi*\psi*u)(\tau)+ {  \sum_{i,j=1}^3(\partial_i( V_i u_j) \psi*\psi*u_j)(\tau)\big) \,dx  \,d\tau }\\
        &\quad 
          +2 \int_s^t E(\tau) \int_{\R^3}  ((u\cdot \nabla (\psi*\psi*u)) \cdot V)(\tau) \,dx  \,d\tau,
    \end{aligned}
\end{equation}
for almost all $s\geq 0$ including $s=0$ and all $t\geq s\geq 0$. { Note that compared to \cite[(5.7)]{KarchPilSchonbek} we have replaced  $  ((V\cdot \nabla u ) \cdot  \psi*\psi*u)$ with  $ \sum_{i,j=1}^3\partial_i( V_i u_j) \psi*\psi*u_j $.  

}

\begin{proof}[Proof of Theorem \ref{mainthm2}]
Let $M>0$ be such that $V$ is $M$-bounded in the Type II.5 sense. 
We follow \cite[Proof of Theorem 2.7. Second part]{KarchPilSchonbek}. 
Let  $\phi=\mathcal F^{-1}( e^{-|\cdot|^2})$. 
The idea in \cite{KarchPilSchonbek} is to bound the energy at time $t$ as follows
\begin{align*}
    \|u(\cdot,t)\|_{L^2}^2 \leq 2(\| \hat \phi \hat u(t)\|_{L^2}^2+ \| (1-\hat\phi)\hat u(t)\|_{L^2}^2). 
\end{align*}
Because the structure of our argument is similar to \cite[Proof of Theorem 2.7. Second part]{KarchPilSchonbek}, we will only highlight the differences, which involve terms from the generalized energy inequality containing $V$. 
The first step in \cite[Proof of Theorem 2.7. Second part]{KarchPilSchonbek} establishes the bound
\begin{align*}
    (2\pi)^\frac{3}{2}\| \hat \phi \hat u(t)\|_{L^2}^2 &\leq I_1(t,s) + C(\|u_0\|_{L^2}) \int_s^t \|\nabla u\|_{L^2}^2\,d\tau 
    \\&\quad +2(2\pi)^{-\frac{3}{2}}\bigg|\int_s^t    \int_{\R^3} {((\nabla\cdot(V\otimes u))} \cdot e^{2(t-\tau)\Delta} \phi*\phi*u)(\tau) \,dx  \,d\tau\bigg|
    \\&\quad +2 (2\pi)^{-\frac{3}{2}}\bigg|\int_s^t   \int_{\R^3}  ((u\cdot \nabla (e^{2(t-\tau)\Delta}\phi*\phi*u)) \cdot V)(\tau) \,dx  \,d\tau\bigg|,
\end{align*}
where 
\begin{align*}
    I_1(t,s) = \| e^{(t-s)\Delta} \phi * u(s)\|_{L^2}^2.
\end{align*}
This came from the generalized energy inequality with $E=1$ and $\psi =  e^{(t-\tau)\Delta}\phi$. 

For the {$\nabla\cdot(V\otimes u_i)$} term, when $i\neq 3$ we can use H\"older's inequality and Lemma \ref{lem::EndpointHardysInequality} to obtain an upper bound which is a multiple of $M\int_s^t \|\nabla u\|_{L^2}^2\,d\tau$. 
When $i=3$, we start by moving the derivative to the convolution term,
\begin{align*} 
    &\int_s^t    \int_{\R^3} {\nabla\cdot(V\otimes u_3)}  e^{2(t-\tau)\Delta} \phi*\phi*u_3(\tau) \,dx  \,d\tau 
    \\&= -\int_s^t    \int_{\R^3} (V\cdot \nabla e^{2(t-\tau)\Delta} \phi*\phi*u_3) \cdot u_3(\tau) \,dx  \,d\tau
    \\&=-\sum_{j=1}^3\int_s^t    \int_{\R^3} (V_j  (\partial_{j}e^{2(t-\tau)\Delta} \phi)*\phi*u_3))   u_3  \,dx  \,d\tau.
\end{align*}
When $j=3$ and since $V_3\in L^{3,\infty}$ we have by a familiar estimate that 
\begin{align*}
    & \bigg| \int_s^t    \int_{\R^3} (V_3  (\partial_{3}e^{2(t-\tau)\Delta} \phi)*\phi*u_3))   u_3  \,dx  \,d\tau\bigg| \\
    &= \bigg|\int_s^t    \int_{\R^3} (V_3  ( e^{2(t-\tau)\Delta} \phi)*\phi*\partial_3 u_3))   u_3  \,dx  \,d\tau\bigg| \leq CM \int_s^t \| \nabla u \|_{L^2}^2\,d\tau,
\end{align*}
where we used Young's inequality and the monotonicity of the heat semigroup in $L^p$ as follows:
$$\|( e^{2(t-\tau)\Delta} \phi)*\phi\|_{L^1}\leq C\|e^{2(t-\tau)\Delta} \phi\|_{L^1}\|\phi\|_{L^1}\leq C\|\phi\|_{L^1}^2\leq  C.$$
If $j=1,2$ and because $\phi$, the kernel of $e^{t\Delta}$ and $u_3$ are all even, 
the vector field $(\partial_{j}e^{2(t-\tau)\Delta} \phi)*\phi*u_3$ is odd in $x_j$ and therefore we can apply the endpoint Hardy inequality along with the fact that $V_j\leq M|x_h|^{-1}$---see Remark \ref{rem::OddEndpointHardy}---to obtain
\begin{align*} 
    \bigg|\int_s^t    \int_{\R^3} (V_j  (\partial_{j}e^{2(t-\tau)\Delta} \phi)*\phi*u_3))   u_3  \,dx  \,d\tau \bigg|&\leq C M\int_s^t   \| u_3\|_{L^2} \| \nabla ( (\partial_{j}e^{2(t-\tau)\Delta} \phi)*\phi*u_3    )\|_{L^2}\,d\tau. 
\end{align*}
We can commute the outermost derivatives and the Fourier multiplier operators to end up with
\begin{align*}
    \| \nabla ( (\partial_{j}e^{2(t-\tau)\Delta} \phi)*\phi*u_3    )\|_{L^2} \leq C \| \partial_{j}e^{2(t-\tau)\Delta} \phi\|_{L^1} \| \nabla u\|_{L^2}.
\end{align*}
By properties of the heat semigroup,
\begin{align*}
    \| \partial_{j}e^{2(t-\tau)\Delta} \phi\|_{L^1} \leq \frac C {(t-\tau)^{1/2}}\|\phi\|_{L^1}.
\end{align*}
Hence, for $j=1,2$,  
\begin{align*} 
    \bigg|\int_s^t    \int_{\R^3} (V_j  (\partial_{j}e^{2(t-\tau)\Delta} \phi)*\phi*u_3))   u_3  \,dx  \,d\tau\bigg|&\leq C  M \| u_0\|_{L^2}\|\phi\|_{L^1}^2 \int_s^t \frac 1 {(t-\tau)^{1/2}} \|\nabla u\|_{L^2}   \,d\tau,
\end{align*}
where we used the fact that $\|u_3\|_{L^2}\leq \|u_0\|_{L^2}$. 

For the other term, we apply Proposition \ref{prop::TypeII.5uVvEstimate}  where we note  $e^{2(t-\tau)\Delta}\phi*\phi*u$ is axisymmetric because convolution with a radial function preserves axisymmetry and it is in $H^2$ because $\phi$ is Schwartz and $u\in L^2$. 
This leads to the bound
\begin{align*}
    \bigg| \int_{\R^3}  ((u\cdot \nabla (e^{2(t-\tau)\Delta}\phi*\phi*u)) \cdot V)(\tau) \,dx\bigg| \leq C M \|\nabla u\|_{L^2}^2.
\end{align*}
We therefore obtain   a version of \cite[(6.7)]{KarchPilSchonbek}, namely
\begin{align*}
    \| \hat \phi \hat u(t)\|_{L^2}^2 &\leq I_1(t,s) + C(\|u_0\|_{L^2}^2, M) \int_s^t \|\nabla u\|_{L^2}^2\,d\tau+C M  \int_s^t \frac 1 {(t-\tau)^{1/2}} \|\nabla u\|_{L^2}(\tau)    \,d\tau.
\end{align*}
In contrast with \cite{{KarchPilSchonbek}}, we cannot show this goes to zero directly. 
Instead we take an average from $s$ to $s+T$ to obtain
\begin{align*}
    \frac 1 T \int_s^{s+T} \|\hat \phi \hat u\|_{L^2}^2(t)\,dt &\leq  \frac 1 T \int_s^{s+T}   I_1(t,s)   \,dt
    \\&+\frac {CM} T \int_s^{s+T}   \int_s^t \|\nabla u\|_{L^2}^2\,d\tau    \,dt
    \\&+\frac {CM} T \int_s^{s+T}    \int_s^{t} \frac 1 {(t-\tau)^{1/2}} \|\nabla u\|_{L^2}(\tau)   \,d\tau  \,dt:= \Gamma_1 +\Gamma_2 +\Gamma_3.
\end{align*}
The last average above is dominated by the following
\begin{align*}
    \frac C T \int_s^{s+T}    \int_s^{s+T} \frac 1 {(t-\tau)^{1/2}} \|\nabla u\|_{L^2} (\tau)  \,d\tau  \,dt \leq \frac C T T^{1/2} \int_s^{s+T}\|\nabla u\|_{L^2}\,dt \leq C \bigg( \int_s^{s+T}   \|\nabla u\|_{L^2}^2\,dt \bigg)^{1/2},
\end{align*}
where we used Tonelli's theorem.
By taking $s$ large we can ensure that  $\int_s^{\infty} \|\nabla u\|_{L^2}^2\,dt$ is as small as desired. 
This implies that $\Gamma_2+\Gamma_3$ can be made small independently of $T$. 
We then observe that, when $s$ is fixed, $I_1(t,s)\to 0$ as $t\to \infty$ by properties of the heat semigroup and, therefore, 
\begin{align*}
    \frac C T  \int_s^{s+T}I_{1}(t,s)\,dt \to 0,
\end{align*}
as well provided we take $T$ sufficiently large. 

{
We now return to the full energy. We note that  because  $\frac {1-\hat \phi} {|\cdot|}\in L^\infty$, we have,
\[
    \| (1-\hat \phi ) \hat u\|_{L^2} =\|  \frac {1-\hat \phi} {|\cdot|} |\cdot| \hat u\|_{L^2}\lesssim C_\phi \| u\|_{\dot H^1},
\]
and therefore
\[
\frac 1 T\int_s^{s+T} \| u\|_{L^2}(t)^2 \,dt \leq C \frac 1 T\int_s^{s+T} \| \hat \phi  \hat u\|_{L^2}(t)^2 + \|\nabla u\|_{L^2}^2  \,dt\to 0\text{ as $s\to\infty$}.
\]
Hence there exists a sequence of times $t_k$, which can be chosen so that the energy inequality is satisfied by $u$ when the initial time is $t_k$, for which $\|u(t_k)\|_{L^2}\to 0$. By monotonicity we find that $\|u(t)\|_{L^2}\to 0$ as $t\to \infty$.\footnote{In \cite[Proof of Theorem 2.7. Second part]{KarchPilSchonbek} the generalized energy inequality was used again to bound the high frequency term. We first attempted it that way but found it  was incompatible with the divergence form of the term $\nabla \cdot (V\otimes u)$. The approach we ended up taking, which is effectively just using the Poincar\'e-type inequality for high frequencies, is much simpler.}
}

\end{proof}

{
\begin{remark}\label{rem::SpecExistenceTypeII.5} In this paper we do not construct weak solutions to \eqref{pns} when $V$ is Type II.5. A glaring issue is that the energy might not be uniformly bounded in time.  Proposition \ref{prop::TypeII.5Vuphi*uEstimate} sheds light on how to get around this. Imagine we are constructing a solution via an approximation scheme $\{u^n\}$. We have a version of the energy inequality \eqref{eq::EnergyInequality} for the approximate solutions. We do not have 
\[
\int \nabla \cdot (V\otimes u^n)\cdot u^n\,dx=0,
\]
however. We also cannot directly hide this term in the diffusion, although we can hide the other linear term in \eqref{eq::EnergyInequality} in the diffusion assuming $M$-boundedness with sufficiently small $M$. Assuming $u^n$ are smooth enough ($H^2$ is sufficient, this is consistent with the Galerkin scheme in \cite{KarchPilSchonbek}),   Proposition \ref{prop::TypeII.5Vuphi*uEstimate} shows that the terms that cannot be absorbed in the diffusion  can be removed from the problem if we assume there is a \textbf{source} at the singularity because then the remaining terms will have a good sign. We do not work out the details in full, and this means that it is possible our stability result is vacuous, but we take this remark as strong evidence in favor of the existence of weak solutions around Type II.5 flows with sources satisfying \eqref{eq::EnergyInequalityModified}. The generalized energy inequality as we use it involves a smoothed vector field and so obtaining it as a limit of a Galerkin scheme is reasonable--indeed our estimates in the Proof of Theorem \ref{mainthm2} support this.
 \end{remark}
 }

\begin{remark}\label{rem::TypeII.5andR3plus}As there are two approaches to asymptotic stability, the semigroup approach \cite{KarchPil} and the energy method approach (see \cite{KarchPilSchonbek} for $\R^3$ and Section \ref{sec::NonStationaryStability} for other domains), it is natural to ask what happens in each for Type II.5 solutions. The semigroup approach has the benefit of working for more domains than the energy method approach due to the latter's dependence on either the Fourier splitting argument \cite{KarchPilSchonbek} or proving small-data global well-posedness in $L^{3,\infty}(\Omega)$ (Section \ref{sec::NonStationaryStability}. The semigroup approach or our approach in Section \ref{sec::NonStationaryStability}, therefore, could give us stability in $\R^3_+$ whereas the approach of \cite{KarchPilSchonbek} does not. In attempting the semigroup approach, we found that while the accretive property in Proposition \ref{prop::analyticsemigroup} is satisfied, this is not clear for the continuity property  because of the term 
\begin{align*}
\int_\Omega (V\cdot\nabla u )\cdot w\,dx,
\end{align*}
which we failed to bound when $V$ is Type II.5. This term also presents an obstacle to proving small-data global well-posedness in $L^{3,\infty}(\Omega)$ when perturbing around Type II.5 singular solutions which is needed for the approach we develop in Section \ref{sec::NonStationaryStability}.  { Under our many assumptions, Remark \ref{rem::SpecExistenceTypeII.5} suggests that this term can in some cases be removed from} energy estimates which motivates our use of the approach from \cite{KarchPilSchonbek}.

\end{remark}

\section{Asymptotic stability for non-stationary flows}
\label{sec::NonStationaryStability}

We work in domains for which small-data global well posedness holds.
Our proof of small-data global well posedness requires the bilinear estimate in Lemma \ref{lem::EBilinearEstimates}.
Let $\Omega\subseteq \R^3$ be a domain for which the Stokes semigroup $e^{-tA}$ is a well-defined linear operator from $L_\sigma^{3/2,1}(\Omega)$ to $L_\sigma^{3,1}(\Omega)$ which satisfies
\begin{equation}
    \label{eq::yamazaki}
    \int_0^\infty \|\nabla e^{-tA}\varphi\|_{L^{3,1}(\Omega)}\,dt\leq C\|\varphi\|_{L^{3/2,1}(\Omega)},
\end{equation}
for every $\varphi\in L^{3/2,1}_\sigma(\Omega)$ and the identity $\left(L^{3/2,1}_\sigma(\Omega)\right)^*=L^{3,\infty}_\sigma(\Omega)$. 
We call such domains \textit{admissible}. 
By \cite[(2.10)]{Yamazaki}, $\Omega=\R^3$ and $\Omega=\R^3_+$ are admissible. 

Fix an admissible domain $\Omega$. 
We will prove asymptotic stability for small $V\in C(0,\infty;L^{3,\infty}(\Omega))$ which are continuous in time and divergence-free. 
Fix such a $V$. When perturbing around a stationary solution, the terms from the linearized equation can be built into the semigroup. 
This is the approach taken in \cite{KarchPil} and Section \ref{sec::StationaryStability}. 
The asymptotic stability of non-stationary fluids, on the other hand, can be proved in $\R^3$ using the Fourier splitting method \cite{KarchPilSchonbek} which we used in Section \ref{sec::TypeII.5Stability}. 
This technique is not directly available in $\R^3_+$. 
To get around this, we introduce a new approach that uses eventual regularity of weak solutions in a novel way.  
In particular $u(\cdot,t)$ is eventually small in $L^{3,\infty}(\Omega)$ at some times. 
We will then establish a small-data global well-posedness theorem for \eqref{pns} and weak-strong uniqueness to conclude that $u$ is small in $L^{3,\infty}(\Omega)$ for all times. 
Once this is known all terms can be viewed perturbatively to the linear problem, an observation which allows us to prove $L^2$-decay.  
We begin by defining mild solutions. 

\begin{definition}
    We call a function $u\in L^\infty(T_0,\infty;L^{3,\infty}(\Omega))$ a mild solution of \eqref{pns} with initial data $u(T_0)$ if and only if $u(t)$ satisfies the identity 
    \begin{equation}
        \label{eq::mildSolution}
        \int_\Omega u(t)\cdot\varphi dx
        = \int_\Omega u(T_0)\cdot e^{-(t-T_0)A}\varphi dx
        +\int_{T_0}^t \int_\Omega (u\otimes u+V\otimes u+u\otimes V)(\tau):\nabla e^{-(t-\tau) A}\varphi \, dx\, d\tau,
    \end{equation}
    for all $\varphi\in L_\sigma^{3/2,1}(\Omega)$ and $t\in [T_0,\infty)$. 
\end{definition}

The next proposition states that weak solutions to \eqref{pns} decay provided they eventually enter $L^{\infty}_tL^{3,\infty}_x$. 
Note that it does not require that $V$ is small because decay follows from the energy inequality once we have that $V$ and $u$ are in $L^{\infty}_tL^{3,\infty}_x$, regardless of their size.

\begin{proposition}
    \label{prop::MildSolutionL2Decay}  
    Fix a weak solution $u$ of \eqref{pns} in $\Omega$. 
    Assume there exists $T_0> 0$ and a mild solution $v\in L^{\infty}(T_0,\infty;L^{3,\infty})$, which is also a weak solution, on $[T_0,\infty)$ with $u(t)=v(t)$ in $L^2(\Omega)$ for a.e. $t>T_0$. 
    Then 
    \begin{align*}
        \lim_{t\to\infty}\|u(t)\|_{L^2(\Omega)}=0.
    \end{align*}
\end{proposition}

\begin{proof}
    Fix $s\geq T_0$ and $\varphi\in C_{c,\sigma}^\infty(\Omega)$ with $\|\varphi\|_{L^2_\sigma}=1$. 
    Since $C_{c,\sigma}^\infty \subseteq L^{3/2,1}(\Omega)$ we have
    \begin{align*}
        \int_\Omega v(s)\cdot\varphi dx=\int_\Omega v(T_0)\cdot e^{-(s-T_0)A}\varphi dx+\int_{T_0}^s \int_\Omega (v\otimes v+V\otimes v+v\otimes V)(\tau):\nabla e^{-(s-\tau) A}\varphi dxd\tau.
    \end{align*} 
    So
    \begin{align*}
    `   \|v(s)\|_{L^2(\Omega)} \leq \|e^{-(s-T_0)A}v(T_0)\|_{L^2(\Omega)}+C\int_{T_0}^s \frac{\|v\otimes v+V\otimes v+v\otimes V\|_{L^2(\Omega)}(\tau)}{(s-\tau)^{1/2}}d\tau.
    \end{align*}
    by \eqref{eq::L2SelfDual}. 
    Note that, for each $\tau>T_0$, \begin{align*}\|v\otimes V\|_{L^2(\Omega)}(\tau)\leq \|v\|_{L^{6,2}(\Omega)}(\tau)\cdot \|V\|_{L^{3,\infty}(\Omega)}(\tau)\leq \|\nabla v\|_{L^2(\Omega)}(\tau)\cdot \|V\|_{L^{3,\infty}(\Omega)}(\tau),\end{align*}
    by the Lorentz refinement of the Holder and Sobolev inequalities. 
    The same type of estimate holds for $\|V\otimes v\|_{L^2(\Omega)}(\tau)$ and $\|v\otimes v\|_{L^2(\Omega)}(\tau)$. 
    So
    \begin{equation}
        \label{eq::BoundOnAverage}
        \begin{aligned}
            \frac{1}{t-T_0}\int_{T_0}^t \|&v\|_{L^2(\Omega)}(s)ds
            \leq \frac{1}{t-T_0}\int_{T_0}^t \|e^{-(s-T_0)A}v(T_0)\|_{L^2(\Omega)}ds\\
            &+\frac{C\left(\|v\|_{L_t^\infty L_x^{3,\infty}(\Omega)}
            +2\|V\|_{L_t^\infty L_x^{3,\infty}(\Omega)}\right)}{t-T_0}\int_{T_0}^t\int_{T_0}^s \frac{1}{(s-\tau)^{1/2}}\|\nabla v\|_{L^2(\Omega)}(\tau)d\tau ds.
        \end{aligned}
    \end{equation}
    The double integral can be simplified using Fubini's theorem
    \begin{align*}
        \frac{1}{t-T_0}\int_{T_0}^t\int_{T_0}^s \frac{1}{(s-\tau)^{1/2}}\|\nabla v\|_{L^2(\Omega)}(\tau)d\tau ds
        &=\frac{1}{t-T_0}\int_{T_0}^t \|\nabla v\|_{L^2(\Omega)}(\tau)\int_\tau^t (s-\tau)^{-1/2}dsd\tau\\
        &=\frac{2}{t-T_0}\int_{T_0}^t (t-\tau)^{1/2}\|\nabla v\|_{L^2(\Omega)}(\tau)d\tau =: I(t).
    \end{align*}
    
    It is not hard to show this integral goes to zero. 
    Decompose $I(t)$ as 
    \begin{align*}
        I(t)=\frac{2}{t-T_0}\int_{T_0}^M (t-\tau)^{1/2}\|\nabla v\|_{L^2(\Omega)}(\tau)d\tau+\frac{2}{t-T_0}\int_M^t (t-\tau)^{1/2}\|\nabla v\|_{L^2(\Omega)}(\tau)d\tau=:I_1(t)+I_2(t),
    \end{align*}
    for some fixed $M$. 
    Then 
    \begin{align*}
        I_1(t)\leq \frac{2}{\sqrt{t-T_0}}\int_{T_0}^M \|\nabla v\|_{L^2(\Omega)}(\tau)d\tau\leq \frac{2}{\sqrt{t-T_0}}\sqrt{M-T_0}\|\nabla v\|_{L^2(T_0,M;L^2_x(\Omega))},
    \end{align*}
    and 
    \begin{align*}
        I_2(t)\leq \frac{2}{t-T_0}\left(\int_M^t (t-\tau)d\tau\right)^{1/2}\|\nabla v\|_{L^2_t(M,\infty;L^2_x(\Omega))}\leq\sqrt{2}\|\nabla v\|_{L^2_t(M,\infty;L^2_x(\Omega))}.
    \end{align*}
    For $\epsilon>0$ let $M$ be such that $\|\nabla v\|_{L^2_t(M,\infty;L^2_x(\Omega))}<\frac{\epsilon}{2\sqrt{2}}$ and 
    \begin{align*}
        t>\max \left\{\frac{16}{\epsilon^2}(M-T_0)\|\nabla v\|_{L^2(T_0,M;L^2_x(\Omega))}^2+T_0,M\right\}.
    \end{align*} 
    Such a $M$ exists because $\|\nabla v\|_{L^2(\Omega)}(t)\in L^2(\R^+)$ by the energy inequality. 
    So $I(t)<\epsilon$. Since $v,V\in L_t^\infty L_x^{3,\infty}$ the bound \eqref{eq::BoundOnAverage} together with the known decay of 
    \begin{align*}
        \frac{1}{t-T_0}\int_{T_0}^t \|e^{-(s-T_0)A}v(T_0)\|_{L^2(\Omega)}ds,
    \end{align*} 
    proves decay of the averages of $\|v\|_{L^2(\Omega)}(\cdot)$ over $(T_0,t)$ as $t\to \infty$.
    Since $v$ is a weak solution, $\|v(t)\|_{L^2(\Omega)}\leq \|v(s)\|_{L^2(\Omega)}$ for a.e. $s$ and all $t\geq s$.  
    We can therefore find a sequence of times $s_k$ so that $\|v(t)\|_{L^2(\Omega)}\leq \|v(s_k)\|_{L^2(\Omega)}$ for $t>s_k$ and $\|v(s_k)\|_{L^2(\Omega)} \to 0$ as $k\to \infty$.  
    $L^2$-decay of $v$ follows from this, as does $L^2$-decay of $u$.
\end{proof}

We now aim to prove the assumptions of Proposition \ref{prop::MildSolutionL2Decay} always hold when $V$ is sufficiently small, which requires we construct a mild solution and establish weak-strong uniqueness. 
For this we need to prove bilinear estimates in $L_t^\infty L_x^{3,\infty}$ and the energy class.  

Recall that the dual of $L^{3,1}(\Omega)$ is $L^{3/2,\infty}(\Omega)$. 
So, if $F$ satisfies $\sup_{0<s<t}\|F(s)\|_{L^{3/2,\infty}}<\infty$, then 
\begin{align*}
    \Lambda_t(\varphi):=\int_0^t(F(s),\nabla e^{-(t-s)A}\varphi)ds,
\end{align*}
is a bounded linear functional on $L^{3/2,1}(\Omega)$ for each $t$ because
\begin{align*}
    \int_0^t\left|(F(s),\nabla e^{-(t-s)A}\varphi)\right|ds
    &\leq C\left(\sup_{0<s<t}\|F(s)\|_{L^{3/2,\infty}(\Omega)}\right) \int_0^t \|\nabla e^{-(t-s)A}\varphi\|_{L^{3,1}(\Omega)}ds\\
    &\leq C\left(\sup_{0<s<t}\|F(s)\|_{L^{3/2,\infty}(\Omega)}\right) \|\varphi\|_{L^{3/2,1}(\Omega)},
\end{align*}
for a constant $C=C(\Omega)$. 
By the duality $(L_\sigma^{3/2,1}(\Omega))^*=L_\sigma^{3,\infty}(\Omega)$ assumed in the admissibility definition, there exists an element $B(F)(t)\in L^{3,\infty}_\sigma(\Omega)$ such that 
\begin{align*}
    \langle B(F)(t),\varphi\rangle=\Lambda_t(\varphi),
\end{align*} 
for all $\varphi\in L_\sigma^{3/2,1}(\Omega)$ and 
\begin{equation}
    \label{eq::EBilinearEstimatesProof}
    \|B(F)(t)\|_{L^{3,\infty}(\Omega)}
    \leq C\sup_{0<s<t}\|F(s)\|_{L^{3/2,\infty}(\Omega)}.
\end{equation}
Define 
\begin{align*}
    B(w,v):=B(w\otimes v).
\end{align*} 
Taking the supremum in $t$ on both sides of \eqref{eq::EBilinearEstimatesProof} proves the following lemma.

\begin{lemma}
    \label{lem::EBilinearEstimates}
   Let  $\Omega$ be admissible. 
   There exists a constant $C_E>0$ which depends only on $\Omega$ such that
    \begin{align*}
        \sup_{t>0} \|B(w,v)(t)\|_{L^{3,\infty}(\Omega)}\leq C_E\|w\|_{L_t^\infty L_x^{3,\infty}(\Omega)}\|v\|_{L_t^\infty L_x^{3,\infty}(\Omega)}.
    \end{align*}
\end{lemma}

This lemma is also stated in \cite[Lemma 7]{Taniuchi} for the whole-space, half-space, exterior domains, perturbed half space and aperture domains.

Suppose $u$ satisfies
\begin{equation}
    \label{eq::ModifiedMildSolution}
    u(t)=e^{-tA}u_0+B(u,u)+B(u,V)+B(V,u).
\end{equation}
Since $u_0\in L^{3,\infty}$, $e^{-tA}$ is interpreted as the dual of $e^{-tA}$ defined on $L^{3/2,1}$. 
Then $u$ also satisfies \eqref{eq::mildSolution} for all $\varphi\in L^{3/2,1}(\Omega)$. 
In the proof of eventual regularity we will also need to know the mild solution we construct is also a weak solution when the data is in $L^2$. 
For this we will also need bilinear estimates in the space
\begin{align*}
    X_T:=L^\infty(0,T;L^2_\sigma(\Omega))\cap L^2(0,T;\dot{H}^1_{0,\sigma}(\Omega)).
\end{align*}
We will also consider the subspace 
\begin{align*}
    \tilde X_T:=BC(0,T;L^2_\sigma(\Omega))\cap L^2(0,T;\dot{H}^1_{0,\sigma}(\Omega)),
\end{align*}
which is a Banach space under the same norm as $X_T$.

\begin{lemma}
    \label{lem::XTBilinearEstimates}
    Let $E:=L^\infty(0,T;L^{3,\infty}(\Omega))$. 
    There exists a constant $C_X>0$ which depends only on $\Omega$ such that:
    \begin{itemize}
        \item If $w\in E$ and $v\in X_T\cap E$ then 
                \begin{align*}
                    \|B(w,v)\|_{X_T}\leq C_X\|w\|_E\|v\|_{X_T}.
                \end{align*}
        \item If $w\in X_T\cap E$ and $v\in E$ then 
                \begin{align*}
                    \|B(w,v)\|_{X_T}\leq C_X\|w\|_{X_T}\|v\|_E.
                \end{align*}
    \end{itemize}
\end{lemma}

\begin{proof}
    Denote by $V_0$ the space $H_{0,\sigma}^1(\Omega)$ and $V_0'$ its dual.  
    For $w\in E$, $v\in X_T$ and $t\in [0,T]$ define the function $F(w,v)(t)\in V_0'$ as the unique element satisfying 
    \begin{equation}
        \label{eq::BilinearStokesForce}
        (F(w,v)(t),\varphi) = (w(t)\cdot\nabla \varphi,v(t)),
    \end{equation}
    for all $\varphi\in V_0$. 
    This defines $F(w,v)\in L^2(0,T;V_0')$ because
    \begin{align*}
        |(F(w,v)(t),\varphi)|\leq \|w(t)\|_{L^{3,\infty}}\|\nabla\varphi\|_{L^2}\|v(t)\|_{L^{6,2}}\leq C\|w(t)\|_{L^{3,\infty}}\|\nabla\varphi\|_{L^2}\|\nabla v(t)\|_{L^2},
    \end{align*}
    by Lorentz-Sobolev and, hence,
    \begin{equation}
    \label{eq::StokesForceEstimate}
        \begin{aligned}
            \|F(w,v)\|^2_{L^2(0,T;(\dot{H}_{0,\sigma}^1)')}
            &= \int_0^T \|F(w,v)(t)\|_{(\dot{H}_{0,\sigma}^1)'}^2dt\\
            &\leq C\int_0^T \|w(t)\|_{L^{3,\infty}}^2\| v(t)\|_{V_0}^2dt
             \leq C\|w\|_{E}^2\|v\|_{L^2(0,T;V_0)}^2.
        \end{aligned}
    \end{equation}
    By \cite[Lemma IV.2.4.2]{SohrBook}, because $B(F)$ is defined by \cite[(IV.2.4.13)]{SohrBook}, $B(F)$ is a weak solution to the evolutionary Stokes problem with force $F(w,v)$ and zero initial data. 
    Furthermore, because $w\otimes v\in L^2(0,T;L^2)$ when $w\in E$ and $v\in X_T$ (this is the condition \cite[Lemma IV.2.4.2.d]{SohrBook}) and our forcing is in divergence form, we have that 
     \begin{equation}
        \label{eq::StokesEnergyInequality}
        \sup_{0\leq t\leq T}\|B(F)(t)\|_{L^2(\Omega)}^2+\int_0^T \|B(F)(s)\|_{L^2(\Omega)}^2 ds
        \leq \|F(w,v)\|_{L^2(0,T;(\dot{H}_{0,\sigma}^1)')}^2
        \leq C \|w\|_E^2 \| v\|_{X_T},
    \end{equation}
    which proves one case of the theorem. 
    The other case is symmetric. 
\end{proof}

These tools allow us to construct a time-global solution by following a Picard argument. 
This is of course a well known approach and our own work is essentially the same as Meyer's proof of small-data global well-posedness for the non-perturbed Navier-Stokes equations in $L^{3,\infty}$ from  \cite{Meyer}. 
We go beyond the conclusions in \cite{Meyer}, however, by showing our solution is also a weak solution when the data is in $L^2$. 

\begin{theorem}[Small-data global well-posedness]
    \label{thm::GlobalWellPosedness}
    Fix $\Omega\subseteq\R^3$ admissible and $u_0\in L_\sigma^{3,\infty}(\Omega)$. 
    There exists $\epsilon_1>0$ and $\epsilon_*>0$ such that a mild solution $u\in L^\infty(0,\infty;L^{3,\infty}(\Omega))$ with initial data $u_0$ exists uniquely and satisfies 
    \begin{align*}
        \sup_{t>0}\|u\|_{L^{3,\infty}(\Omega)}<\frac{3C\epsilon_1}{2},
    \end{align*}
    whenever $\|u_0\|_{L^{3,\infty}(\Omega)}<\epsilon_1$ and $\|V\|_{L^{3,\infty}(\Omega)}<\epsilon_*$. 
    Additionally, upon possibly shrinking $\epsilon_1$ and $\epsilon_*$, if $u_0\in L^2_\sigma(\Omega)$, then $u(t)$ is a weak solution on the time-interval $[0,T)$ for every $T>0$ and, furthermore, $u(t)\in \tilde X_T$ and satisfies the energy equality \eqref{eq::EnergyEquality} for all $0\leq s\leq t< \infty$. 
\end{theorem}

From the perspective of stability, this theorem says that $V$ is $L^{3,\infty}$-stable under small perturbations. One can ask if $\| u\|_{L^{3,\infty}}\to 0$ as $t\to\infty$. This is not generally possible as is demonstrated by a counterexample in \cite{BW}.

\begin{proof}
    First, let $E=L^\infty(0,\infty; L^{3,\infty}(\Omega))$. 
    We've shown there exists a constant $C_E>1$ such that
    \begin{align*}
        \|B(w,v)\|_E\leq C_E\|w\|_E\|v\|_E,
    \end{align*} 
    for all $w,v\in E$. 
    We can also assume $\epsilon_*>0$ is small enough so that 
    \begin{align*}
        \|B(w,V)\|_E+\|B(V,w)\|_E\leq \frac{1}{8}\|w\|_E.
    \end{align*} 
    For each positive integer $n$, let 
    \begin{align*}
        u^{(0)}:=e^{-tA}u_0,\qquad u^{(n)}=u^{(0)}+B(u^{(n-1)},u^{(n-1)})+B(V,u^{(n-1)})+B(u^{(n-1)},V).
    \end{align*}
    First suppose $\epsilon_1>0$ is small enough so that $\epsilon:=\|e^{-tA}u_0\|_E\leq \frac{1}{8C_E}$. 
    Then 
    \begin{align*}
        \|u^{(1)}\|_E\leq \epsilon+\frac{1}{8}\epsilon+\frac{1}{8}\epsilon\leq \frac{3}{2}\epsilon.
    \end{align*} 
    Inductively, assuming $\|u^{(n-1)}\|_E\leq \frac{3}{2}\epsilon$ we have 
    \begin{align*}
        \|u^{(n)}\|_E\leq \epsilon+C_E\frac{9}{4}\epsilon^2+\frac{3}{16}\epsilon\leq \frac{47}{32}\epsilon\leq \frac{3}{2}\epsilon.
    \end{align*}
    This shows the map $u^{(n)}\mapsto u^{(n+1)}$ is a map from the ball $B_{3\epsilon/2}$ to itself. 
    It remains to show that the sequence $u^{(n)}$ is Cauchy. 
    Let $d_n:=u^{(n)}-u^{(n-1)}$ Note that 
    \begin{equation}
        \label{eq::CauchyEquation}
        d_{n+1}=B(d_n,u^{(n)})+B(u^{(n-1)},d_n)+B(d_n,V)+B(V,d_n)
    \end{equation}
    So 
    \begin{align*}
        \|d_{n+1}\|_E\leq \left(C_E\|u^{(n)}\|_E+C_E\|u^{(n-1)}\|_E+\frac{1}{8}\right)\|d_n\|_E\leq \left(\frac{3}{16}+\frac{3}{16}+\frac{1}{8}\right)\|d_n\|_E=\frac{1}{2}\|d_n\|_E.
    \end{align*}
    Thus $\|d_n\|$ is Cauchy and hence converges in $B_{3\epsilon/2}$. 
    The limit $u$ solves $\eqref{eq::ModifiedMildSolution}$ and satisfies 
    \begin{align*}
        \|u\|_E\leq \frac{3}{2}\|e^{-tA}u_0\|_E\leq \frac{3}{2}C\epsilon_1
    \end{align*} 
    for some constant $C>0$ independent of $u_0$. This solution is in fact unique in $B_{3\epsilon/2}$. 
    Indeed, if $v\in B_{3\epsilon/2}$ also solved $\eqref{eq::ModifiedMildSolution}$ then 
    \begin{align*}
        \|u-v\|_E\leq \frac{1}{2}\|u-v\|_E.
    \end{align*} 
    This shows $u=v$ a.e. 
    The mild-solution formulation then implies $u=v$ in $L^{3,\infty}$ for every $t\in [0,\infty)$.

    It remains to show that $u$ is also a weak solution in $[0,T]$ for all $T>0$. 
    Assume $u_0\in L^2_\sigma(\Omega)$. 
    Fix $T>0$. 
    We will show that $u\in X_T$. 
    We first establish that $M_n:=\|u^{(n)}\|_{X_T}$ is bounded uniformly. 
    By Lemma \cite[Lemma IV.1.5.1]{SohrBook}, $u^{(0)}=e^{-tA}u_0$ solves 
    \begin{align*}
        \partial_t u^{(0)}+Au^{(0)}=0,\qquad u^{(0)}(0)=u_0\in L^2_\sigma(\Omega),
    \end{align*}
    and moreover belongs to $C([0,T];L^2)$.
    So $u^{(0)}\in \tilde X_T$.
    Noting that, if $u^{(n)}\in X_T \cap E$, then $u^{n}\otimes u^{(n)}+V\otimes u^{(n)}+u^{(n)}\otimes V\in L^2(0,T;L^2(\Omega))$ because, e.g.,
    \begin{align*}
        \int_0^T \|  V\otimes u^{(n)}\|_{L^2}^2\,dt 
        \leq \int_0^T \| V\|_{L^{3,\infty}}^2\| \nabla u^{(n)}\|_{L^2}^2 \,dt 
        \leq C(V) \| u^{(n)}\|_{X_T}.
    \end{align*}
    We can therefore use \cite[Lemma IV.2.4.2.d]{SohrBook} to see that $u^{(n+1)}$ belongs to $\tilde X_T$, possibly after redefining it on a null set of $(0,T]$. 
    This re-definition does not change the definition of $u^{(n+1)}$ because the dependence of $u^{(n+1)}$ on $u^{(n)}$ is through a time-integral.
    By Lemma \ref{lem::XTBilinearEstimates} and the definition of $u^{(n+1)}$,
    \begin{align*}
        M_{n+1} 
        & \leq \|u^{(0)}\|_{X_T}+C_X \|u^{(n)}\|_{X_T}\|u^{(n)}\|_E+2C_X \|u^{(n)}\|_{X_T}\|V\|_E\\
        & \leq M_0+C_X M_{n} \frac{3}{2}\epsilon + 2C_X M_{n}\|V\|_E\\
        & = M_0 + C_X\left(\frac{3}{2}\epsilon+2\|V\|_E\right)M_{n},
    \end{align*}
    We know $\theta:=C_X\left(\frac{3}{2}\epsilon+2\|V\|_E\right)<1$ by assumption. 
    By induction 
    \begin{align*}
        M_{n+1}\leq \sum_{k=0}^n M_0 \theta^k,
    \end{align*} 
    and hence $M_{n+1}\leq \frac{M_0}{1-\theta}$ for all $n$. 
    So the Picard iterates are uniformly bounded in the $X_T$-norm and all belong to $\tilde X_T$.
    By \eqref{eq::CauchyEquation}, 
    \begin{align*}
        \|d_{n+1}\|_{X_T}\leq C_X\left(\|u^{(n)}\|_E+\|u^{(n-1)}\|_E+2\|V\|_E\right)\|d_n\|_{X_T}\leq C_X(3\epsilon+2\|V\|_E)\|d_n\|_{X_T}.
    \end{align*}
    After possibly modifying $\epsilon_1$ and $\epsilon_*$, we can assume $C_X(3\epsilon+2\|V\|_E)<1$ and hence $u^{(n)}$ is a Cauchy sequence in $\tilde X_T$. 
    So there exists some $v\in \tilde X_T$ such that $u^{(n)}\to v$  in $\tilde X_T$. 
    Finally note that in this paragraph $u^{(n)}$ were each possibly re-defined on a null set in $(0,T]$. 
    This does not change the quantities 
    \begin{align*}
        \operatorname{ess\,sup}_{0\leq t\leq T} \|d^{(n)}\|_{L^{3,\infty}},
    \end{align*}
    and, hence, we have shown the Cauchy property in both $\tilde X_T$ and $E$, we can take the sequence to be Cauchy in the space $\tilde X_T\cap E$ endowed with the norm $\|\cdot \|_{X_T}+\|\cdot \|_E$ and therefore conclude that the limit $v$ is also a mild solution satisfying the same properties as the limit $u$. 
    Furthermore, they agree except on a null set in $(0,T]$. 
    The integral formulas for $u$ and $v$ then imply $u=v$ in $L^{3,\infty}$ for every $t\in [0,T]$. 
    
    Our final step is to show that $u$ satisfies the energy equality
    \begin{equation}
        \label{eq::EnergyEquality}
        \|u(t)\|^2_{L^2(\Omega)}+2\int_s^t\|\nabla u (\tau)\|^2_{L^2(\Omega)}d\tau
        = \|u (s)\|^2_{L^2(\Omega)}+2\int_s^t \int_\Omega u \cdot\nabla u \cdot V\,dx\, d\tau.  
    \end{equation}
    We note that by \cite[Lemma IV.2.4.2]{SohrBook}, $u^{(n)}$ all satisfy the energy equalities
   \begin{align*}
      \|u^{(n)}(t)\|^2_{L^2(\Omega)}
       + 2\int_s^t\|\nabla u^{(n)} (\tau)\|^2_{L^2(\Omega)}d\tau
      &= \|u^{(n)} (s)\|^2_{L^2(\Omega)}
       + 2\int_s^t \int_\Omega u^{(n-1)} \cdot\nabla u^{(n)} \cdot V\,dx\, d\tau\\
      &  \quad 
       - 2\int_s^t \int \big((V+u^{(n-1)})\cdot\nabla u^{(n-1)}\big)\cdot u^{(n)}\,dx\,d\tau.
   \end{align*}
    Because we have convergence in $\tilde X_T$, which we emphasize amounts to uniform convergence in the time-variable in $L^\infty(0,T;L^2)$ and strong convergence in $L^2(0,T;\dot H^1)$, we obtain,
    \begin{align*}
      & \|u^{(n)}(t)\|^2_{L^2(\Omega)} + 2\int_s^t\|\nabla u^{(n)} (\tau)\|^2_{L^2(\Omega)}d\tau - \|u^{(n)} (s)\|^2_{L^2(\Omega)}\\
      & \qquad  
        \to \|u (t)\|^2_{L^2(\Omega)} + 2\int_s^t\|\nabla u (\tau)\|^2_{L^2(\Omega)}d\tau - \|u(s)\|^2_{L^2(\Omega)}.
    \end{align*}
    On the other hand 
    \begin{align*}
        &\bigg| \int_s^t \int_\Omega \big( u\cdot\nabla u -   u^{(n-1)} \cdot\nabla u^{(n)} \big)\cdot V\,dx\, d\tau\bigg|\\
        &\leq C \| V\|_{L^{3,\infty}} (\|u-u^{(n-1)}\|_{X_T}\|u\|_{X_T}+ \|u-u^{(n)}\|_{X_T}\|u^{(n-1)}\|_{X_T}))\to 0.
    \end{align*}    
    A similar estimate implies
    \begin{align*}
        \int_s^t \int \big(V\cdot\nabla u^{(n-1)}\big)\cdot u^{(n)}\,dx\,d\tau \to  \int_s^t \int \big(V \cdot\nabla u \big)\cdot u \,dx\,d\tau =0,
    \end{align*}
    by the divergence free condition. For the remaining term, by the divergence free condition we have 
    \begin{align*}
     &   \int_s^t\int   u^{(n-1)}\cdot\nabla u^{(n-1)} \big)\cdot u^{(n)}\,dx\,d\tau\\
     & = \int_s^t\int   u^{(n-1)}\cdot\nabla (u^{(n-1)} - u^{(n)})\big)\cdot u^{(n)}\,dx\,d\tau\\
     &   \leq C \| u^{(n-1)} - u^{(n)}\|_{X_T}( \|u^{(n)}\|_{X_T} \|u^{(n-1)}\|_{E} ) \to 0,
    \end{align*} 
    because $u^{(n-1)}$ is Cauchy in $X_T$.
\end{proof}

The last result we need is the following.

\begin{theorem}[Weak Strong Uniqueness]
    Fix $\Omega\subseteq\R^3$ open and connected and $T_0\geq 0$.
    Suppose $u$ is a weak solution on $[T_0,\infty)$, $v\in L^\infty([T_0,\infty);L^{3,\infty}(\Omega))$, and 
    \begin{align*}
        u(T_0)=v(T_0)\in L^2_\sigma(\Omega)\cap L^{3,\infty}(\Omega).
    \end{align*} 
    Suppose additionally that $v$ is a weak solution on every finite time interval $[T_0,T_1]$ for any $T_1>T_0$. 
    There exists $\epsilon_{WSU}>0$ such that if 
    \begin{align*}
        \|v\|_{L_t^\infty L_x^{3,\infty}}+\|V\|_{L_t^\infty L_x^{3,\infty}}<\epsilon_{WSU},
    \end{align*}
    then, for a.e.~$t> T_0$, $u(t)=v(t)$ in $L^2(\Omega)$.
\end{theorem}

\begin{proof}
    Let $w=u-v$. 
    Then 
    \begin{align*}
        w\in C_w(T_0,\infty;L^2_\sigma(\Omega))\cap L^2_{loc}(T_0,\infty;\dot{H}^1_{0,\sigma}(\Omega)),\qquad w(T_0)=0.
    \end{align*}
    Further, for every test function 
    \begin{align*}
        \varphi\in C(T_0,\infty;H^1_{0,\sigma}(\Omega))\cap C^1(T_0,\infty;L^2_\sigma(\Omega)),
    \end{align*}
    and every time interval $T_0\leq s\leq t<\infty$ $w$ satisfies the distributional problem
    \begin{equation}
        \label{eq::WeakSolutionToWeakStrongUniquenessPDE}
        \begin{aligned}
                \int_\Omega w(t)\cdot\varphi(t)dx
              + \int_s^t\int_\Omega \nabla w:\nabla\varphi dxd\tau 
            & - \int_s^t\int_\Omega ((v+w)\cdot\nabla)w\cdot\varphi dxd\tau
              - \int_s^t\int_\Omega (w\cdot\nabla)v\cdot \varphi dxd\tau\\
            & - \int_s^t\int_\Omega (w\cdot\nabla)\varphi\cdot V dxd\tau
              + \int_s^t\int_\Omega (V\cdot\nabla)w\cdot \varphi dxd\tau\\
            & = \int_\Omega w(s)\cdot\varphi(s)dx
              + \int_s^t\int_\Omega w\cdot\partial_\tau \varphi dxd\tau.
        \end{aligned}
    \end{equation}
    So, $w$ is a weak-solution to 
    \begin{equation}
        \label{eq::WeakStrongUniquenessPDE}
        \partial_t w - \Delta w + ((v+w)\cdot\nabla)w + (w\cdot\nabla)v + (w\cdot\nabla)V + (V\cdot\nabla)w + \nabla\pi = 0
    \end{equation}
    Formally testing \eqref{eq::WeakStrongUniquenessPDE} against $w$, we have 
    \begin{align*}
        \frac{1}{2}\frac{d}{dt}\|w(t)\|_{L^2}^2+\|\nabla w(t)\|_{L^2}^2=-\int_\Omega (w\cdot \nabla)w\cdot vdx-\int_\Omega (w\cdot \nabla)V\cdot wdx.
    \end{align*}
    Using Lorentz-Sobolev, one can control the right hand side and show that the time derivative of $\|w(t)\|_{L^2}^2$ is non-positive. 
    Hence $w=0$ by the zero initial data. 
    The issue, is that we do not have the necessary regularity to use $w$ itself as a test function. 
    To make these heuristics precise, we appeal to an approximation argument which will allow us to prove
    \begin{equation}
        \label{eq::WeakStrongUniquenessEnergyInequality}
          \frac{1}{2}\int_\Omega |w(t)|^2\,dx
        + \int_{T_0}^t\int_\Omega |\nabla w|^2 \,dx\,dt\leq 
        - \int_{T_0}^t \int (w\cdot\nabla)v\cdot w\,dx\,dt
        - \int_{T_0}^t\int_\Omega (w\cdot\nabla)V\cdot w\,dx\,dt.
    \end{equation}
    The proof is essentially the same as the proof of the analogous inequality in \cite[Theorem 4.4]{TsaiBook}. 
    As in \cite[Lemma 4.3]{TsaiBook}, fix some positive, even, test function, $\varphi\in C_c^\infty(\R)$ with $\int_{-1}^1\varphi=1$ and $\varphi(t)=0$ for $|t|\geq 1$. 
    Fix some $t'\in ({T_0},T]$. 
    For any $h>0$ we define
    \begin{align*}
        v_h(x,t)=\int_{T_0}^{t'} v(x,\tau')\phi_h(t-\tau')d\tau',
    \end{align*}
    where $\phi_h(t)=h^{-1}\phi(t/h)$. 
    Then, 
    \begin{align*}
        v_h\in C([T_0,T],H^1_{0,\sigma}(\Omega))\cap C^1([T_0,T]),L^2_\sigma(\Omega)).
    \end{align*} 
    By the distributional form of \eqref{pns}, 
    \begin{align*}
            \int_\Omega u(t)\cdot v_h(t) dx
        & + \int_{T_0}^t\int_\Omega \nabla u: \nabla v_h dxd\tau 
          + \int_{T_0}^t\int_\Omega (u\cdot\nabla)u\cdot v_h dxd\tau
          - \int_{T_0}^t\int_\Omega (u\cdot\nabla)v_h \cdot Vdxd\tau\\
        & + \int_{T_0}^t\int_\Omega (V\cdot\nabla)u\cdot v_h dxd\tau
          = \int_\Omega u({T_0})\cdot v_h({T_0})
          + \int_{T_0}^t \int_\Omega u\cdot \partial_\tau v_h dxd\tau,
    \end{align*}
    and
    \begin{align*}
            \int_\Omega v(t)\cdot u_h(t) dx
        & + \int_{T_0}^t\int_\Omega \nabla v: \nabla u_h dxd\tau 
          + \int_{T_0}^t\int_\Omega (v\cdot\nabla)v\cdot u_h dxd\tau
          - \int_{T_0}^t\int_\Omega (v\cdot\nabla)u_h \cdot Vdxd\tau\\
        & + \int_{T_0}^t\int_\Omega (V\cdot\nabla)v\cdot u_h dxd\tau
          = \int_\Omega v({T_0})\cdot u_h({T_0})
          + \int_{T_0}^t \int_\Omega v\cdot \partial_\tau u_h dxd\tau,
    \end{align*}
    for all $h>0$. 
    As stated in the proof of \cite[Theorem 4.4]{TsaiBook}, one can check that these weak formulations converge to the limit as $h\to 0$. 
    The only terms new to \eqref{pns} are the terms involving $V$ which can be handled using the usual Lorentz-Holder, Lorentz-Sobolev estimates such as: 
    \begin{align*}
        \left|\int_{T_0}^t\int_\Omega (u\cdot\nabla)(v_h-v)\cdot Vdxd\tau\right|
        \leq \|\nabla u\|_{L_t^2 L_x^2}\|\nabla (v_h-v)\|_{L_t^2 L_x^2}\|V\|_{L_t^\infty L_x^{3,\infty}} 
        \to 0.
    \end{align*}
    The convergence $v_h\to v$ in $L^\infty_tL^2_x\cap L^2_t H^1_{0,\sigma}$ is already used in \cite{TsaiBook}. 
    The only difficult term is 
    \begin{align*}
        \int_{T_0}^t\int_\Omega (u\cdot\nabla)u\cdot v_h\,dx\,d\tau\to \int_{T_0}^t\int_\Omega (u\cdot\nabla)u\cdot v\,dx\,d\tau.
    \end{align*}
    By the log convexity of $L^p$ norms, as $\frac{1}{3}=\frac{1}{2}\cdot\frac{1}{2}+\frac{1}{2}\cdot \frac{1}{6}$, and the Sobolev inequality,
    \begin{align*}
        \|v_h-v&\|^4_{L^2(T_0,t;L^{3,\infty}(\Omega))}\leq C(t)\|v_h-v\|^4_{L^4(T_0,t;L^{3,\infty}(\Omega))} 
          = C(t)\int_{T_0}^{t}\|v_h-v\|_{L^{3,\infty}(\Omega)}^4d\tau\\
        & \leq C(t)\int_{T_0}^{t}\|v_h-v\|_{L^3(\Omega)}^4d\tau
          \leq C(t)\|v_h-v\|_{L^2(T_0,t;L^2(\Omega))}^2\|\nabla(v_h-v)\|_{L^2(T_0,t;L^2(\Omega))}^2
          \to 0.
    \end{align*}
    By the Lorentz-Sobolev inequality, 
    \begin{align*}
        \|(u\cdot\nabla)u\|_{L^2(T_0,t;L^{3/2,1}(\Omega))}
        \leq C_S\|\nabla u\|_{L^2(T_0,t;L^2(\Omega))}^2
        < \infty,
    \end{align*}
    and $(L^2(T_0,t;L^{3/2,1}(\Omega)))'=L^2(T_0,t;L^{3,\infty}(\Omega))$, we have that 
    \begin{align*}
        \left|\int_{T_0}^t\int_\Omega (u\cdot\nabla)u\cdot (v_h-v)\,dx\,d\tau\right|
        & = |\langle v_h-v,(u\cdot\nabla)u\rangle|\\
        & \leq \|v_h-v\|_{L^2(T_0,t;L^{3,\infty}(\Omega))}\|(u\cdot\nabla)u\|_{L^2(T_0,t;L^{3/2,1}(\Omega))}
          \to 0.
    \end{align*}
    
    We have thus confirmed that the following equalities hold in the limit,
    \begin{equation}
        \label{eq::wsu::idk}
        \begin{aligned}
            \int_{T_0}^t(u,\partial_t v) - (\nabla u,\nabla v) - (u\cdot\nabla u,v) - ((V\cdot\nabla)v,u) \,d\tau
                    &=[(u,v)]_{T_0}^t - \int_{T_0}^t((v\cdot\nabla)u,V) \,d\tau,\\
            \int_{T_0}^t(v,\partial_t u) - (\nabla v,\nabla u) - (v\cdot\nabla v,u) - ((V\cdot\nabla)u,v) \,d\tau
                    &=[(v,u)]_{T_0}^t - \int_{T_0}^t((u\cdot\nabla)v,V) \,d\tau.
        \end{aligned}
    \end{equation}
    Note that  
    \begin{align*}
        \int_{T_0}^t-((v\cdot\nabla)v,u) - ((u\cdot\nabla)u,v)\, d\tau
            = -\int_{T_0}^t(w\cdot\nabla u,v) \,d\tau
            = -\int_{T_0}^t (w\cdot\nabla w,v) \,d\tau,
    \end{align*}
    and 
    \begin{align*}
        \int_{T_0}^t-((V\cdot\nabla)v,u) - ((V\cdot\nabla)u,v)d\tau = 0,
    \end{align*}
    by integration by parts. 
    So the sum of \eqref{eq::wsu::idk} and the energy inequalities, 
    \begin{align*}
        \frac{1}{2}\|u(t)\|_{L^2(\Omega)}^2 + \int_{T_0}^t\|\nabla u\|_{L^2(\Omega)}^2 \,d\tau
            &\leq \frac{1}{2}\|u(T_0)\|_{L^2(\Omega)}^2 + \int_{T_0}^t((u\cdot\nabla)u,V) \,d\tau,
            \end{align*}
            and
            \begin{align*}
        \frac{1}{2}\|v(t)\|_{L^2(\Omega)}^2+\int_{T_0}^t \|\nabla v\|_{L^2(\Omega)}^2 \,d\tau
            &\leq \frac{1}{2}\|v(T_0)\|_{L^2(\Omega)}^2 + \int_{T_0}^t((v\cdot\nabla)v,V) \,d\tau,
    \end{align*}
    lead to \eqref{eq::WeakStrongUniquenessEnergyInequality}. 

    Our by now standard estimates for the right-hand side of \eqref{eq::WeakStrongUniquenessEnergyInequality} imply that 
    \begin{align*}
        \frac{1}{2}\int |w(t)|^2 \,dx \leq 0,
    \end{align*}
    for almost every $t>T_0$, which proves the weak-strong uniqueness assertion in the theorem.
\end{proof}

Now we have all the tools needed to prove Theorem \ref{mainthm3}.

\begin{proof}[Proof of Theorem \ref{mainthm3}]
    Suppose $u$ is a weak solution to the perturbed Navier Stokes equations with initial data $u_0\in L_\sigma^2(\Omega)$.
    Fix $\epsilon_1,\epsilon_*>0$ small enough so that 
    \begin{align*}
        \frac{3C_1}{2}\epsilon_1+\epsilon_*<\epsilon_{\mathrm{WSU}}\qquad \text{and}\qquad\alpha>0
    \end{align*} 
    where $C_1$ is as in Theorem \ref{thm::GlobalWellPosedness} and $\alpha$ is as in \eqref{eq::EnergyInequality}. 
    Let $G_1\subseteq[0,\infty)$ denote the set of times $s$ for which $u(s)\in \dot{H}^1_{0,\sigma}(\Omega)$ and 
    \begin{align*}
        \|u(t)\|_{L^2(\Omega)}^2+\alpha\int_s^t\|\nabla u(\tau)\|^2_{L^2(\Omega)}d\tau\leq\|u(s)\|_{L^2(\Omega)}^2
    \end{align*}
    holds for every $t\geq s$. As both of these properties hold a.e., the set $G_1$ has full measure in $[0,\infty)$. 
    By the energy inequality, $u$ satisfies  
    \begin{align*}
        \limsup_{t\to\infty}\|u\|_{L^2(\Omega)}^2+\alpha\int_0^\infty \|\nabla u\|_{L^2(\Omega)}^2(t)dt\leq \|u_0\|_{L^2(\Omega)}^2<\infty.
    \end{align*}
    By the Sobolev embedding,
    \begin{align*}
        \lim_{k\to\infty}\int_k^{k+1} \|u\|_{L^6(\Omega)}^2(t)dt=0.
    \end{align*}
    So there exists a sequence of times $t_k\in G_1$ such that $t_k\to\infty$ and
    \begin{align*}
        \lim_{k\to\infty}\|u(t_k)\|_{L^3}\leq \lim_{k\to\infty}\|u(t_k)\|_{L^2}^{1/2}\cdot \|u(t_k)\|_{L^6}^{1/2}\leq \lim_{k\to\infty}\|u_0\|_{L^2}^{1/2}\|u(t_k)\|_{L^6}^{1/2}=0.
    \end{align*} 
    Since $L^3(\Omega)\hookrightarrow L^{3,\infty}(\Omega)$, this also gives $\|u(t_k)\|_{L^{3,\infty}}\to 0$. 
    So we can choose $T_0\in G_1$ large enough so that $w_0:=u(T_0)$ satisfies $\|w_0\|_{L^{3,\infty}}<\epsilon_1$. 
    By Theorem \ref{thm::GlobalWellPosedness}, there exists a mild solution $v$ to \eqref{pns} with initial data $v_0=u(T_0)$ with $V$ replaced by $V_{T_0}(t):=V(T_0+t)$. 
    Let $\tilde{v}(t)=v(t-T_0)$. 
    Then $\tilde{v}\in L^\infty(T_0,\infty;L^{3,\infty})$ is a mild solution of \eqref{pns} with the original $V$. 
    By weak-strong uniqueness we know that $\tilde{v}(t)=u(t)$ in $L^2$ for almost every $t>T_0$. 
    The theorem now follows immediately from Proposition \ref{prop::MildSolutionL2Decay}.
\end{proof}


 \bigskip
\noindent{\bf Acknowledgments.} The research of Z.~Bradshaw and D.~Palmer was supported in part by the  NSF via grant DMS-2307097.
We are grateful to Xukai Yan for discussions about Landau solutions and their generalizations. Z.~Bradshaw also thanks Dallas Albritton for introducing him to half-space analogues of Landau solutions.


\bibliographystyle{siam}

\end{document}